\documentclass[10pt,oneside,reqno]{amsart}
 \pdfoutput=1
\usepackage[a4paper]{geometry}
\usepackage{tikz}
	\usetikzlibrary{trees,shapes.geometric,arrows}
\usepackage[english]{babel}
\usepackage{color}
	\definecolor{darkgreen}{cmyk}{0.5,0.5,1,0.4}
	\definecolor{grayA}{gray}{0.85}
	\definecolor{grayB}{gray}{0.65}

\usepackage{multicol}
\usepackage{verbatim}
\usepackage{float} 
\usepackage{paralist}
\usepackage{hyperref} 
	\hypersetup{colorlinks=true, linkcolor=darkgreen, citecolor=darkgreen} 
\usepackage[hypcap]{caption}
\usepackage{amsfonts, amssymb, amsthm,mathtools}
\usepackage{thmtools}
	\declaretheoremstyle[bodyfont=\normalfont]{noncursive}
	\declaretheorem[name=Theorem,numberwithin=section]{theorem}
	\declaretheorem[sibling=theorem,style=noncursive]{definition}
	\declaretheorem[sibling=theorem]{corollary} 
	\declaretheorem[sibling=theorem]{proposition}
	\declaretheorem[sibling=theorem]{lemma}
	\declaretheorem[sibling=theorem,style=remark]{remark}
	
	\declaretheorem[sibling=theorem,style=definition]{example}

	\numberwithin{equation}{section} 

	\newcommand{\N}{\mathbb N}

	\newcommand{\R}{\mathbb R}
	\newcommand{\C}{\mathbb C}
	\DeclareMathOperator{\UnitSphere}{\mathbb S^1}
	\DeclareMathOperator{\re}{Re}
	\DeclareMathOperator{\im}{Im}
	\DeclareMathOperator{\Aut}{Aut}
	\DeclareMathOperator{\Auto}{Aut_0}
	\DeclareMathOperator{\stab}{stab}
	
	\DeclareMathOperator{\sgn}{sgn}
	\DeclareMathOperator{\Span}{span}

	\DeclareMathOperator{\id}{id}

	\DeclareMathOperator{\imu}{\mathrm i}

	\DeclareMathOperator{\eps}{\varepsilon}
	\DeclareMathOperator{\NTwo}{\mathcal N}

	\DeclareMathOperator{\FTwo}{\mathcal F}

	\DeclareMathOperator{\Isotropies}{\mathcal G}

	\newcommand{\HeisenbergOne}[1]{\mathbb H^{#1}}
	\newcommand{\HeisenbergTwo}[2]{\mathbb H_{#1}^{#2}}
	
	\newcommand{\Sphere}[1]{\mathbb S^{#1}}
	\newcommand{\Hyperquadric}[2]{\mathbb S_{#1}^{#2}}
	\newcommand{\HomHyperquadric}[2]{\hat {\mathbb S}_{#1}^{#2}}
	
	\newcommand{\MapMainTheoremOne}[1]{H_{#1}} 
	\newcommand{\MapMainTheoremTwo}[2]{H_{#1}^{#2}}
	\newcommand{\MapOneParameterFamilyTwo}[2]{G_{#1}^{#2}} 
	\newcommand{\MapOneParameterFamilyThree}[3]{G_{#1,#2}^{#3}}	
	\newcommand{\MapParticularValueTwo}[2]{\mathcal G_{#1}^{#2}} 
	\newcommand{\SimpleMapParticularValueTwo}[2]{\mathcal H_{#1}^{#2}}

	\newcommand{\SimpleMapTranslationParticularValueTwo}[2]{\mathcal H_{#1,p_0}^{#2}}
	\newcommand{\MapTranslationParticularValueTwo}[2]{\mathcal G_{#1,p}^{#2}}
	\newcommand{\MapRenormalizationTranslationParticularValueTwo}[2]{\widetilde{\mathcal G}_{#1,p}^{#2}}
	
	\newcommand{\GeneralUpDown}[3]{#1_{#2}^{#3}}

\begin{document}
\title{Classification of Holomorphic Mappings of \\ Hyperquadrics from $\mathbb C^2$ to $\mathbb C^3$}
\author{ Michael Reiter}
 \thanks{The author was supported by the FWF, projects Y377 and I382, and QNRF, project NPRP 7-511-1-098.}

\pagestyle{plain}

\begin{abstract}
We give a new proof of Faran's and Lebl's results by means of a new CR-geometric approach and classify all holomorphic mappings from the sphere in $\mathbb C^2$ to Levi-nondegenerate hyperquadrics in $\mathbb C^3$. We use the tools developed by Lamel, which allow us to isolate and study the most interesting class of holomorphic mappings. This family of so-called nondegenerate and transversal maps we denote by $\mathcal F$. For $\mathcal F$ we introduce a subclass $\mathcal N$ of maps which are normalized with respect to the group $\mathcal G$ of automorphisms fixing a given point. With the techniques introduced by Baouendi--Ebenfelt--Rothschild and Lamel we classify all maps in $\mathcal N$. This intermediate result is crucial to obtain a complete classification of $\mathcal F$ by considering the transitive part of the automorphism group of the hyperquadrics.
\end{abstract}

\maketitle

\tableofcontents

\section{Introduction and Results}
\label{sec:intro}

Poincar\'{e} \cite{poincare} asked whether for two given real-analytic real hypersurfaces in $\C^2$ one can find holomorphic mappings sending one into the other. He also gave an intuitive answer, originally for biholomorphisms, that for two given arbitrary real-analytic hypersurfaces in general it is unlikely to find holomorphic mappings sending locally one hypersurface into the other. \\
Considerable work was done classifying Levi-nondegenerate hypersurfaces of $\C^N, N\geq 2$ up to biholomorphisms: In $\C^2$, this ``biholomorphic equivalence problem'' was solved by Cartan \cite{cartan1,  cartan2} and for $N\geq2$ by Tanaka \cite{tanaka} and Chern--Moser \cite{chernmoser}.\\
For the class of strictly pseudoconvex hypersurfaces Poincar\'{e}'s question is answered by this classification of Levi-nondegenerate hypersurfaces and results by Pin\v{c}uk \cite{pincuk} and Alexander \cite{alexander1, alexander2}. They proved that any non-constant holomorphic self-mapping of a strictly pseudoconvex hypersurface in $\C^N$ is necessarily an automorphism. This implies that if we consider two biholomorphically equivalent strictly pseudoconvex hypersurfaces, then modulo the action of the automorphisms of the source and target hypersurface there is only one non-constant holomorphic map.

For $N'> N$ and a mapping $H: \C^N \rightarrow \C^{N'}$ we refer to the number $N'-N$ as the \textit{codimension} of $H$. If we consider holomorphic mappings of high codimension the situation changes drastically compared to the equidimensional case. Here models of Levi-nondegenerate hypersurfaces, i.e., hyperquadrics received a lot of attention. For $k\in \N$ and $k \leq N$ we denote the \textit{hyperquadric $\Hyperquadric{k}{N}$ of signature $(k,N-k)$} in $\C^N$ by 
\begin{align*}
\Hyperquadric{k}{N}  & \coloneqq \bigl\{(z_1,\ldots, z_N) \in \C^N:~|z_1|^2 +\ldots +|z_k|^2 -|z_{k+1}|^2 - \ldots - |z_N|^2 = 1\bigr\},
\end{align*}
and write $\Sphere{N} \coloneqq \Hyperquadric{N}{N}$ for the \textit{sphere} in $\C^N$. While studying holomorphic mappings of hyperquadrics it is natural to introduce an equivalence relation for these mappings. We consider the homogeneous model $\HomHyperquadric{k}{N}$ of $\Hyperquadric{k}{N}$ given by
\begin{align*}
\HomHyperquadric{k}{N}  & \coloneqq \bigl\{(z_1,\ldots, z_N,t) \in \C^{N+1}:~|z_1|^2 +\ldots +|z_k|^2 -|z_{k+1}|^2 - \ldots - |z_N|^2 -|t|^2 = 0\bigr\}.
\end{align*}
Let us denote by $SU(N-k,k+1)$ the special unitary group with respect to the Hermitian form  in $\C^{N+1}$ with signature $(N-k,k+1)$ induced by the quadratic form which occurs in the definition of $\HomHyperquadric{k}{N}$. The group of automorphisms of $\HomHyperquadric{k}{N}$ is $SU(N-k,k+1)/K$, where $K$ is the subgroup of $SU(N-k,k+1)$ consisting of diagonal matrices with all entries being equal to $\zeta$ a $(N+1)$-root of unity, see e.g. \cite[\S1]{chernmoser} or \cite[\S2]{BER00}.\\ 
Let $V \subset \C^N$ be an open neighborhood of $p \in \Hyperquadric{k}{N}$. Any holomorphic mapping $H: V \rightarrow \C^{N'}$ which satisfies $H(V \cap \Hyperquadric{k}{N}) \subset \Hyperquadric{k'}{N'}$ can be identified with a CR-mapping $\hat H: \hat V \subset \C^{N+1} \rightarrow \C^{N'+1}$ for some open neighborhood $\hat V$ of $\hat p \in \HomHyperquadric{k}{N}$ satisfying $\hat H(\hat V \cap \HomHyperquadric{k}{N}) \subset \HomHyperquadric{k'}{N'}$. Following \cite[\S2]{faran} and \cite[sections 3.4-3.5]{lebl2} we say that two holomorphic mappings $H_1,H_2$, which both satisfy $H_m: \C^{N} \supset V_m \rightarrow \C^{N'}$, where $V_m$ is a neighborhood of $p_m \in \Hyperquadric{k}{N}$, such that $H_m(V_m \cap \Hyperquadric{k}{N}) \subset \Hyperquadric{k'}{N'}$ for $m=1,2$, are \textit{equivalent} if there exist matrices $U \in SU(N-k,k+1)$ and $U' \in SU(N'-k',k'+1)$ such that $\hat H_2 = U' \circ \hat H_1 \circ U$. 

If $N'\geq 2 N$ D'Angelo \cite{dangelo1} showed that there exist infinitely many quadratic mappings from $\Sphere{N}$ to $\Sphere{N'}$ which are not equivalent. In low codimension the family of holomorphic mappings is less rich. Webster \cite{webster} proved that for holomorphic mappings between the spheres in $\C^N$ and $\C^{N+1}$, where $N\geq 3$, there is only one equivalence class, namely the one generated by the linear embedding. Faran \cite{faran2} extended this result to holomorphic mappings of $\Sphere{N}$ to $\Sphere{N'}$ with $N \geq 3$ and $N' \leq 2N-2$, see also Huang \cite{huang1}. The case of mappings from $\Sphere{N}$ to $\Sphere{2N-1}$ for $N\geq 3$ is treated by Huang--Ji \cite{hj}, where they showed that there exist two equivalence classes of mappings.\\
To study holomorphic mappings between hyperquadrics from $\C^2$ to $\C^3$ we introduce the hypersurface $\Hyperquadric{\eps}{3}$, which for $\eps = \pm 1$ is given by  
\begin{align*}
\Hyperquadric{\pm}{3} & \coloneqq \bigl\{(z_1,z_2,z_3) \in \C^3:~|z_1|^2 + |z_2|^2 \pm |z_3|^2 = 1 \bigr\},
\end{align*}
and we set $\Sphere{3} = \Hyperquadric{+}{3}$. It is well known that $\Sphere{2}$ and $\Hyperquadric{\eps}{3}$ are the only Levi-nondegenerate hyperquadrics in $\C^2$ and $\C^3$, respectively up to biholomorphisms. \\
Faran classified holomorphic mappings between balls in $\C^2$ and $\C^3$ with certain boundary regularity. Below we formulate the main result of Faran in terms of mappings between spheres disregarding regularity issues. 

\begin{theorem}[label=theorem:faran,name=\cite{faran}]
Let $p \in \Sphere{2}$, $U \subset \C^2$ be an open and connected neighborhood of $p$ and $F: U \rightarrow \C^3$ a non-constant holomorphic mapping satisfying $F(U \cap \Sphere{2}) \subset \Sphere{3}$. Then $F$ is equivalent to exactly one of the following maps:
\begin{multicols}{2}
\begin{compactitem}
\item[\rm{(i)}] $F_1(z,w)= (z,w,0)$
\item[\rm{(ii)}] $F_2(z,w) = (z,z w,w^2)$
\item[\rm{(iii)}] $F_3(z,w) = (z^2, \sqrt{2} z w, w^2)$
\item[\rm{(iv)}] $F_4(z,w)= (z^3, \sqrt{3} z w, w^3)$
\end{compactitem}
\end{multicols}
\end{theorem}

Faran's proof consists of giving a characterization of so-called planar maps from $\C^2$ to $\C^3$ which send complex lines to complex planes and uses Cartan's method of moving frames. \\ 
Cima--Suffridge \cite{cs89} approached Faran's Theorem via a reflection principle deduced in \cite{cs83} by the same authors, which contains some inconsistencies when using certain degeneracy conditions. 
Recently Ji \cite{ji} gave a new proof of Faran's Theorem based on Huang's study \cite{huang1} of the Chern--Moser operator and several preceding articles \cite{hj,huang2,hjx, cjx}. In \cite{ji} a small fixable mistake leads to a wrong mapping at the very end of the article. \\ 
More recently Lebl classified mappings sending $\Sphere{2}$ to $\Hyperquadric{-}{3}$, using a classification result for quadratic maps and Faran's approach:

\begin{theorem}[label=theorem:lebl,name=\cite{lebl1}]
Let $p \in \Sphere{2}$, $U \subset \C^2$ be an open and connected neighborhood of $p$ and $L: U \rightarrow \C^3$ a non-constant holomorphic mapping satisfying $L(U \cap \Sphere{2}) \subset \Hyperquadric{-}{3}$. Then $L$ is equivalent to exactly one of the following maps:
\begin{compactitem}
\item[\rm{(i)}] $L_1(z,w) = (z,w,0)$
\item[\rm{(ii)}] $L_2(z,w) = \left(z^2, \sqrt{2} w,w^2\right)$
\item[\rm{(iii)}] $L_3(z,w) = \left(\frac 1 z, \frac {w^2} {z^2}, \frac{w} {z^2}\right)$
\item[\rm{(iv)}] $L_4(z,w) = \frac{\left(z^2 + \sqrt{3} z w + w^2 - z, w^2 + z - \sqrt{3} w -1, z^2-\sqrt{3} z w + w^2 -z\right)}{w^2 + z + \sqrt{3} w -1}$
\item[\rm{(v)}] $L_5(z,w) = \frac{\left(\sqrt[4]{2} (z w - \imu z), w^2-\sqrt{2} \imu w + 1,\sqrt[4]{2}(z w + \imu z)\right)}{w^2+\sqrt{2} \imu w +1}$
\item[\rm{(vi)}] $L_6(z,w) = \frac{\left(2 w^3, z (z^2+3), \sqrt{3} w(z^2 - 1) \right)}{3 z^2+1}$
\item[\rm{(vii)}] $L_7(z,w) = \bigl(1,\ell(z,w),\ell(z,w) \bigr)$, for an arbitrary non-constant holomorphic function $\ell:\C^2 \rightarrow \C$
\end{compactitem}
\end{theorem}

In this article we give a direct proof of both results of Faran and Lebl based on a very different and independent approach. Our main result is the following theorem.

\begin{theorem}[label=theorem:MainTheorem]
Let $p \in \Sphere{2}$, $U \subset \C^2$ be an open and connected neighborhood of $p$ and $H: U \rightarrow \C^3$ a non-constant holomorphic mapping satisfying $H(U \cap \Sphere{2}) \subset \Hyperquadric{\eps}{3}$. Then $H$ is equivalent to exactly one of the following maps:
\begin{compactitem}
\item[\rm{(i)}] $\MapMainTheoremTwo{1}{\eps}(z,w) = (z,w,0)$
\item[\rm{(ii)}] $\MapMainTheoremTwo{2}{\eps}(z,w) = \Bigl(z^2,\frac{ (1- \eps + z (1+ \eps)) w}{\sqrt{2}},w^2\Bigr)$
\item[\rm{(iii)}] $\MapMainTheoremTwo{3}{\eps}(z,w) = \Bigl(z, \frac{(1-\eps+z^2(1+\eps)) w}{2 z}, \frac{ (1- \eps +z(1+\eps )) w^2}{2 z} \Bigr)$
\item[\rm{(iv)}] $\MapMainTheoremTwo{4}{\eps}(z,w) = \frac{\left(4 z^3, ( 3(1- \eps) +(1+3 \eps) w^2) w, \sqrt{3} (1 -\eps + 2 (1 +\eps) w + (1-\eps) w^2 ) z\right)}{1 + 3 \eps + 3 (1-\eps) w^2}$
\end{compactitem}
Additionally for $\eps = -1$ we have:
\begin{compactitem}
\item[\rm{(v)}] $H_5(z,w) = \Bigl(\frac{(2 + \sqrt{2} z)z}{1+\sqrt{2} z + w}, w,  \frac{(1+\sqrt{2} z - w)z}{1+\sqrt{2} z+w}\Bigr)$ 
\item[\rm{(vi)}] $H_6(z,w) =\frac{ \left((1-w) z, 1+ w - w^2,(1+ w) z \right)} {1 - w - w^2}$
\item[\rm{(vii)}] $H_7(z,w) = \bigl(1,h(z,w),h(z,w)\bigr)$ for some non-constant holomorphic function $h: \C^2 \rightarrow \C$
\end{compactitem}
Further, $\MapMainTheoremTwo{3}{-}$ is equivalent to $L_3$, $\MapMainTheoremTwo{4}{-}$ to $L_6$, $\MapMainTheoremOne{5}$ to $L_4$ and $\MapMainTheoremOne{6}$ to $L_5$.
\end{theorem}

We would like to point out that one advantage of our chosen method is, that we prove Faran's and Lebl's result in a unified manner, i.e., we treat mapping from $\Sphere{2} \rightarrow \Sphere{3}$ and $\Sphere{2} \rightarrow \Hyperquadric{-}{3}$ in the same way and use the same techniques for both situations.

Let us provide some details of our proof. First we reformulate the problem and study holomorphic mappings from $\HeisenbergOne{2}$ to $\HeisenbergTwo{\eps}{3}$, which are biholomorphic images of $\Sphere{2}$ and $\Hyperquadric{\eps}{3}$ except one point and are given by
\begin{align*}
\HeisenbergOne{2} =  \bigl\{(z,w) \in \C^2:~\im w = | z|^2 \bigr\}, \qquad \quad \HeisenbergTwo{\eps}{3} = \bigl\{(z_1',z_2', w') \in \C^3:~\im w' = |z_1'|^2 + \eps |z_2'|^2 \bigr\},
\end{align*}
respectively and write $\HeisenbergOne{3} = \HeisenbergTwo{+}{3}$. Further we set $\langle z, w\rangle_{\eps} \coloneqq z_1 w_1 + \eps z_2  w_2$ and $|z|^2_{\eps} \coloneqq \langle z,\bar z\rangle_{\eps}$.
We introduce the class $\FTwo$ consisting of germs of $2$-nondegenerate transversal mappings, defined below in \autoref{def:F2}. They form the most interesting class of mappings to study. In the first part of the proof we consider the action of isotropies, i.e., automorphisms fixing a given point, on $\FTwo$ to provide a normal form $\NTwo$ for $\FTwo$. In a next step we give a classification of the mappings in $\NTwo$. More precisely we have the following theorem.

\begin{theorem}[label=theorem:ReductionOneParameterFamilies]
The set $\NTwo$ consists of explicitly given, rational mappings denoted by $\MapOneParameterFamilyTwo{1}{\eps}(z,w)$, $\MapOneParameterFamilyThree{2}{s}{\eps}(z,w)$ and $\MapOneParameterFamilyThree{3}{s}{\eps}(z,w)$, where $s\geq 0$. The first two maps are of degree $2$, the last one is of at most degree $3$. Each map in $\NTwo$ is not equivalent to any different map of $\NTwo$ with respect to automorphisms fixing $0$.
\end{theorem}

This theorem, given in full details in \autoref{theorem:ReductionOneParameterFamilies2} below, still gives infinitely many mappings under equivalence with respect to isotropies, but reduces the problem to a study of one-parameter families of rational mappings. The second part of the proof consists of studying the action of transitive automorphisms on mappings in $\NTwo$. We choose certain values for $s$ and define the following mappings:
\begin{align}
\label{eq:FinalMaps}
 \MapParticularValueTwo{1}{\eps} (z,w)  & \coloneqq \MapOneParameterFamilyThree{2}{0}{\eps}(z,w), &  \MapParticularValueTwo{2}{\eps}(z,w) & \coloneqq \MapOneParameterFamilyThree{2}{1/2}{\eps}(z,w), &   \MapParticularValueTwo{3}{\eps}(z,w) & \coloneqq \MapOneParameterFamilyThree{2}{1}{\eps}(z,w),\\
\nonumber
& &  \MapParticularValueTwo{4}{\eps}(z,w)&  \coloneqq \MapOneParameterFamilyThree{3}{0}{\eps}(z,w).
\end{align}

Under the equivalence with respect to transitive automorphisms and isotropies we reduce the quotient space of $\FTwo$ under automorphisms to a finite set of classes of mappings. More details on the equivalence relation we use can be found in \autoref{def:transitiveEquiv} below. We obtain the following theorem.

\begin{theorem}[label=theorem:ReductionFinite]
For $m=2,3$ and $1 \leq k \leq 4$ let $\MapOneParameterFamilyThree{m}{s}{\eps}$ be as in \Autoref{theorem:ReductionOneParameterFamilies} and $\MapParticularValueTwo{k}{\eps}$ as in \eqref{eq:FinalMaps}. \\
For $\eps=+1$ we have:
\begin{compactitem}
\item[\rm{(i)}] For every $s \geq 0$ the mapping $\MapOneParameterFamilyThree{2}{s}{+}$ is equivalent to $\MapParticularValueTwo{1}{+}$.
\item[\rm{(ii)}] For every $s \geq 0$ the mapping $\MapOneParameterFamilyThree{3}{s}{+}$ is equivalent to $\MapParticularValueTwo{4}{+}$.
\end{compactitem}
For $\eps=-1$ we have:
\begin{compactitem}
\item[\rm{(iii)}] For every $0\leq s < \frac 1 2$ the mapping $\MapOneParameterFamilyThree{2}{s}{-}$ is equivalent to $\MapParticularValueTwo{1}{-}$.
\item[\rm{(iv)}] For every $s > \frac 1 2$ the mapping $\MapOneParameterFamilyThree{2}{s}{-}$ is equivalent to $\MapParticularValueTwo{3}{-}$.
\item[\rm{(v)}] For every $0\leq s \neq \frac 1 2$ the mapping $\MapOneParameterFamilyThree{3}{s}{-}$ is equivalent to $\MapParticularValueTwo{4}{-}$ and $\MapOneParameterFamilyThree{3}{1/2}{-}= \MapParticularValueTwo{2}{-}$. 
\end{compactitem}
\end{theorem}

From our chosen approach and our careful study of the action of automorphisms on mappings we obtain in some sense computational effectiveness. More precisely our technique allows us to give explicit formulas for the automorphisms which bring an arbitrary mapping to one of the mappings listed in \autoref{theorem:MainTheorem}. Thus we think we provide a new proof of Faran's and Lebl's results which is, in some sense, easier to verify and more elementary. Nevertheless our proof is long, technical and features some nontrivial computations partly carried out with \textit{Mathematica 7.0.1.0} \cite{wolfram}.

The very last part of the proof of \autoref{theorem:MainTheorem} shows that the quotient space of $\FTwo$ under automorphisms indeed consists of the classes of mappings from the previous theorem. More precisely we provide a list of biholomorphic invariants associated to each mapping of \autoref{theorem:ReductionFinite} to show that the maps listed in \autoref{theorem:MainTheorem} are not equivalent to each other.

We organize this work as follows:
In \Autoref{sec:Preliminaries} we give most of the relevant definitions and introduce all biholomorphic invariants we use in order to obtain a class $\FTwo$. For this class of mappings, we compute a normal form in \Autoref{sec:NormalForm} and obtain $\NTwo \subset \FTwo$, the set of normalized mappings with respect to the stability groups. For $\NTwo$ we compute a jet parametrization in \Autoref{sec:MappingsInN} and after a so-called desingularization it turns out that $\NTwo$ consists of one separated mapping and two one-parameter families of mappings. In \Autoref{sec:global} we use the transitive automorphisms to show that $\FTwo$ consists of finitely many orbits of maps. 
Finally we complete the proof of \autoref{theorem:MainTheorem} in \Autoref{sec:ClassificationFinale}. 
 This article is partly based on the author's thesis \cite{reiter} at the University of Vienna.

\section{Preliminaries}
\label{sec:Preliminaries}

\begin{definition}
We fix coordinates $(z,w)=(z_1,\ldots, z_n,w) \in \C^{n+1}$.
\begin{compactitem} 
\item[\rm{(i)}] Let $h: \C^{n+1} \rightarrow \C$ be a holomorphic function given by $h(z,w) = \sum_{\alpha, \beta} a_{\alpha \beta} z^{\alpha} w^{\beta}$ defined near $0$. We write $\bar{h}(\bar z,\bar w) \coloneqq \overline{h(z,w)}=\sum_{\alpha, \beta} \bar{a}_{\alpha \beta} \bar z^{\alpha} \bar w^{\beta}$  for the complex conjugate of $h$.
For derivatives of $h$ with respect to $z$ or $w$ we write $h_{z^{\alpha} w^{\beta}}(0)\coloneqq  \frac{\partial^{|\alpha|+ |\beta|} h}{z^{\alpha} w^{\beta}}(0)$. For $n\geq 1$ and a  mapping $H: \C^{n+1} \rightarrow \C^{n'+1}$ defined near $0$ with components $H= \bigl(f_1,\ldots,f_{n'},g \bigr)$ we write $H_{z^{\alpha} w^{\beta}}(0) = \bigl (f_{1 z^{\alpha} w^{\beta}}(0),\ldots, f_{n' z^{\alpha} w^{\beta}}(0),g_{z^{\alpha} w^{\beta}}(0) \bigr)$.
\item[\rm{(ii)}] For $H=(f_1,\ldots,f_{n'},g)$ a holomorphic mapping of $\C^{n+1}$ to $\C^{n'+1}$ near $0$ we denote
\begin{align*}
\Delta(\alpha_1,\beta_1; \ldots;\alpha_{n'},\beta_{n'}) \coloneqq 
\left| \begin{array}{c c c}
f_{1 z^{\alpha_1} w^{\beta_1}} (0) &\cdots & f_{1 z^{\alpha_{n'}} w^{\beta_{n'}}} (0)  \\
\vdots & & \vdots \\
f_{n' z^{\alpha_1} w^{\beta_1}} (0) &  \cdots & f_{n' z^{\alpha_{n'}} w^{\beta_{n'}}} (0)
\end{array} \right|.
\end{align*}
\item[\rm{(iii)}] Let $H:\C^{n+1} \rightarrow \C^{n'+1}$ be a mapping defined at $p \in \C^{n+1}$ and $\alpha \in \N^{n+1}$. We denote by $j^k_pH$ the \textit{$k$-jet of $H$ at $p$} defined as
\begin{align*}
j_p^k H \coloneqq \left( \frac{\partial^{|\alpha|} H} {\partial Z^{\alpha} } (p): |\alpha| \leq k\right). 
\end{align*}
We denote by $J^k_p$ the collection of all $k$-jets at $p$. We write $J_p^k(M,p;M',p')$ for the collection of all $k$-jets at $p$ of mappings, which send $(M,p) \subset (\C^N,p)$ to $(M',p')\subset (\C^{N'},p')$.
\item[\rm{(iv)}]  For a rational, holomorphic mapping $H: \C^N \rightarrow \C^{N'}$ given by $H= (P_1,\ldots,P_{N'})/Q$, where $P_1,\ldots,P_{N'}$ and $Q$ are polynomial and complex-valued we say $H$ is \textit{reduced} if $P_1,\ldots,P_{N'}$ and $Q$ do not possess any common factor. Then the degree $\deg H$ of a reduced rational map $H$ is defined as $\deg H \coloneqq \max \bigl( (\deg P_k)_{k=1,\ldots, N'},\deg Q \bigr)$.
\end{compactitem} 
\end{definition}

\subsection{Automorphisms and Isotropic Equivalence}
\label{subsec:Automorphisms}

\begin{definition}
We denote the collection of locally real-analytic CR-diffeomorphisms of $(M,0)$ by $\Aut(M,0) \coloneqq \{H: (\C^{n+1},0) \rightarrow \C^{n+1}: H \text{ holomorphic}, H(M) \subset M, \det(H'(0)) \neq 0 \}$ and the group of \textit{isotropies} or \textit{stability group} of $(M,0)$ by $\Auto(M,0) \coloneqq \{H \in \Aut(M,0): H(0)=0\}$.
\end{definition}

\begin{definition}[label=def:Autom]
\begin{compactitem} 
\item[\rm{(i)}] We write $\R^+\coloneqq\{x \in \R:~x>0\}$, denote the unit sphere in $\C$ by $\UnitSphere \coloneqq \{e^{\imu t}: 0 \leq t < 2 \pi \}$ and set $\Gamma \coloneqq \R^+ \times \R \times \UnitSphere \times \C$. Then we parametrize $\Auto(\HeisenbergOne{2},0)$ via $\Gamma$ and write for $\gamma=(\lambda, r,u,c) \in \Gamma$:
\begin{align}
\label{eq:sigma}
\sigma_{\gamma}(z,w) \coloneqq \frac{(\lambda u (z + c w), \lambda^2 w)}{1 - 2 \imu \bar c z + (r - \imu |c|^2) w}.
\end{align}
\item[\rm{(ii)}] For $p=(p_1,p_2) \in \HeisenbergOne{2}$ we introduce the following mappings which form the so-called \textit{translations of $\HeisenbergOne{2}$}:
\begin{align}
\label{eq:TranslationC2}
 t_p: \HeisenbergOne{2} \rightarrow \HeisenbergOne{2}, \quad t_p(z,w) & \coloneqq (z +p_1, w + p_2 + 2 \imu  \bar p_1 z).
\end{align}
\item[\rm{(iii)}]
We define $\mathcal S^2_{\eps,\sigma} \coloneqq \big\{a' =(a_1', a_2') \in \C^2:  |a_1'|^2 + \eps |a_2'|^2 = \sigma \big\}$ where $\sigma = \pm 1$ if $\eps = -1$ and $\sigma = +1$ if $\eps=+1$ and let
\begin{align}
\label{eq:UTwoOneOne}
U' \coloneqq \left(\begin{array}{cc}
u' a_1' & - \eps u' a_2'  \\
\bar a_2' & \bar a_1' \end{array}\right),  \qquad u' \in \UnitSphere, \quad a'=(a_1',a_2') \in \mathcal S^2_{\eps,\sigma}.
\end{align}
We set $\Gamma' \coloneqq \R^+ \times \R \times \UnitSphere \times \mathcal S^2_{\eps,\sigma} \times \C^2$ to parametrize $\Auto(\HeisenbergTwo{\eps}{3},0)$ via $\Gamma'$ and write for $\gamma' =(\lambda',r',u',a',c') \in \Gamma'$:
\begin{align}
\label{eq:tau}
\sigma'_{\gamma'}(z',w') \coloneqq \frac{(\lambda' U'~{^t(z' +c' w')}, \sigma {\lambda'}^2 w')}{1 - 2 \imu\langle \bar c',z'\rangle_{\eps} + \bigl(r' - \imu |c'|^2_{\eps} \bigr) w'}.
\end{align}
\item[\rm{(iv)}] For $p'=(p_1',p_2',p_3') \in \HeisenbergTwo{\eps}{3}$ we define the following mapping, which we call a \textit{translation of $\HeisenbergTwo{\eps}{3}$}:
\begin{align}
 \label{eq:TranslationC3}
 t'_{p'}:  \HeisenbergTwo{\eps}{3} \rightarrow \HeisenbergTwo{\eps}{3}, \quad t'_{p'}(z',w') & \coloneqq  \bigl(z_1' - p_1',z_2'-p_2', w' - {\bar p_3}' - 2 \imu (\bar p_1' z_1' + \eps \bar p_2' z_2') \bigr).
\end{align}
\item[\rm{(v)}]  We write $\Isotropies \coloneqq \Auto(\HeisenbergOne{2},0) \times \Auto(\HeisenbergTwo{\eps}{3},0) $ for the direct product of the isotropy groups of $\HeisenbergOne{2}$ and $\HeisenbergTwo{\eps}{3}$.
\end{compactitem}
\end{definition}

\begin{remark}
If we set $\eps = -1$ and take $a_1'=0$ and $a_2' = u'=1$ in \eqref{eq:UTwoOneOne} we obtain the following automorphism $\pi'$ of $\HeisenbergTwo{-}{3}$:
\begin{align}
\label{eq:Pi}
\pi'(z_1',z_2',w') \coloneqq (z_2',z_1',-w').
\end{align}
If we do not mention otherwise we take $\sigma = +1$ in the definition of $\sigma'_{\gamma'}$ and use $\pi'$ separately.
\end{remark}

\begin{remark}[label=rem:FormOfAutomorphism]
Let us write $M$ for either $\HeisenbergOne{2}$ or $\HeisenbergTwo{\eps}{3}$. We note that since the automorphisms given in \eqref{eq:sigma}--\eqref{eq:TranslationC3} generate $\Aut(M,0)$, we immediately obtain that if we let $\phi \in \Aut(M,0)$, then there exists a unique translation $t$ and isotropy $\sigma$ of $(M,0)$ such that $\phi = t \circ \sigma$.
\end{remark}

\begin{definition}[label=definition:localEquivalence]
Let $G,H: (\HeisenbergOne{2},0) \rightarrow (\HeisenbergTwo{\eps}{3},0)$ be germs of holomorphic mappings. We let $(\gamma,\gamma') \in \Gamma \times \Gamma'$ to define $H_{\gamma,\gamma'}(z,w) \coloneqq \bigl(\sigma'_{\gamma'} \circ H \circ \sigma_{\gamma} \bigr)(z,w)$ and $O_0(H) \coloneqq \bigl\{H_{\gamma,\gamma'} : (\gamma,\gamma') \in \Gamma \times \Gamma'\bigr\}$, which we call the \textit{isotropic orbit} of $H$. We say $G$ is \textit{isotropically equivalent} to $H$ if $G \in O_0(H)$. We will refer to the elements of $\Gamma \times \Gamma'$ as \textit{standard parameters}. In the case where we take standard parameters $(\gamma,\gamma') \in \Gamma\times \Gamma'$ such that $(\sigma_{\gamma},\sigma'_{\gamma'}) = (\id_{\C^2}, \id_{\C^3})$, we say the standard parameters are \textit{trivial}.
\end{definition}

\subsection{Setup}
\label{sec:setup}

\begin{definition}
We define the following biholomorphism $T_N: \C^N \setminus \{z_1 =-1\} \rightarrow \C^N\setminus\{ z_N = -\imu\}$:
\begin{align} 
\label{eq:Cayley}
 T_N(z_1,\ldots,z_N) \coloneqq \Bigl(z_2, \ldots, z_N,\imu (1-z_1)\Bigr) / (1+z_1), 
\end{align}
with inverse
\begin{align}
\label{eq:CayleyInv}
T_N^{-1}(w_1,\ldots, w_N) = \Bigl(1+ \imu w_N, 2 w_1, \ldots, 2 w_{N-1}\Bigr)/(1- \imu w_N).
\end{align}
\end{definition}

Let $q \in \Hyperquadric{k}{N}$ and decompose $\C^{N} =  \C q \oplus q^{\perp}$, where $q^{\perp} \coloneqq \{v\in \C^N:  \langle q,\bar v \rangle_k = 0\}$ and $\langle . , . \rangle_k$ denotes the real bilinear form in $\C^N$ induced by the quadratic form defining $ \Hyperquadric{k}{N}$. Let $H$ be a mapping as in the assumption of \autoref{theorem:MainTheorem} with $H(p)=p' \in \Hyperquadric{\eps}{3}$. Then we decompose $\C^2$ and $\C^3$ with respect to $p$ and $p'$ as described above, such that $p=(1,0)$ and $H(1,0) = (1,0,0) \in \Hyperquadric{\eps}{3}$. We consider $\hat H\coloneqq T_{3} \circ H \circ T_{2}^{-1}$, where we possibly need to shrink $U$ to avoid the poles of $T_2^{-1}$ and $T_3$ respectively. Moreover in these coordinates $\hat H$ satisfies $\hat H(0)=0$ and $\hat H(U \cap \HeisenbergOne{2}) \subset \HeisenbergTwo{\eps}{3}$. Thus $\hat H = (f_1,f_2,g)$ has to satisfy the  \textit{mapping equation}:
\begin{align}
\label{eq:firstMapEq}
\im\bigl(g(z,w) \bigr) = |f_1(z,w)|^2 + \eps |f_2(z,w)|^2,
\end{align}
if $\im w = |z|^2$ for $(z,w) \in U$. In order to work with such an equation in a more convenient way we complexify \eqref{eq:firstMapEq} by setting $\chi \coloneqq \bar z$ and $\tau \coloneqq \bar w$ to obtain the \textit{complexified mapping equation}:
\begin{align}
\label{eq:MapEq}
g(z,\tau + 2 \imu z \chi) - \bar g(\chi,\tau)= 2\imu \Bigl(f_1(z,\tau + 2 \imu z \chi) \bar{f}_1(\chi,\tau) + \eps f_2(z,\tau + 2 \imu z \chi)  \bar{f}_2(\chi,\tau)\Bigr),
\end{align}
which holds for all $(z,\chi, \tau) \in \C^3$ sufficiently close to $0$. If we evaluate \eqref{eq:MapEq} at $(z,\chi, \tau)=(z,0,0)$ we obtain $g(z,0) = 0$. Moreover differentiating \eqref{eq:MapEq} with respect to $z$ and $\chi$ and evaluating the result at $0$ we have
\begin{align}
\label{eq:TransversalImmersive}
g_w(0) = | f_{1 z}(0)|^2 + \eps | f_{2 z}(0)|^2,
\end{align}
which implies $g_{w}(0) \in \R$.

\subsection{Biholomorphic Invariants of Mappings}

\subsubsection{Transversality of Mappings}
This section is devoted to introduce a well-known first-order biholomorphic invariant for mappings.

\begin{definition}[label=def:transversality]
Let $M\subset \C^N$ and $M' \subset \C^{N'}$ be real-analytic real hypersurfaces and $U\subset \C^N$ be a neighborhood of $p\in M$. A holomorphic mapping $H: \C^N \rightarrow \C^{N'}$ with $H(U \cap M) \subset M'$ is called \textit{transversal to $M'$ at $H(p)$} if
\begin{align}
\label{eq:transversality}
T_{H(p)} M' + d H (T_p\C^N)= T_{H(p)} \C^{N'}.
\end{align}
\end{definition}

\begin{proposition}[label=prop:transversality]
Let $U \subset \C^2$ be an open, connected neighborhood of $0$ and $H: U \rightarrow \C^3$ a non-constant holomorphic mapping with components $H=(f_1,f_2,g)$ satisfying $H(0)=0$ and $H(U \cap \HeisenbergOne{2}) \subset \HeisenbergTwo{\eps}{3}$. Then we have the following two mutually exclusive statements:
\begin{compactitem}
\item[\rm{(i)}] $H$ is transversal to $\HeisenbergTwo{\eps}{3}$ outside a proper, real-analytic subset $X$ of $U \cap \HeisenbergOne{2}$. If $H$ is transversal to $\HeisenbergTwo{\eps}{3}$ at $0$ we can assume $g_w(0) \neq 0$.
\item[\rm{(ii)}] The mapping satisfies $H(U) \subset \HeisenbergTwo{\eps}{3}$.
\end{compactitem}
Furthermore (ii) can only happen if $\eps = -1$.
\end{proposition}

\begin{proof}
The statements in (i) and (ii) are proved in \cite[Theorem 1.1]{BER07} in more generality and the second statement in (i) is proved in \cite[Theorem 5.2]{ER}. Next we assume that (ii) holds for $\eps = +1$, such that
\begin{align*}
g(z,w) - \bar g(\chi,\tau) - 2 \imu \bigl( f_1(z,w) \bar f_1(\chi, \tau) + f_2(z,w) \bar f_2(\chi, \tau) \bigr)= 0,
\end{align*}
for all $(z,w,\chi,\tau) \in \C^4$ near $0$. Setting $\chi=\tau=0$ we obtain $g(z,w) = 0$, which immediately implies $(f_1,f_2) \equiv 0$ if we set $\chi=\bar z$ and $\tau = \bar w$, thus $H$ is constant, which we excluded.
\end{proof}

\begin{remark}[label=rem:TransversalityInvariance]
In view of \eqref{eq:transversality} it is easy to observe that transversality is invariant under biholomorphic changes of coordinates.
\end{remark}
\begin{remark}[label=rem:NonTransversality]
One can show that $H$ is transversal to $M'$ at $H(p)$ if and only if there exists a holomorphic function $A: (\C^{2 N+2},p) \rightarrow \C$ such that $\rho' \bigl(H(Z), \bar H(\zeta) \bigr) = A(Z,\zeta) \rho(Z,\zeta)$, for $\rho,\rho'$ defining functions for $M$ and $M'$ respectively and $A(p,\bar p) \neq 0$. $(Z,\zeta)\in \C^{2 N}$ denote coordinates for the complexification of $M$. The set $\{q \in M: A(q,\bar q) = 0 \}$ defines a proper, real-analytic subset of $M$ and hence we say \textit{$H$ is transversal to $M'$ outside a proper, real-analytic subset of $M$} if $H$ is transversal to $M'$ at $H(p)$ for some $p \in M$. Otherwise if $A(p,\bar p) = 0$ for $p \in U \cap M$ we have that $H$ is not transversal at $H(p)$.
\end{remark}
\begin{remark}[label=rem:Immersive]
\autoref{prop:transversality} (i) together with \eqref{eq:TransversalImmersive} shows that a transversal mapping $H$ from $\HeisenbergOne{2}$ to $\HeisenbergTwo{\eps}{3}$ is immersive.
\end{remark}

\subsubsection{Degeneracy of Mappings}
The next biholomorphic invariant we need is the well-known (finite) degeneracy for mappings. This invariant was used by among others Faran \cite{faran}, Cima--Suffridge \cite{cs83} and Forstneri\v{c} \cite{forstneric} to extend proper holomorphic mappings, which are smooth up to the boundary of their domain, holomorphically past the boundary. This section is based on \cite[Section 2.5]{lamel}.

\begin{definition}[label=def:degeneracy]
Let $M \subset \C^N$ and $M' \subset \C^{N'}$ be generic, real-analytic submanifolds of codimension $d$ and $d'$ respectively and denote $n \coloneqq N-d$ and $n' \coloneqq N'-d'$. For $p \in M, p' \in M'$ and $U \subset \C^N$ a neighborhood of $p$ we let $H: U \rightarrow \C^{N'}$ be a holomorphic mapping satisfying $H(U \cap M) \subset M'$. We choose coordinates $Z$ and $Z'$ centered at $p$ and $p'$ for $M$ and $M'$ respectively. In the complexification of $M$ and $M'$ we write $\zeta\coloneqq \bar Z$ and $\zeta' \coloneqq \bar Z'$. For $\rho'=(\rho'_1,\ldots, \rho'_{d'})$ a defining function for $M'$ near $p'$ we denote for $1 \leq j \leq d'$ the complex gradient $\rho'_{j,Z'}(Z',\bar Z')$ of $\rho'_j$ with respect to $Z'$ by defining $\rho'_{j,Z'}(Z',\zeta') \coloneqq \left( \frac{\partial \rho_j'(Z',\zeta')}{\partial z_1'}, \ldots, \frac{\partial \rho'(Z',\zeta')}{\partial z'_{N'}} \right)$. For $L_1,\ldots,L_{n}$ a basis of CR-vector fields for $M$ near $p$ and $\alpha =(\alpha_1, \ldots, \alpha_n)\in \N^n$ we denote $L^{\alpha} \coloneqq L_1^{\alpha_1} \cdots L_n^{\alpha_n}$. Then we define for $k\geq 0$ and $q\in M$ near $p$ the following vector spaces after possibly shrinking $U$:
\begin{align*}
E'_k(q) \coloneqq \Span_{\C} \left\{ L^{\alpha} \rho'_{j,Z'}\bigl(H(Z), \bar H(\zeta)\bigr) \Big\vert_{(Z,\zeta)=(q,\bar q)}:  0 \leq |\alpha| \leq k, 1\leq j\leq d' \bigr)\right\} \subset \C^{N'}.
\end{align*}
Since for $k \geq 0$ the $E'_k(q)$ form an ascending chain of vector spaces in $\C^{N'}$, there exists a minimal $k_0\geq 0$ such that $E'_k(p) = E'_{k_0}(p)$ for all $k \geq k_0$ and $E'_{k_0-1}(p) \subsetneq E'_{k_0}(p)$ for $p \in M$ in a neighborhood of $q \in M$. The number $s(q) \coloneqq N' - \dim_{\C} E'_{k_0}(q)$ is called the \textit{degeneracy} of $H$ at $q$ and $H$ is called $\bigl(k_0,s(q) \bigr)$\textit{-degenerate} at $q \in M$. If $s = s(p)$ is constant for $p \in M$ in a neighborhood of $q\in M$ we say $H$ is \textit{constantly $(k_0,s)$-degenerate} near $q\in M$ and $s$ is called \textit{constant degeneracy} of $H$. If for some $q \in M$ we have $s(q)=0$, then $E'_{k_0}(q) = \C^{N'}$ which means that $H$ is of constant degeneracy $s = 0$ near $q$ and $H$ is called $k_0$\textit{-nondegenerate}.
\end{definition}

\begin{remark}[label=rem:DegeneracyInvariance]
In \cite[Section 2.3]{lamel} it is shown that \autoref{def:degeneracy} is independent of the choices of a basis of CR-vector fields, the defining function and holomorphic coordinates in $\C^N$ and $\C^{N'}$.
\end{remark}

\begin{definition}[label=def:GenericDegeneneracy]
Let $M \subset \C^N$ and $M' \subset \C^{N'}$ be generic, real-analytic submanifolds and $U \subset \C^N$ be a neighborhood of $p\in M$. Let $H: U \rightarrow \C^{N'}$ be a holomorphic mapping satisfying $H(U \cap M) \subset M'$ and fix $V \subset U$ a neighborhood of $p \in M$ such that $\overline{V \cap M} \subset U$. The number $s_H(V) \coloneqq \min \{s(q): q \in \overline{V \cap M} \}$ is called \textit{generic degeneracy for $H$} in $V \subset \C^N$ a neighborhood of $p\in M$. 
\end{definition}
By \cite[Lemma 22]{lamel} it follows that $H$ is constantly $\bigl(k_0,s_H(V)\bigr)$-degenerate outside a proper, real-analytic subset of $V \cap M \subset U$ for some $k_0 \in \N$, hence if we take a smaller neighborhood $W \subset V$ in \autoref{def:GenericDegeneneracy} then $s_H(V) = s_H(W)$. We skip the argument in $s_H(V)$ and write $s_H$ from now on. Next we obtain bounds for the generic degeneracy $s_H$ and $k_0$ adapted to our setting.

\begin{proposition}[label=prop:BoundsDegeneracy]
Let $U \subset \C^2$ be a neighborhood of $p \in \HeisenbergOne{2}$ and $H: U \rightarrow \C^3$ a holomorphic mapping with components $H=(f_1,f_2,g)$ and $H(U \cap \HeisenbergOne{2}) \subset \HeisenbergTwo{\eps}{3}$ which is transversal to $\HeisenbergTwo{\eps}{3}$ outside a proper real-analytic subset of $\HeisenbergOne{2}$. 
There exists a proper, real-analytic subset $X$ of $U \cap \HeisenbergOne{2}$ such that after shrinking $U$ and performing a change of coordinates in $U\setminus X$ the following two mutually exclusive statements hold: 
\begin{compactitem} 
\item[\rm{(i)}] $H$ is $2$-nondegenerate, such that $f_{1z}(0) f_{2 z^2}(0) - f_{2 z}(0) f_{1 z^2}(0) \neq 0$.
\item[\rm{(ii)}] $H$ is constantly $(1,1)$-degenerate.
\end{compactitem}
\end{proposition}

\begin{proof}
By \cite[Lemma 22]{lamel} we have $(k_0,s_H)$-degeneracy outside a proper, real-analytic subset of $\HeisenbergOne{2}$. By \autoref{rem:TransversalityInvariance} and \autoref{rem:DegeneracyInvariance} after a change of coordinates we assume that $0$ is a point where $H$ is constantly $(k_0,s_H)$-degenerate and transversal to $\HeisenbergTwo{\eps}{3}$. This change of coordinates is performed via composing $H$ with translations such that $0$ gets mapped to a point $q$ where $H$ is constantly $(k_0,s_H)$-degenerate and transversal to $\HeisenbergTwo{\eps}{3}$, i.e., we consider the mapping $t'_{H(q)} \circ H \circ t_q$ from \eqref{eq:TranslationC2} and \eqref{eq:TranslationC3} instead of $H$. At this point it is possible that we need to shrink $U$. Then we apply \cite[Lemma 23--24]{lamel} to obtain the desired result.
\end{proof}

\begin{remark}[label=remark:NotTwoNondegeneratePoints]
We let $H=(f_1,f_2,g)$ be as in \autoref{prop:BoundsDegeneracy}. According to \autoref{def:degeneracy} and formula $(25)$ in \cite[Section 2.3]{lamel} we note that the set $N$ of points in $\HeisenbergOne{2}$, where $H$ is not $2$-nondegenerate, is given by $N = \left\{p \in \HeisenbergOne{2}: L f_1(p) L^2 f_2(p) - L f_2(p) L^2 f_1(p) = 0\right\}$, where $L$ is a basis of CR-vector fields for $\HeisenbergOne{2}$. 
\end{remark}

\subsection{Initial Classification and the Class \texorpdfstring{$\FTwo$}{F}}

\begin{definition}[label=def:F2]
For a neighborhood $U \subset \C^2$ of $0$ let us denote the set $\FTwo(U)$ of holomorphic mappings $H=(f_1,f_2,g)$ with $H(U \cap \HeisenbergOne{2}) \subset \HeisenbergTwo{\eps}{3}$, which satisfy $H(0) =0$,
\begin{align}
\label{eq:F2Cond}
\Delta(1,0;2,0) = f_{1 z} (0) f_{2 z^2} (0) -f_{2 z} (0)f_{1 z^2} (0) \neq 0 \qquad \text{and} \qquad  g_w(0) >0.
\end{align}
We denote by $\FTwo$ the set of germs $H$, such that $H \in \FTwo(U)$ for some neighborhood $U \subset \C^2$ of $0$.
\end{definition}

\begin{proposition}[label=proposition:FirstProperties]
Let $U \subset \C^2$ be an open and connected neighborhood of $0$ and $H: U \rightarrow \C^3$ a non-constant holomorphic mapping given by $H=(f_1,f_2,g)$ with $H(U \cap \HeisenbergOne{2}) \subset \HeisenbergTwo{\eps}{3}$ and $H(0) = 0$. Then, after  possibly shrinking $U$, changing coordinates or composing $H$ with automorphisms, one of the following mutually exclusive statements holds:
\begin{compactitem} 
\item[\rm{(i)}] $H$ is transversal to $\HeisenbergTwo{\eps}{3}$ and $2$-nondegenerate at $0$ and we can assume that $H \in \FTwo$.
\item[\rm{(ii)}] $H$ is equal to the linear embedding $(z,w) \mapsto (z,0,w)$.
\item[\rm{(iii)}] For $\eps = -1$: $H$ is a mapping of the form $(z,w) \mapsto (h(z,w),h(z,w),0)$ for some non-constant holomorphic function  $h: U \rightarrow \C$ with $h(0) = 0$.
\end{compactitem}
\end{proposition}

\begin{proof}
We apply \autoref{prop:transversality} to obtain that either $H$ is transversal to $\HeisenbergTwo{\eps}{3}$ outside a proper, real-analytic set of $U \cap \HeisenbergOne{2}$ or for $\eps = -1$ we have $H$ maps a neighborhood $U\subset \C^2$ of $0$ to $\HeisenbergTwo{-}{3}$.\\
We assume the first condition for $H$ and apply \autoref{prop:BoundsDegeneracy} to obtain that after possibly composing $H$ with translations that $H$ is transversal to $\HeisenbergTwo{\eps}{3}$ at $0$ and either $2$-nondegenerate or constantly $(1,1)$-degenerate near $0$. By \autoref{prop:transversality} (i) we can assume that $g_w(0) \neq 0$. For $\eps = +1$ by \eqref{eq:TransversalImmersive} we immediately have $g_w(0) > 0$. If $\eps = -1$ and  we have $g_w(0) < 0$ we compose $H$ with the automorphism $\pi'$ from \eqref{eq:Pi}.\\
If we assume $H$ is transversal to $\HeisenbergTwo{\eps}{3}$ at $0$ and $2$-nondegenerate near $0$, we immediately obtain (i) by \autoref{prop:BoundsDegeneracy} (i).\\ 
If we assume $H$ is transversal to $\HeisenbergTwo{\eps}{3}$ at $0$ and $(1,1)$-degenerate near $0$ we either refer to \cite[Chapter 7]{reiter} or we apply \cite[Theorem 1.1]{ES}, which implies that the image of $H$ is contained in a $2$-dimensional complex hyperplane. From \autoref{rem:Immersive} we know that $H$ is immersive, hence after a change of coordinates we may assume that $H=(f,0,g)$ where $(z,w)\rightarrow \bigl(f(z,w),g(z,w)\bigr)$ is a biholomorphism from $(\C^2,0)$ to $(\C^2,0)$. Since $H$ maps $\HeisenbergOne{2}$ to $\HeisenbergTwo{\eps}{3}$ and fixes $0$ we conclude that $H$ is an isotropy, hence (ii) follows.\\
To finish the proof we need to treat the case if $\eps = -1$ and $H$ maps a neighborhood $U\subset \C^2$ to $\HeisenbergTwo{-}{3}$. Here the following mapping equation holds for all $(z,w,\chi,\tau) \in W$ for some neighborhood $W \subset \C^4$ of $0$: 
\begin{align*}
g(z,w) - \bar g(\chi,\tau) - 2 \imu \bigl( f_1(z,w) \bar f_1(\chi, \tau) - f_2(z,w) \bar f_2(\chi, \tau) \bigr)= 0.
\end{align*}
Setting $\chi=\tau=0$ we obtain $g(z,w) = 0$ such that the above equation reduces to $|f_1(z,w)|^2 = |f_2(z,w)|^2$. Next we apply \cite[Chapter 3, Proposition 3]{dangelo2} and an automorphism of $\HeisenbergTwo{-}{3}$ of the form $(z_1',z_2',w') \mapsto (z_1', u z_2', w')$ with $|u|=1$ to $(z,w) \mapsto (f(z,w),0)$, such that the image of $H$ is contained in the complex variety given by $\bigl\{(z_1',z_2',w') \in \C^3: z'_1 = z'_2, w' =0\bigr\}$. Thus $H$ is equivalent to the map $(z,w) \mapsto (h(z,w), h(z,w),0)$ for some non-constant holomorphic function $h: \C^2 \rightarrow \C$ with $h(0) =0$.
\end{proof}

\section{Normal Form \texorpdfstring{$\NTwo$}{N} for Mappings in \texorpdfstring{$\FTwo$}{F}}
\label{sec:NormalForm}

Note that the conditions \eqref{eq:F2Cond} which define the class $\FTwo$ are preserved if we apply isotropies fixing $0$ to a map in $\FTwo$.

\begin{proposition}[label=proposition:NormalForm2Nondeg]
Let $H \in \FTwo$. Then there exist isotropies $(\sigma ,\sigma') \in \Isotropies$ such that $\widehat H \coloneqq \sigma' \circ H \circ \sigma$ satisfies $\widehat H (0) =0$ and the following conditions:
\begin{multicols}{3}
\begin{compactitem}
\item[\rm{(i)}] $\widehat H_z(0) =(1,0,0)$
\item[\rm{(ii)}] $\widehat H_w(0) = (0,0,1)$
\item[\rm{(iii)}] $\widehat f_{2 z^2}(0) = 2$
\item[\rm{(iv)}] $\widehat f_{2 z w}(0) = 0$
\item[\rm{(v)}] $\widehat f_{1 w^2}(0) = |\widehat f_{1 w^2}(0)| \geq 0$
\item[\rm{(vi)}] $\re\bigl(\widehat g_{w^2}(0)\bigr) = 0$
\item[\rm{(vii)}] $\re\bigl(\widehat f_{2 z^2 w}(0) \bigr) = 0$
\end{compactitem}
\end{multicols}
\end{proposition}

\begin{definition}[label=def:N2]
We refer to the equations and inequalities given in \Autoref{proposition:NormalForm2Nondeg} as \textit{normalization conditions}. A holomorphic mapping of $\FTwo$ satisfying the normalization conditions is called a \textit{normalized mapping}. The set of normalized mappings is denoted by $\NTwo$.
\end{definition}

\begin{proof}[Proof of \autoref{proposition:NormalForm2Nondeg}]
We consider $\widehat H \coloneqq \sigma' \circ H \circ \sigma$, where we use all standard parameters in $\sigma$ and $\sigma'$ with the notation of \eqref{eq:sigma} and \eqref{eq:tau}. Then we compute the coefficients of $\widehat H$ we want to normalize and solve the resulting equations for the standard parameters. The first equations are the following, where we take the $2\times 2$-matrix $U'$ as in \eqref{eq:UTwoOneOne}:
\begin{align*}
\widehat H_{z}(0) = &  \left( u \lambda \lambda' U' {^t ( f_{1z}(0), f_{2 z}(0))} , 0\right) = \bigl(1,0,0\bigr),
\\
\widehat H_w(0)  = & \left( u \lambda \lambda' U'
 \left(\begin{array}{c}
c_1' \lambda g_w(0) + \lambda f_{1 w}(0) + c u f_{1 z}(0) \\
c_2' \lambda g_w(0) + \lambda f_{2 w}(0) + c u f_{2 z}(0) \end{array}\right), \lambda^2 {\lambda'}^2 g_w(0)\right) = ~(0,0,1),
\end{align*}
which can be solved using \eqref{eq:TransversalImmersive} by
\begin{align*}
a_1' = \frac{ \bar f_{1 \chi}(0)}{u u' |f_z(0)|_{\eps}}, \qquad
a_2' = -\frac{\bar f_{2 \chi}(0)}{u u' |f_z(0)|_{\eps}},
\end{align*}
such that $a' = (a_1',a_2') \in \mathcal S^2_{\eps,\sigma}$ and by
\begin{align*}
c_1'  = \frac{-c u f_{1 z}(0) - \lambda f_{1 w}(0)}{\lambda  g_w(0)}, \qquad c_2' = \frac{-c u f_{2 z}(0) - \lambda f_{2 w}(0)}{\lambda  g_w(0)},  \qquad \lambda '  =\frac{1}{\lambda \sqrt{g_w(0)}},
\end{align*}
since we require $\lambda,g_w(0)>0$. Then we use \eqref{eq:TransversalImmersive} as well as the formulas for the standard parameters for $a',c'=(c_1',c_2')$ and $\lambda'$ to obtain the following equation:
\begin{align*}
\widehat f_{2 z w}(0) = ~& \frac{u^2 u' \lambda}{g_w(0)^2} \Bigl(c u g_w(0)\Delta(1,0;2,0) + \lambda \bigl (g_{zw}(0) \Delta(0,1;1,0) + g_w(0) \Delta(1,0;1,1) \bigr) \Bigr) = 0,
\end{align*}
which has a unique solution $c\in \C$ by \eqref{eq:F2Cond} and is given by
\begin{align*}
c = & ~-\frac{\lambda  \bigl( g_{z w}(0) \Delta(0,1;1,0) + g_w(0)\Delta(1,0;1,1)  \bigr)}{u g_w(0)\Delta(1,0;2,0)}.
\end{align*}
Then using the representations for $a', \lambda'$ and equation \eqref{eq:TransversalImmersive}:
\begin{align*}
\widehat f_{2 z^2}(0)  = & ~\frac{u^3 u' \lambda\Delta(1,0;2,0)}{ |f_z(0)|^2_{\eps}} = 2,
\end{align*}
the solution is given by
\begin{align}
\label{eq:parametersLambdaU}
\lambda = \frac {2 |f_z(0)|^2_{\eps}}{|\Delta(1,0;2,0) |}, \qquad u' = \frac{ |\Delta(1,0;2,0) |}{u^3 \Delta(1,0;2,0)},
\end{align}
since $\Delta(1,0;2,0) \neq 0$. Then, using all the previously deduced standard parameters, we compute $\widehat f_{1 w^2}(0) = T_1\bigl(j_0^2H\bigr)/u$, where $T_1\bigl(j_0^2H\bigr) \in \C$ is a real-analytic function in $j_0^2 H$, which does not depend on $u$. Thus there is a $u$ with $|u|=1$, such that $\widehat f_{1 w^2}(0) = |\widehat f_{1 w^2}(0)| \geq 0$. Finally we consider the following coefficients, where $\lambda>0$ is given by \eqref{eq:parametersLambdaU}
\begin{align}
\nonumber
\re\bigl(\widehat g_{w^2}(0)\bigr)  = & -2 r  - 2 r' \lambda^2 g_w(0) + T_2\bigl(j_0^2H\bigr)= 0,\\ 
\label{eq:f2z2w}
\re\bigl(\widehat f_{2 z^2 w}(0)\bigr)  = & -2 r - r' \lambda^2 g_w(0) + T_3\bigl(j_0^3H\bigr) = 0,
\end{align}
where $T_2\bigl(j_0^2H\bigr), T_3\bigl(j_0^3H\bigr) \in \R$ are real-analytic functions in $j_0^2H$ and $j_0^3H$ respectively and both do not depend on $r$ or $r'$. Thus we can uniquely solve for the real parameters $r$ and $r'$.
\end{proof}

\begin{remark}[label=rem:moreCoefficientsN]
Further inspection of $T_3$ in \eqref{eq:f2z2w} shows that the coefficients of $H$ at $0$ of order $3$ occurring in $T_3$ are $f_{z^3}(0)$ and $H_{z^2 w}(0)$.
\end{remark}

\begin{remark}[label=remark:SummaryJetNormalizedMapping]
If we assume $H \in \NTwo$, then by inspecting \eqref{eq:MapEq} we obtain the following conditions for some coefficients belonging to the $3$-jet of $H$ at $0$.
\begin{multicols}{2}
\begin{compactitem}
\item[\rm{(i)}]  $H(0)= (0,0,0)$
\item[\rm{(ii)}]  $H_z(0) = (1,0,0)$
\item[\rm{(iii)}]  $H_w(0) = (0,0,1)$
\item[\rm{(iv)}] $H_{z^2}(0) = (0,2,0)$
\item[\rm{(v)}] $H_{zw}(0) = (\frac{\imu \eps} 2 ,0,0)$
\item[\rm{(vi)}]  $H_{w^2}(0) = (|f_{1 w^2}(0)|, f_{2 w^2}(0),0)$
\item[\rm{(vii)}] $H_{z^3}(0)=(0, 12 \eps |f_{1 w^2}(0)|, 0)$
\item[\rm{(viii)}] $H_{z^2 w}(0) = (4 \imu |f_{1 w^2}(0)|,\imu \im(f_{2 z^2 w}(0)), 0)$
\item[\rm{(ix)}] $H_{z w^2}(0) = \bigl(\frac 1 {4} \bigl(-1 + 2 \re(g_{w^3}(0))\bigr), \\
 f_{2 z w^2}(0), 2 \imu  |f_{1 w^2}(0)|\bigr)$ 
\end{compactitem}
\end{multicols}
\end{remark}

\section{Classification of Mappings in \texorpdfstring{$\NTwo$}{N}}
\label{sec:MappingsInN}

Let us state the explicit version of \autoref{theorem:ReductionOneParameterFamilies}.
\begin{theorem}[label=theorem:ReductionOneParameterFamilies2]
The set $\NTwo$ consists of the following mappings, where for $H =(f_1,f_2,g)  \in \NTwo$ we denote the parameter $s \coloneqq 2 f_{1 w^2}(0) \geq0$:
\begin{align*}
\MapOneParameterFamilyTwo{1}{\eps}(z,w)\coloneqq&~ \Bigl(2 z (2 + \imu \eps w),4 z^2,4 w\Bigr)/(4-w^2),\\
\MapOneParameterFamilyThree{2}{s}{\eps}(z,w)\coloneqq& ~\Bigl(4 z - 4 \eps s z^2 + \imu (\eps  - s^2) z w  + s w^2,  4 z^2 + s^2 w^2,w (4- 4 \eps s z - \imu  (\eps +s^2) w) \Bigr) \\
& \quad  \slash \Bigl( 4 - 4 \eps s z - \imu (\eps+ s^2) w   - 2 \imu s z w-  \eps s^2 w^2 \Bigr),\\
\MapOneParameterFamilyThree{3}{s}{\eps}(z,w)\coloneqq & ~\Bigl(256 \eps z  + 96 \imu z w + 64 \eps s w^2 + 64 z^3 + 64 \imu \eps s z^2 w - 3 (3 \eps - 16 s^2) z w^2  + 4 \imu s w^3,\\
& \quad  256 \eps z^2 - 16 w^2 + 256s z^3 + 16 \imu z^2 w - 16 \eps s z w^2 - \imu \eps w^3, \\
& \quad w(256 \eps - 32 \imu w + 64 z^2  - 64 \imu \eps s z w  -  (\eps + 16 s^2) w^2)\Bigr)  \slash  \Bigl( 256 \eps - 32 \imu w + 64 z^2 \\
& \quad - 192 \imu \eps s z w  - (17 \eps + 144 s^2) w^2  + 32 \imu \eps z^2 w + 24 s z w^2 + \imu w^3 \Bigr).
\end{align*}
Each mapping in $\NTwo$ is not isotropically equivalent to any different mapping in $\NTwo$.
\end{theorem}

The family of mappings $\MapOneParameterFamilyThree{3}{s}{\eps}$ in \autoref{theorem:ReductionOneParameterFamilies2} is not of degree $3$ for each $s \geq 0$: If we set $\eps = -1$ and $s = 1/2$ in $\MapOneParameterFamilyThree{3}{s}{\eps}$  the denominator and the numerator of each component is divisible by $16 \imu - 8 \imu z + w$, resulting in a mapping of degree $2$, which coincides with $\MapOneParameterFamilyThree{2}{1/2}{-}$. The following lemma shows that this is the only possibility.

\begin{lemma}[label=lem:WhenDeg3IsDeg2]
The mapping $\MapOneParameterFamilyThree{3}{s}{\eps}$ from \autoref{theorem:ReductionOneParameterFamilies2} is of degree $2$ if and only if $\eps = -1$ and $s= 1/2$ in $\MapOneParameterFamilyThree{3}{s}{\eps}$.
\end{lemma}

\begin{proof}
The necessary direction can be verified directly. The other direction is proved as follows: We let $H$ denote an arbitrary rational mapping of degree $2$ with $H(0) = 0$ defined in a sufficiently small neighborhood $U\subset \C^2$ of $0$. We require $H$ to be holomorphic in $U$. Then $H$ is of the form $H=(p_1,p_2,p_3)/q$, where for $1\leq j\leq 3$ the terms $p_j$ and $q$ are polynomials of degree $2$ given by $p_j(z,w) = a_j z + b_j w + c_j z^2 + d_j z w + e_j w^2$ and $q(z,w)  = 1+ a_4 z + b_4 w + c_4 z^2 + d_4 z w + e_4 w^2$, where each element of $\Lambda_m\coloneqq \{a_m,b_m,c_m,d_m,e_m\}$ is a complex number for $1\leq m \leq 4$. We denote by $\Lambda$ the collection of all $\Lambda_m$. If we compare the $3$-jets of $H$ and $\MapOneParameterFamilyThree{3}{s}{\eps}$ and solve for the elements of $\Lambda$ we obtain 
\begin{align*}
H(z,w) = \frac{\Bigl( 16 z - 16 s \eps z^2 + 5 \imu \eps z w + 4 s w^2, 16 z^2 - \eps w^2, w \bigl(16 - 16 \eps s z - 3 \imu \eps w\bigr)\Bigr)}  {16 + 16 \eps s z + 3 \imu \eps w + 8 \imu s z w + (1 + 8 \eps s^2) w^2}.
\end{align*}
Comparing the $f_{1 z^3 w}(0)$-coefficients of $H$ and $\MapOneParameterFamilyThree{3}{s}{\eps}$ we find a solution if and only if $\eps = -1$ and $s=1/2$. Then we observe with these choices the mapping $H$ coincides with $\MapOneParameterFamilyThree{3}{1/2}{-}$.
\end{proof}

The proof of \autoref{theorem:ReductionOneParameterFamilies2} is based on the following lemmas. First we state them and then we show how \autoref{theorem:ReductionOneParameterFamilies2} is deduced from these lemmas. Afterwards we provide the proofs of these lemmas. \\
In the first lemma we obtain a so-called \textit{jet parametrization} for $H \in \NTwo$ at $0$ along the second Segre set. In order to simplify our formulas we introduce the following notation:
\begin{align*}
A_{k \ell}\coloneqq f_{1 z^k w^{\ell}}(0),\quad B_{k\ell}\coloneqq f_{2 z^k w^{\ell}}(0),\quad C_{k\ell}\coloneqq g_{z^k w^{\ell}}(0), \quad D_{\ell}: = D_{0\ell}, 
\end{align*}
for $k, \ell \geq 0$ and $D \in\{A,B,C\}$. In the list of coefficients of a mapping $H \in \FTwo$ we gave in \autoref{remark:SummaryJetNormalizedMapping}, there are still some unknown coefficients belonging to $J_0^4$. These remaining coefficients we denote by  
\begin{align}
\label{eq:BasicIdentityUnknowns}
j \coloneqq \left(A_2, B_2, B_{21}, B_{12},A_3, B_3, C_3, A_{22},B_{22},C_{22},A_{13},B_{13}, C_{13},A_4,B_4,C_4\right).
\end{align}
We refer to the coefficients $D_{k\ell}$ we listed in \eqref{eq:BasicIdentityUnknowns} as components of $j$. We take $N_0 \coloneqq 16$ and define the following set:  
\begin{align}
\label{eq:BasicIdentityCollectionJets}
J \coloneqq \left\{ j \in \C^{N_0}: A_2 \geq 0, C_3 \in \R, B_{2 1} \in \imu \R \right\} \subset \C^{N_0}.
\end{align}
We consider $j$ from \eqref{eq:BasicIdentityUnknowns} as variable for $J \subset \C^{N_0}$. The following lemma is based on \cite[Proposition 25, Corollary 26--27]{lamel}.

\begin{lemma}[label=lemma:BasicIdentity2Nondeg]
Let $H \in \NTwo$. Then there exists an explicitly computable, rational mapping $\Psi$ satisfying
\begin{align}
\label{eq:BasicIdentity}
H(z,2 \imu z \chi)  = \Psi\bigl(z,\chi, j\bigr) 
\end{align} 
for all $(z,\chi) \in \C^2$ sufficiently near $0$. The formula for $\Psi$ is given in \hyperref[appendix:ParametrizationFormula]{Appendix A}, where we scaled $j \in J$ for simplification. 
\end{lemma}

\begin{remark}
 In order to compute $\Psi$ in \Autoref{lemma:BasicIdentity2Nondeg} we only need to assume the nondegeneracy of $H$, but to simplify expressions we require $H \in \NTwo$. 
\end{remark}

The approach we take in the next lemmas follows the line of thought of \cite[Proposition 2.11--3.1,\S6]{BER97}. 

\begin{lemma}[label=lemma:Desingularization]
Let $H \in \NTwo$ and $\Psi$ be given as in \Autoref{lemma:BasicIdentity2Nondeg}. If $\psi(z,w) \coloneqq \Psi\bigl(z, w / (2 \imu z), j\bigr)$ is holomorphic for $(z,w) \in \C^2$ near $0$ and $j_0^4 \psi = j_0^4 H$, then $\psi \in \{\psi_1, \dots, \psi_5\}$ is of at most degree $3$ and depends on $A_2,B_2, B_{21}, A_{22},B_{22}$ and $C_{22}$ satisfying $A_2 \geq 0$ and $\re(B_{21})=0$, whenever these parameters are present in $\psi$. The concrete formulas for $(\psi_k)_{k=1,\ldots,5}$ are listed in \hyperref[appendix:FormulaPsi]{Appendix C}. 
\end{lemma}

\begin{lemma}[label=lemma:OneParameterFamilies]
Let $U \subset \C^2$ be a sufficiently small neighborhood of $0$ and $\psi \in \{\psi_1, \dots, \psi_5\}$ from \Autoref{lemma:Desingularization} satisfies $\psi(U \cap \HeisenbergOne{2}) \subset \HeisenbergTwo{\eps}{3}$. Then $\psi \in \{\MapOneParameterFamilyTwo{1}{\eps}, \MapOneParameterFamilyThree{2}{s}{\eps},\MapOneParameterFamilyThree{3}{s}{\eps}\}$ from \Autoref{theorem:ReductionOneParameterFamilies2}, where $s \coloneqq A_2 \geq0$.
\end{lemma}

Next we describe how to prove \autoref{theorem:ReductionOneParameterFamilies2} from the previously stated lemmas.

\begin{proof}[Proof of \autoref{theorem:ReductionOneParameterFamilies2}]
Let $H \in \NTwo$ and $U \subset \C^2$ be a sufficiently small neighborhood of $0$. We write $\rho(z,w,\chi,\tau) \coloneqq w - \tau - 2 \imu z \chi$ for a defining function of the complexification of $\HeisenbergOne{2}$. The second Segre set $\mathcal S^2_0$ of $\HeisenbergOne{2}$ at $0$ is given as the image of $v^2_0(z,\chi) \coloneqq (z, 2 \imu z \chi)$, for $(z,\chi) \in U$. Since $v^2_0$ is of rank $2$ outside of the complex variety $X \coloneqq \big\{(z,\chi) \in U: z=0 \big\}$ in $\C^2$, it follows that $\mathcal S^2_0$ contains an open set $V \subset U\setminus X$ of $\C^2$. From \Autoref{lemma:BasicIdentity2Nondeg} we know after scaling the variable $j \in J$ from \eqref{eq:BasicIdentityUnknowns}, that
\begin{align}
\label{eq:PsiTaylor}
H\bigl(v^2_0(z,\chi) \bigr)=\Psi(z,\chi,j) = \sum_{k,\ell} \Psi_{k\ell}(j ) z^k \chi^{\ell}, 
\end{align}
holds, where we have written $\Psi$ in the Taylor expansion with coefficients $\Psi_{k\ell}(j) \in \C^3$ depending on $j \in J$. Then for $(z,w) \in V$ we have:
\begin{align}
\label{eq:NonHolomorphic}
H(z,w) = H\left(v^2_0\left(z, \frac w {2 \imu z}\right) \right) = \Psi \left(z, \frac w {2 \imu z},j \right) = \sum_{\alpha,\beta} \widehat \Psi_{\alpha\beta}(j) z^{\alpha} w^{\beta},
\end{align}
where $\widehat \Psi_{\alpha\beta} (j) \in \C^3$. The left-hand side of \eqref{eq:NonHolomorphic} is required to be holomorphic in a neighborhood of $0$, thus \eqref{eq:NonHolomorphic} yields equations $\widehat \Psi_{\alpha \beta}(j) = 0$ for $\alpha < 0$ or equivalently we obtain $\Psi_{k\ell}(j) = 0$ for $\ell > k$ by \eqref{eq:PsiTaylor}. We examine these equations for $j$ in the proof of \Autoref{lemma:Desingularization} to end up with $\Psi \bigl(z,w/(2 \imu z), j \bigr)$ being one of $5$ holomorphic mappings $\widehat \psi_1(z,w), \ldots, \widehat \psi_5(z,w)$, defined in a neighborhood of $0$ and given in \hyperref[appendix:FormulaPsi]{Appendix C}. Moreover \eqref{eq:NonHolomorphic} can only hold if $j_0^4H(z,w) = j_0^4 \Psi\bigl(z, w / (2 \imu z), j \bigr)= j_0^4 \widehat \psi_k(z,w)$ for each $1 \leq k \leq 5$. We carry out these computations in the last part of the proof of \Autoref{lemma:Desingularization}, which yield $H$ being one of the holomorphic mappings $\psi_1, \ldots, \psi_5$ according to \autoref{lemma:Desingularization} listed in \hyperref[appendix:FormulaPsi]{Appendix C}. Since we require $H$ being a mapping of $\HeisenbergOne{2}$ to $\HeisenbergTwo{\eps}{3}$ and $j$ was an arbitrary variable in $J$ so far, we have to ensure $\psi_k$ sends $\HeisenbergOne{2}$ to $\HeisenbergTwo{\eps}{3}$ for $1 \leq k \leq 5$. This last step is carried out in \Autoref{lemma:OneParameterFamilies} and we end up with the mappings $\MapOneParameterFamilyTwo{1}{\eps}, \MapOneParameterFamilyThree{2}{s}{\eps}$ and $\MapOneParameterFamilyThree{3}{s}{\eps}$ as claimed in \autoref{theorem:ReductionOneParameterFamilies2}, where $s = 2 f_{1 w^2}(0)$. The last claim, that the maps we listed in \autoref{theorem:ReductionOneParameterFamilies2} are not isotropically equivalent is proved in \autoref{theorem:NonIntersectingOrbits} below.
\end{proof}

\subsection{Jet Parametrization}

\begin{proof}[Proof of \autoref{lemma:BasicIdentity2Nondeg}]
We need to carry out the following steps: From the mapping equation we can determine $H$ along the germ of the second Segre set $\mathcal S^2_0$ of $\HeisenbergOne{2}$ near $0$ in terms of the $2$-jet of $H$ evaluated along the germ of the conjugated version of the first Segre set $\bar{\mathcal S}^1_0 =\{(\chi, 0): \chi \in \C\}$ of $\HeisenbergOne{2}$ near $0$. In a similar way we obtain formulas for the $2$-jet of $H$ along $\mathcal S^1_0$ depending on $j\in J$. In both steps it is essential to assume $2$-nondegeneracy. The resulting representation of $H$ gives the desired mappings $\Psi$ depending on $j$. Now we give the detailed version of the proof.\\
We let $\rho'$ be a defining function for $\HeisenbergTwo{\eps}{3}$ and $L \coloneqq \frac{\partial}{\partial \chi} - 2 \imu z  \frac{\partial}{\partial \tau}$ a basis of CR-vector fields of the complexification of $\HeisenbergOne{2}$. We compute $\Phi_{r+1}(z,w,\chi,\tau)\coloneqq L^r \rho' \bigl(H(z,w),\bar H(\chi,\tau) \bigr)$ for $ 0 \leq r \leq 2$ to obtain
\begin{align}
\label{eq:systemPsi}
\nonumber
\Phi_1(z,w,\chi,\tau) \coloneqq & ~g(z,w) - \bar g(\chi,\tau) - 2 \imu (f_1(z,w) \bar f_1(\chi,\tau) + \eps f_2(z,w) \bar f_2(\chi,\tau) ),\\
\Phi_2(z,w,\chi,\tau)  \coloneqq & - L \bar g(\chi,\tau)  - 2 \imu \Bigl(f_1(z,w) L \bar f_1 (\chi,\tau) + \eps f_2(z,w) L \bar f_2(\chi,\tau)\Bigr), \\
\nonumber
\Phi_3(z,w,\chi,\tau): = & -L^2 \bar g(\chi,\tau)  - 2 \imu \Bigl(f_1(z,w) L^2 \bar f_1(\chi,\tau) + \eps f_2(z,w) L^2 \bar f_2(\chi,\tau)\Bigr).
\end{align}
We introduce the following variables for expressions which occur in $\Phi_j$ for $1 \leq j \leq 3$:
\begin{align*}
(Z,\zeta) & \coloneqq (z,w,\chi,\tau), \quad  (Z', \zeta') \coloneqq \bigl(H(z,w), \bar H(\chi,\tau) \bigr), \quad
 W  \coloneqq \left(\frac{\partial^{|\beta|}} {\partial \zeta^{\beta}} \bar H(\chi,\tau) \right)_{1 \leq |\beta| \leq 2}.
\end{align*}
By a slight abuse of notation we obtain $\Phi_j(Z,\zeta,Z',\zeta',W) = 0$ for $1 \leq j \leq 3$ when restricted to $\HeisenbergOne{2}$, i.e., setting $Z = (z, \tau + 2 \imu z \chi)$. Further if we write $\Phi \coloneqq (\Phi_1,\Phi_2,\Phi_3)$ we have 
\begin{align*}
\det \left(\frac{\partial \Phi}{\partial Z'}(0)\right) = \eps \Bigl(\bar f_{2 \chi}(0) \bar f_{1 \chi^2}(0) - \bar f_{1 \chi}(0) \bar f_{2 \chi^2}(0)\Bigr) = -\eps \neq 0,
\end{align*}
since we assumed $H \in \NTwo \subset \FTwo$. Hence we can explicitly solve the system given in $\eqref{eq:systemPsi}$ for $Z'$ near $0$ as follows. First we denote 
\begin{align*}
B(z,\chi,\tau) \coloneqq \left(
\begin{array}{ccc}
\bar f_1(\chi,\tau) &\eps  \bar f_2(\chi,\tau) & -\frac{\imu} 2 \\
L \bar f_1(\chi,\tau)   &\eps L \bar f_2(\chi,\tau) &  0\\
L^2 \bar f_1(\chi,\tau)  &\eps L^2\bar f_2(\chi,\tau)  &  0
\end{array}
\right),
\end{align*}
then we obtain for all $(z,\chi,\tau) \in \C^3$ near $0$ the following identity
\begin{align}
\label{eq:BasicIdentityEq1}
H(z,\tau + 2 \imu z \chi)=
-\frac 1 {2 \imu} B^{-1}(z,\chi,\tau) ~^t(\bar g(\chi,\tau),L \bar g(\chi,\tau),L^2 \bar g(\chi,\tau)).
\end{align}
If we evaluate $\eqref{eq:BasicIdentityEq1}$ at $\tau = 0$ we obtain a formula for $H$ along $\mathcal S^2_0$ depending on the $2$-jet of $\bar H$ along $\bar{\mathcal S}^1_0$. So to finish our computations we need to find formulas for $j^2_{(\chi,0)}\bar H$. 
To this end we introduce the vector field $S$ tangent to $\HeisenbergOne{2}$ defined as $S \coloneqq \partial / \partial w + \partial / \partial \tau$, such that $S^{k} H(z,\tau + 2 \imu z \chi) = H_{w^k}(z,\tau + 2 \imu z \chi)$ for $k \in \N$. Applying $S$ and $S^2$  to $\eqref{eq:BasicIdentityEq1}$ and setting $\chi=\tau = 0$ we obtain formulas for $H_w(z,0)$ and $H_{w^2}(z,0)$ respectively, which are rational and depend on $j \in J$. After conjugating these expressions we obtain the components of $j^2_{(\chi,0)}\bar H$ as rational functions of $j$. The resulting mapping is denoted by $\Psi$ and depends on $j \in J$. In order to get rid of powers of $2$ in formulas we scale $j$ as follows: 
\begin{align*}
&(A_2, B_2, B_{12},A_3, B_3, C_3, A_{22}, B_{22}, A_{13}, B_{13}, C_{13} , A_4, B_4, C_4) \mapsto\\
\nonumber
  & \left(\frac{A_2} 2, \frac{B_2} 2, \frac{B_{12}} 4,  \frac{A_3} 4, \frac{B_3} 4, \frac{C_3} 2, \frac{A_{22}} 2, \frac{B_{22}} 2, \frac{A_{13}} 8, \frac{B_{1 3}} 8, \frac{C_{13}} 4, \frac{A_4} 8, \frac{B_4} 8, \frac{C_4} 4 \right).
\end{align*}
The numerator of the components of $H$ are polynomials of highest degree $(3,8)$ in $(z,\chi)$ and are homogeneous in $z$. The components of $H$ have the same denominator, which is a polynomial of highest degree $(3,9)$ in $(z,\chi)$. The complete expression is listed in \hyperref[appendix:ParametrizationFormula]{Appendix A}.
\end{proof}

\subsection{Desingularization}

We introduce the following relation:

\begin{definition}
\label{definition:specialCase}
For $J_1, J_2\subset J$ from \eqref{eq:BasicIdentityCollectionJets} we denote variables $j_1 \in J_1$ and $j_2 \in J_2$ as in \eqref{eq:BasicIdentityUnknowns} respectively. We set $\Psi_1(z,\chi)\coloneqq \Psi(z,\chi,j_1)$ and $\Psi_2(z,\chi) \coloneqq \Psi(z,\chi,j_2)$, where $\Psi$ is given in \autoref{lemma:BasicIdentity2Nondeg}. We say that $\Psi_1$ is a \textit{special case} of $\Psi_2$, if $J_1 \subset J_2$.
\end{definition}
More geometrically this means that the variety given by the defining equations for $\Psi_1$ is contained in the variety generated by the defining equations for $\Psi_2$. 

\begin{proof}[Proof of \autoref{lemma:Desingularization}]
As described in the proof of \autoref{theorem:ReductionOneParameterFamilies2}, in \eqref{eq:PsiTaylor} we expand the mapping $\Psi\bigl(z,\chi, j\bigr)$ from \eqref{eq:BasicIdentity} into a power series $\Psi(z,\chi,j) = \sum_{k,\ell} \Psi_{k\ell} (j) z^k \chi^{\ell}$ around $0$. For the components we write $\Psi_{k\ell}(j) = \Bigl(\Psi^1_{k\ell}(j), \Psi^2_{k\ell}(j),\Psi^3_{ k\ell} (j) \Bigr)$ and we set
\begin{align}
\label{eq:PsiEqual0}
\Psi_{k\ell} (j) = 0,\qquad  \forall ~\ell > k,
\end{align}
according to \eqref{eq:PsiTaylor}. These equations allow us to obtain conditions for $j \in J$. Each solution of an equation from \eqref{eq:PsiEqual0} corresponds to considering maps as in \eqref{eq:BasicIdentity}, but instead $j \in J$ we have $j \in J'$, where $J'$ is a subvariety of $J$. This means that we gradually restrict the space of possible mappings in $\FTwo$. In the following we describe which coefficients $\Psi_{k \ell}$ we consider and which components of $j$ we can eliminate from equations given as in \eqref{eq:PsiEqual0}. \\
We start considering $\Psi^3_{34} = \Psi^1_{34} = \Psi^3_{45} = \Psi^1_{45}=0$, which determine the following components of $j$:
\begin{align*}
A_3 = ~& \frac{\imu}{2}\Bigl(6 A_2^3 + 3 \eps B_{12} - A_2 \bigl(6 B_2 + \eps (-3 + C_3) \bigr) \Bigr)\\
B_3 =~ & \frac{\eps}{10} \bigl(-18 \imu \eps A_2^4 + 15 \imu A_2  B_{12} - 2 B_2 \bigl(9 \eps B_{21} - 4 \imu (3 - C_3) \bigr) \\
	& \quad  - 3 \imu A_2^2 (3 - 6 \eps B_2 + 6 \imu \eps B_{21}-  C_3) \Bigr)\\
A_4 = ~& \frac{\eps}{5} \Bigl(-324 \eps A_2^5 - 15 A_2^2 (-\eps A_{22}+ 2 B_{12} + \eps C_{13}) +  5 \bigl(-3 \eps A_{22} B_2 + \imu B_{13}   \\
	& \quad + B_{12} \bigl(-6 \imu B_{21}  + \eps (-6  + C_3) \bigr) +  3 \eps B_2 C_{13}\bigr) + A_2 \bigl(-5 \imu A_{13} + 30 \imu B_{21} + 10 \imu B_{21} C_3\\
	& \quad  - 5 \eps (6 B_{21}^2 + (-5 + C_3) C_3) + 5 \imu C_4 +  3 B_2(44 + 48 \imu \eps B_{21} - 18 C_3 + 15 \imu C_{22}) \bigr)  \\
	& \quad + 3 A_2^3 (-34 + 108 \eps B_2 - 28 \imu \eps B_{21} + 28 C_3 - 15 \imu C_{22}) \Bigr) \\
\end{align*}
\begin{align*}
B_4 = ~ & \frac{\eps}{20} \Bigl(3060 \eps A_2^6- 45 \eps B_{12}^2 + 2 B_2 \bigl(-40 \imu A_{13} + 102 \eps B_{21}^2 + 5 \imu B_{21} (-33 + 23 C_3) \\
	& \quad + \eps (-42 -  30 B_{22} + 78 C_3 - 28 C_3^2) + 40 \imu C_4\bigr) + 180A_2^3 (-\eps A_{22} + B_{12} + \eps C_{13}) \\
	& \quad + 20 A_2 \bigl(9 \eps A_{22} B_2 + \imu B_{13}+ 6 \imu B_{12} B_{21} +  2 \eps B_{12} C_3 - 3 B_2 (B_{12} + 3 \eps C_{13})\bigr)\\
	& \quad + A_2^2 \bigl(60 \imu A_{13} + 900 \eps  B_2^2 + 150 \imu B_{21} - 290 \imu B_{21} C_3 + \eps (9 - 24 B_{21}^2 +  60 B_{22} \\
	& \quad  - 106 C_3 + 61 C_3^2) - 60  \imu C_4 + 12 B_2 (-79 - 117 \imu \eps B_{21} + 63 C_3 - 65 \imu C_{22})\bigr) \\
	& \quad  - 12 A_2^4 (-69 +  330 \eps B_2 -  97 \imu \eps B_{21} + 73 C_3 - 45 \imu C_{22}) + 240 \imu B_2^2 C_{22} \Bigr)
\end{align*}
Then we consider $\Psi^2_{34}=0$ to obtain two cases, either 
\begin{compactitem}
\item[\rm{(i)}] Case A: \quad $B_{12} = \frac{2 \eps A_2} {5}\Bigl(6 A_2^2 + 5 B_2 + 6 \imu B_{21} + \eps (3 - C_3)\Bigr)$, \quad or
\item[\rm{(ii)}] Case B: \quad $B_2 = A_2^2$.
\end{compactitem}
Next we assume one of the expressions for $B_{12}$ or $B_2$ respectively for $\Psi$ and consider another equation from \eqref{eq:PsiEqual0} in order to solve for further components of $j$ in terms of the remaining elements. It turns out that each of the remaining equations of the system given in \eqref{eq:PsiEqual0} has more than one possible solution, resulting in case distinctions. In \hyperref[appendix:CaseAandB]{Appendix B} we give two diagrams of this elimination process for \hyperref[figure:CaseA]{case A} and \hyperref[figure:CaseB]{case B} respectively. In these diagrams we keep track of all the equations $\Psi_{k\ell}(j)=0$ we consider, which components of $j$ we are able to determine and which holomorphic expressions we obtain in the end. Now we describe the diagrams in a more detailed way:\\
Let us write $\gamma \coloneqq (A_2,C_3, B_{21},C_4, A_{13},B_{13}, C_{13},A_{22},B_{22},C_{22})$. In case A $\Psi$ still depends on the variables $\gamma$ and $B_2$ and in case B $\Psi$ depends on the variables $\gamma$ and $B_{12}$. Since both cases are treated in the same way we write $\Lambda$ for the set of the remaining variables in $\Psi$ with components denoted by $(D_1, \ldots, D_{11})$.\\
Inductively we consider equations $\Psi^j_{k\ell} =0$ to determine further variables $D_{m_1}, \ldots, D_{m_n} \in \Lambda$, where $1\leq m_j \leq 11$ for $1\leq j \leq n$. Each variable $D_{m_j}$ which we solved for corresponds to a case $E_{r {s_i}}$. It turns out  that we have $0 \leq r \leq 7$ and $1 \leq s_i \leq 13$, where $r=0$ corresponds to the starting node from case A or B. The notation for $E_{r {s_i}}$ is chosen such that the first index $r$ indicates the number of nodes one has to pass in order to get from the starting node, i.e., case A or case B from above, to $E_{r {s_i}}$. \\
Let us denote by $E$ some already achieved case, starting with case A or case B. In the diagram such an induction step is displayed as in the following \Autoref{figure:GeneralStep}:
\begin{figure}[H]
\begin{center}
\begin{tikzpicture}
\draw 
(1.5,0) node[text width= 1.5cm,align=center,rectangle,line width = 0.5pt,draw] (E) {$E$}
(4,0) node[fill=grayA,ellipse,line width = 0.5pt,draw] (Hkl) {$\Psi_{k\ell}=0$}
(6.7,1) node[text width= 2cm,align=center,rectangle,line width = 0.5pt,draw] (E1) {$E_{r {s_1}}$ \\ $D_{m_1} = \ldots$}
(6.7,-1) node[text width= 2cm,align=center,rectangle,line width = 0.5pt,draw] (E3) {$E_{r {s_n}}$ \\ $D_{m_n} = \ldots$};
\draw[->,line width=1pt]  (E) -- (Hkl);
\draw[->,line width=1pt]  (Hkl) -- (E1);
\draw[->,line width=1pt]  (Hkl) -- (5.7,0);
\draw[line width=1pt,dotted]  (6.7,0.3) -- (6.7,-0.3);
\draw[->,line width=1pt]  (Hkl) -- (E3);
\end{tikzpicture}
\caption{Diagram for new cases}
\label{figure:GeneralStep}
\end{center} 
\end{figure}
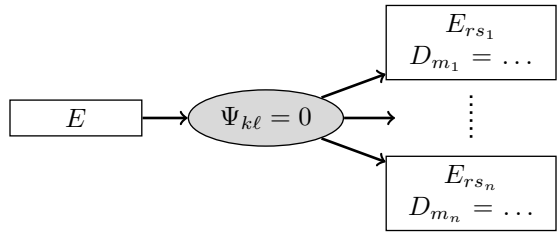
Now we take all parameters from the preceding cases of $E_{r {s_i}}$, plug them into $\Psi$ and denote the resulting rational mapping by $\varphi(z,\chi)$. Then we have several possibilities: 
\begin{compactitem}
\item[\rm{(i)}] If $\varphi(z,\frac w {2 \imu z})$ is holomorphic near $0$ we do not consider further equations. Then we have the possibility that $\varphi$ is a special case of a holomorphic mapping $\varphi'$ from some other case, which is indicated in \autoref{figure:SpecialCasesHolomorphic} or $\varphi$ is not a special case of any of the occurring mappings in the diagrams, which is indicated in \Autoref{figure:NewMap}.
\begin{figure}[H]
\begin{center}
\begin{tikzpicture}
\draw 
(1.5,0) node[text width= 2cm,align=center,rectangle,line width = 0.5pt,draw] (E) {$E_{r {s_i}}$ \\ $D_{k_m} = \ldots$}  
(4,0) node[diamond,line width = 0.5pt,draw] (psi) {$\varphi'$};
\draw[->,line width=1pt]  (-0.5,0) -- (E);
\draw[->,densely dotted,line width=1pt]  (E) -- (psi);
\end{tikzpicture}
\caption{Diagram for special cases of holomorphic maps}
\label{figure:SpecialCasesHolomorphic}
\end{center}
\end{figure}
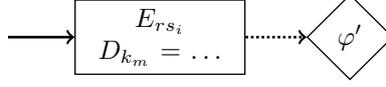 
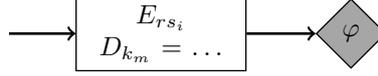
\begin{figure}[H]
\begin{center}
\begin{tikzpicture} 
\draw 
(1.5,0) node[text width= 2cm,align=center,rectangle,line width = 0.5pt,draw] (E) {$E_{r {s_i}}$ \\ $D_{k_m} = \ldots$}  
(4,0) node[fill=grayB,diamond,line width = 0.5pt,draw] (phi) {$\varphi$};
\draw[->,line width=1pt]  (-0.5,0) -- (E);
\draw[->,line width=1pt]  (E) -- (phi);
\end{tikzpicture}
\caption{Diagram for new holomorphic maps}
\label{figure:NewMap}
\end{center}
\end{figure} 
\item[\rm{(ii)}]  If $\varphi(z,\frac w {2 \imu z})$ is not holomorphic, we either proceed with another induction step as shown in \Autoref{figure:GeneralStep} or we recognize that the mapping $\varphi$ is a special case of a mapping $\varphi''$ from some case $E_{r'' {s''_i}}$. We indicate this situation as $E_{r {s_i}} \subset E_{r'' {s''_i}}$, which is shown in the following \Autoref{figure:SpecialCases}.
\begin{figure}[H]
\begin{center}
\begin{tikzpicture}
\draw (1.5,0) node[text width= 3cm,align=center,rectangle,line width = 0.5pt,draw] (E) {$E_{r {s_i}} \subset E_{r'' {s''_i}}$ \\ $D_{k_m} = \ldots$} ;
\draw[->,line width=1pt]  (-1,0) -- (E);
\end{tikzpicture}
\caption{Diagram for special cases of maps}
\label{figure:SpecialCases}
\end{center}
\end{figure}
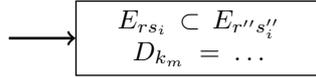 
\end{compactitem}
The complete case distinction is carried out in \hyperref[appendix:CaseAandB]{Appendix B}, where we denote the cases $E_{r s_i}$ by ``$Ars_i$'' and ``$Brs_i$'' for case $A$ and $B$ respectively. As mentioned above after at most $7$ steps the process terminates, which means, that after setting $\chi = \frac w {2 \imu z}$ in $\Psi(z,\chi,j)$ we obtain a holomorphic expression. It turns out that we obtain $5$ rational, holomorphic mappings, which we denote by $\widehat \psi_k(z,w)$ for $1 \leq k\leq 5$, as can be seen in the diagrams and is indicated in \Autoref{figure:NewMap}. We point out that these mappings include all $H \in \NTwo$ by construction. The formulas for $\widehat \psi_k$ are given in \hyperref[appendix:FormulaPsi]{Appendix C}. \\
We write $\widehat \psi_k = (\widehat \psi_k^1, {\widehat \psi_k}^2,\widehat \psi_k^3)$ and proceed by verifying $j_0^4 \widehat \psi_k = j_0^4 H$. Whenever we have expressed one component of $j$ in terms of the remaining components we use this expression for the subsequent equations.\\
First we treat $\widehat \psi_1$ and consider $\widehat \psi_{1w^3}^3(0) = C_3 / 2$ to obtain $C_3 = 3(1 + 2 \imu \eps B_{21})$. Next we consider $\widehat \psi_{1z^2w^2}^1(0) = A_{22} / 2$ to get $C_{13} = 3 A_{22} / 2$. Then we inspect $\widehat \psi_{1z^2 w^2}^2(0) = \frac{B_{22}} 2$ which gives $
C_4 = A_{13} + 18 B_{21} - 6 \imu \eps (1 - B_{21}^2)$. Verifying the normalization conditions we obtain $\re(B_{21})=0$ and we end up with the mapping $\psi_1$ as claimed, which still depends on $B_{21},A_{22},B_{22}$ and $C_{22}$ and is given in \hyperref[appendix:FormulaPsi]{Appendix C}. For $\widehat \psi_2$ we start with considering $\widehat \psi_{2 z w^2}^1(0) =\frac{A_{13}} 8$ to obtain
\begin{align*}
A_{13} =-10 B_{21} + \imu \eps (4 + B_{22}) + C_4 - 2 \imu \eps A_2 (A_{22} - C_{13}) + 2 A_2^2 (6 \imu - C_{22}),
\end{align*}
such that $\widehat \psi_2$ is independent of $B_{22}$ and $C_4$. Then we compute $\widehat \psi_{2 z w^3}^3(0) = \frac{C_{13}} 4$ to get
\begin{align*}
C_{13} = \frac 3 2 \Bigl(A_{22} + A_2 \bigl(2 \imu B_{21} + \eps (4 - \imu C_{22})\bigr)\Bigr).
\end{align*}
The rest of the coefficients are already in the correct form and the normalization conditions give $A_2 \geq 0$ and $\re(B_{21})=0$. The resulting mapping is denoted by $\psi_2$, depends on $A_2,B_{21},A_{22}$ and $C_{22}$ and is given in \hyperref[appendix:FormulaPsi]{Appendix C}. The maps $\widehat \psi_k$ for $k=3,4,5$ already satisfy $j_0^4 \widehat \psi_k = j_0^4H$ and by verifying the normalization conditions we obtain for $k=3,5$ that $A_2 \geq 0$ and additionally for $k=3$ that $\re(B_{21})=0$. Finally we denote $\psi_k = \widehat\psi_k$ for $k=3,4,5$. The mapping $\psi_3$ depends on $A_2,B_2$ and $B_{21}$, $\psi_4$ on $B_2$ and $C_{22}$ and $\psi_5$ depends on $A_2$ and $C_{22}$. All these mappings are listed in \hyperref[appendix:FormulaPsi]{Appendix C}.
\end{proof}

\subsection{Reduction to One-Parameter-Families of Mappings}

\begin{proof}[Proof of \autoref{lemma:OneParameterFamilies}]
We plug $\psi_k$ into the complexified version of the mapping equation \eqref{eq:MapEq} and compare coefficients with respect to $z,\chi$ and $\tau$. We list the monomials $z^k \chi^{\ell} \tau^m$ we consider in the mapping equation and for which of the remaining coefficients of $H$ in $\psi_k$ we are able to solve. Whenever $B_{21}$ is present in $\psi_k$ we write $B_{21} = \imu b_{21}$, where $b_{21} \in \R$. Moreover we recall that $A_2 \geq 0$.\\
We start with $\psi_1$ in which we have the terms $b_{21}, A_{22}, B_{22}$ and $C_{22}$. The coefficient of $\chi^2 \tau^2$ yields $C_{22} =0$ and $\chi \tau^3$ gives $A_{22} =0$. We write $B_{22} = \re(B_{22}) + \imu \im(B_{22})$ to get from $\tau^4$ that $\re(B_{22}) = 2 (1 - 3 \eps b_{21} + b_{21}^2)$.
The coefficient of $\tau^6$gives the following equation $\im(B_{22})^2 + 4 b_{21}^2 (1 - 2 \eps b_{21})^2 = 0$. Thus $\im(B_{22})=0$ and either $b_{21} = 0$ or $b_{21} = \eps/2$. The first case $b_{21} = 0$ results in $\MapOneParameterFamilyTwo{1}{\eps}$ and from the second case if $b_{21} = \eps/2$ we obtain $\MapOneParameterFamilyThree{2}{0}{\eps}$.\\
Next, we insert $\psi_2$ into $\eqref{eq:MapEq}$, which depends on $A_2,b_{21}, A_{22}$ and $C_{22}$. The coefficient of $\chi^2 \tau^2$ gives $C_{22} = 2 \imu \eps A_2^2$ and the coefficient of $z \chi^2 \tau^2$ shows $A_{22}= 2 (5 \imu A_2 b_{21} + 3 A_2^3)$. The coefficient of $\tau^4$ yields two cases: Either $b_{21}= - \frac {3 A_2^2} 2$ or $b_{21}=\frac{ \eps + A_2^2} 2$.

Assuming the first case $b_{21}= - \frac {3 A_2^2} 2$, we obtain from the coefficient of $\tau^6$ either $A_2=0$, which results in $\MapOneParameterFamilyTwo{1}{\eps}$, or $1 + 12 \eps A_2^2 + 48 A_2^4 + 64 \eps A_2^6 = 0$, which, since $A_2\geq 0$, has the only solution if we take $\eps = -1$ and $A_2 = 1/2$. This choice of parameters gives $\MapOneParameterFamilyThree{2}{1/2}{-}$.

In the second case $b_{21}=\frac{\eps + A_2^2} 2$ we immediately obtain the mapping $\MapOneParameterFamilyThree{2}{s}{\eps}$, where we set $s=A_2 \geq 0$.\\
If we handle $\psi_3$, which depends on $A_2,B_2$ and $b_{21}$, we first consider the coefficient of $\chi^2 \tau^2$ in $\eqref{eq:MapEq}$ to get $B_2 = A_2^2$. Then the coefficient of $\tau^4$ yields two cases:

The first one is $b_{21}= \frac{ A_2^2} 2$. If we consider the coefficient of $\chi \tau^3$ we obtain $A_2 = 0$ and thus the mapping $\MapOneParameterFamilyTwo{1}{\eps}$. 

The second case is $b_{21}=\frac{\eps + A_2^2} 2$ which again gives $\MapOneParameterFamilyThree{2}{s}{\eps}$ after setting $s=A_2\geq 0$.\\
Treating $\psi_4$, which depends on $B_2$ and $C_{22}$, we proceed as follows: The coefficient of $\chi^2 \tau^2$ shows $C_{22} = 2 \imu \eps \bar B_2$ and $\tau^4$ gives $B_2 = \frac{e^{\imu t}} 4$ for $t\in \R$. In order to get rid of $e^{\imu t}$ in $\psi_4$ we apply the following matrices 
\begin{align*}
U_2 \coloneqq
\left(\begin{array}{c c}
1/v & 0 \\
0 & 1 
\end{array}
\right), \qquad
U'_3\coloneqq
\left(\begin{array}{c c  c}
v& 0 & 0 \\
0 & v^2 &  0 \\
0 & 0 & 1 
\end{array}
\right) \qquad \text{ where} \quad v=\frac{ 2 e^{-\frac{ \imu t}{2}}} {1-\eps + \imu (1 + \eps)}  \in \UnitSphere, 
\end{align*}
to $\psi_4$, which do not affect the normalization. The resulting mapping is $\MapOneParameterFamilyThree{3}{0}{\eps}$.\\
Finally we deal with $\psi_5$ in which the terms $A_2$ and $C_{22}$ occur. We write $C_{22} = \re(C_{22}) + \imu \im(C_{22})$ and consider the coefficient of $\chi^2\tau^2$ to obtain $\im(C_{22})= -\frac 1 2$ and $\re(C_{22})=0$. We end up with the mapping $\MapOneParameterFamilyThree{3}{s}{\eps}$ after setting $s=A_2 \geq 0$, which completes the proof of the lemma.
\end{proof}

\subsection{Jet Determination}

In this section we provide a jet determination result based on \autoref{theorem:ReductionOneParameterFamilies2}.

\begin{corollary}[label=cor:jetDetermination]
Let $U \subset \C^2$ be a neighborhood of $0$ and $H: U \rightarrow \C^3$ a holomorphic mapping. We denote the components of $H$ by $H=(f,g)=(f_1,f_2,g)$ and write $j_0(H) \coloneqq \{ j^2_0 H, f_{z^2 w}(0)\}$. If for $H_1,H_2 \in \FTwo$ the coefficients belonging to $j_0(H_1)$ and $j_0(H_2)$ coincide, we have $H_1 \equiv H_2$.
\end{corollary}

\begin{proof}
We note that $\NTwo$ is the collection of the mappings $\MapOneParameterFamilyTwo{1}{\eps}, \MapOneParameterFamilyThree{2}{s}{\eps}$ and $\MapOneParameterFamilyThree{3}{s}{\eps}$ from \autoref{theorem:ReductionOneParameterFamilies2}. The only parameter left in maps belonging to $\NTwo$ is $s=2 f_{1w^2}(0)$. Let $H_1,H_2\in \NTwo$, then we verify that if the coefficients which belong to $j_0(H_1)$ and $j_0(H_2)$ coincide, this yields $H_1 \equiv H_2$.\\
If $s=0$ in $H_1$ or $H_2$, then if we compare $j_0(H_1)$ and $j_0(H_2)$ by looking at the coefficients $f_{2 w^2}(0)$ and $f_{2 z^2 w}(0)$, the only possibility is $H_1 \equiv H_2$. \\
If $s\neq 0$, the coefficient $f_{1 w^2}(0)$ yields that we may have $\MapOneParameterFamilyThree{2}{s}{\eps} = \MapOneParameterFamilyThree{3}{t}{\eps}$ for some $s,t\geq 0$. According to \autoref{lem:WhenDeg3IsDeg2} this is only possible if and only if $t=s=1/2$ and $\eps = -1$. In this case we have $\MapOneParameterFamilyThree{2}{1/2}{-} \equiv \MapOneParameterFamilyThree{3}{1/2}{-}$. Next we note the following: In order to be able to apply \autoref{theorem:ReductionOneParameterFamilies2} to a mapping $H \in \FTwo$ we need to compose $H$ with isotropies according to \autoref{proposition:NormalForm2Nondeg}. We see from the proof of \autoref{proposition:NormalForm2Nondeg} and \autoref{rem:moreCoefficientsN} that the standard parameters used to normalize $H$ precisely depend on the elements of $j_0(H)$ as well as $g_{z^2 w}(0)$ and $f_{z^3}(0)$. To show the dependence of $g_{z^2 w}(0)$ on $j_0^2 f$ we take derivatives of \eqref{eq:MapEq} twice with respect to $z$ and once with respect to $\tau$ and evaluate at $0$ to obtain
\begin{align*}
g_{z^2 w}(0) = 2 \imu \Bigl(f_{1 z^2}(0) \bar f_{1 w}(0) + \eps f_{2 z^2}(0) \bar f_{2 w}(0) \Bigr).
\end{align*}
To get rid of the dependence of $f_{z^3}(0)$ we consider the system of equations in \eqref{eq:systemPsi} and set $(w,\chi,\tau)=0$. Then due to the $2$-nondegeneracy of $H$ we can solve for $f(z,0)$, which then depends on elements of $j_0^2 H$. This completes the proof of the jet determination.
\end{proof}

\subsection{Isotropic Inequivalence of Mappings in \texorpdfstring{$\NTwo$}{N}}

The theorem below proves the remaining statement in \autoref{theorem:ReductionOneParameterFamilies2} which says that the isotropic orbit of a given normalized map does not intersect the isotropic orbit of a different normalized map.

\begin{theorem}[label=theorem:NonIntersectingOrbits]
Let $H_1, H_2 \in \NTwo$ and $(\sigma,\sigma') \in \Isotropies$ such that $\sigma' \circ H_1 \circ \sigma = H_2$, then $H_1 \equiv H_2$.
\end{theorem}

\begin{proof}
We let $H_k=(f^k_1,f^k_2,g^k)$ for $k=1,2$ from the hypothesis satisfy the conditions we collected in \autoref{remark:SummaryJetNormalizedMapping}. We write $s_k\coloneqq 2 |f^k_{1 w^2}(0)| \geq 0, x_k \coloneqq f^k_{2 w^2}(0) \in \C$ and $y_k\coloneqq \im\bigl(f^k_{2 z^2 w}(0)\bigr) \in \R$. Further we set $F\coloneqq \sigma' \circ H_1 \circ \sigma - H_2$ and write $F=(f,g) = (f_1,f_2,g)$. By \autoref{cor:jetDetermination} we only need to consider components of $j_0 F$. We let $(\sigma,\sigma') \in \Isotropies$ with the notation from \eqref{eq:sigma} and \eqref{eq:tau} respectively. The coefficients of order $1$ of $F$, which are $f_z(0)$ and $F_w(0)$, are required to vanish and the corresponding equation is given as follows:
\begin{align}
\label{eq:freenessFirstOrderZ}
U'~ {^t(u \lambda \lambda',0)} & = (1,0),\\
\label{eq:freenessFirstOrderW}
U' ~{^t(u c + \lambda c_1', \lambda c_2',\lambda \lambda')} & = (0,0,1).
\end{align}
These equations imply $\lambda' = 1/\lambda, a_2' = c_2' = 0, a_1' = 1/(u u')$ and $c_1' = - u c/\lambda$. Assuming these standard parameters we consider the coefficients of order $2$, which are $f_{z^2}(0), F_{zw}(0)$ and $F_{w^2}(0)$, to obtain:
\begin{align}
\label{eq:freenessSecondOrderZ2}
(0,2 u' u^3 \lambda) & = (0,2),\\
\label{eq:freenessSecondOrderZW}
\left(-r - \lambda^2 r' + \frac{\imu \eps \lambda^2} 2, 2 u' u^3 \lambda c,0\right) & =\left(\frac{\imu \eps}2,0,0\right),\\
\label{eq:freenessSecondOrderW2}
\left(\lambda^2 (\lambda s_1 + \imu \eps u c)/u,u u' \lambda (\lambda^2 x_1 + 2 u^2 c^2), -2 (r + \lambda^2 r')\right) & = (s_2,x_2,0).
\end{align}
The second component of \eqref{eq:freenessSecondOrderZW} implies $c=0$. Assuming this value for $c$ we obtain for the third order terms $f_{z^2 w}(0)$ the following equation:
\begin{align}
\label{eq:freenessThirdOrderZ2W}
\left(4 \imu u \lambda^3 s_1, u' u^3  \lambda (-4 r - 2 \lambda^2 r' + \imu \lambda^2 y_1) \right) & = (4 \imu s_2, \imu y_2).
\end{align}
The second component of \eqref{eq:freenessSecondOrderZ2} shows $\lambda=1$. Furthermore the third component of \eqref{eq:freenessSecondOrderW2} shows $r'= - r$ and since from the second component of \eqref{eq:freenessSecondOrderZ2} we get $u' u^3 = 1$, we obtain from the second component of \eqref{eq:freenessThirdOrderZ2W} that $r=0$. The equation from $f_{2 z^2}(0)$ given by $u' u^3 = 1$ uniquely determines $u'$. The remaining equation from the first component of \eqref{eq:freenessSecondOrderW2}, which comes from the coefficient $f_{1 w^2}(0)$, is $s_1/u = s_2$. We have to consider two cases:\\
If $s_2>0$, then $s_1 >0$, which implies $u =1$ is the only possibility. This gives $\sigma= \id_{\C^2}$ and $\sigma'= \id_{\C^3}$ and hence $H_1 \equiv H_2$.\\
If $s_2= 0$, then also $s_1 = 0$ and from the equations \eqref{eq:freenessFirstOrderZ}--\eqref{eq:freenessThirdOrderZ2W} we obtain that $j_0(H_1) = j_0(H_2)$ such that \autoref{cor:jetDetermination} implies $H_1 \equiv H_2$.
\end{proof}

\section{Equivalence of Mappings in \texorpdfstring{$\NTwo$}{N}}
\label{sec:global}

In this section we prove \autoref{theorem:ReductionFinite}. For this purpose we compose the mappings $\MapParticularValueTwo{k}{\eps}$ with translations depending on a parameter $p$ to obtain mappings denoted by $\MapTranslationParticularValueTwo{k}{\eps}$. These mappings are in general not elements of $\NTwo$, hence we have to renormalize to obtain maps denoted by $\MapRenormalizationTranslationParticularValueTwo{k}{\eps}$. By \autoref{theorem:ReductionOneParameterFamilies2} the mappings $\MapRenormalizationTranslationParticularValueTwo{k}{\eps}$ are necessarily maps as $\MapOneParameterFamilyThree{k}{s}{\eps}$, with the difference that $s=s(p)$ depends on the parameter $p$, which suffices to show that $\FTwo$ consists of finitely many orbits.

By \autoref{rem:FormOfAutomorphism} it is easy to see that the following relation is an equivalence relation.

\begin{definition}[label=def:transitiveEquiv]
For $k=1,2$ let $H_k:U_k \rightarrow \C^3$ be a holomorphic mapping where $U_k$ is an open and connected neighborhood of $p_k \in \HeisenbergOne{2}$ and $H_k(U_k \cap \HeisenbergOne{2}) \subset \HeisenbergTwo{\eps}{3}$. We say that $H_1$ and $H_2$ are \textit{(transitively) equivalent} if there exist isotropies $(\sigma,\sigma') \in \Isotropies$ and points $(p,p') \in \HeisenbergOne{2} \times \HeisenbergTwo{\eps}{3}$ such that $H_2 = \sigma' \circ t'_{p'} \circ H_1 \circ t_p \circ \sigma$. The set of all mappings, which are transitively equivalent to a given mapping $H$, is called the \textit{(transitive) orbit} of $H$ and is denoted by $O(H)$.
\end{definition}

\begin{remark}
If both mappings $H_1, H_2$ fix $0$, then necessarily $p' = H_1(p)$ in \autoref{def:transitiveEquiv} above, which is the same construction as in \cite[Section 4]{huang1}. If $p=0$ then transitive equivalence coincides with isotropic equivalence. We define for suitable $p \in \HeisenbergOne{2}$ the map $H_p \coloneqq t'_{H(p)} \circ H \circ t_p = (f_{1,p},f_{2,p}, g_p)$.
\end{remark}

\subsection{Mappings at Points of Different (Non-)Degeneracy}
\label{section:higherNonDegeneracy}

\begin{definition}
For $U \subset \C^2$ a neighborhood of $0$ let $H: U \rightarrow \C^3$ be a holomorphic mapping given by $H=P/Q$, where $P,Q$ are polynomials such that $P(0)=0$ and $Q(0)\neq 0$ with components $H=(f_1,f_2,g)$. We define $D_H  \coloneqq \{ p \in \HeisenbergOne{2}: Q(p) = 0\}$ and $A_H \coloneqq \HeisenbergOne{2} \setminus D_H$. The set $N_H$ of points $p \in \HeisenbergOne{2}$, such that $H_p$ does not satisfy the first condition in \eqref{eq:F2Cond} is given by 
\begin{align*}
N_H \coloneqq\big\{p \in A_H: f_{1,p z}(0) f_{2,p z^2}(0) - f_{2,p z}(0) f_{1,p z^2}(0) = 0\big\}.
\end{align*}
The set $T_H$ of points $p \in \HeisenbergOne{2}$, where $H_p$ does not satisfy the second condition of \eqref{eq:F2Cond} is given by
\begin{align*}
T _H\coloneqq \big\{p \in A_H: g_{p w}(0)= 0\big\}.
\end{align*}
\end{definition}

\begin{remark}
It is easy to show that $N_H = N$ from \autoref{remark:NotTwoNondegeneratePoints} and $T_H$ is the set of points in $\HeisenbergOne{2}$ where $H$ is not transversal according to \autoref{rem:NonTransversality}.
\end{remark}

For a mapping $H \in \FTwo$ we have the following possibilities: Either there exists $p\in \HeisenbergOne{2}$ such that $H$ is not $2$-nondegenerate at $p$ or $H$ is $2$-nondegenerate everywhere in the domain of $H$.\\
In the first case we consider $H_p$ and in the second case we take $H$ and compose the map with isotropies fixing $0$. Then in both cases we normalize with respect to some different normal form than the one we introduced in \autoref{proposition:NormalForm2Nondeg}.\\ 
The following example gives a mapping, which is $2$-nondegenerate everywhere in its domain:

\begin{example}[label=ex:G1plusEverywhereNonDeg]
We consider the mapping $H \coloneqq \MapOneParameterFamilyTwo{1}{+}$ such that $A_H = \HeisenbergOne{2} \setminus \{(0, \pm 1)\}$. Then we need to compute $N_H$, which is given by the following equation if we write $p = (r e^{\imu \theta}, v + \imu r^2), r \geq 0, \theta, v \in \R$:
\begin{align*}
\frac{(1+\imu v + r^2)(1+v - \imu r^2)(1- v + \imu r^2)}{(1+(r^2- \imu v)^2)^3} = 0,
\end{align*}
which admits no solution in $A_H$. 
\end{example}

In the following paragraphs we deduce some mappings of different (non-)degeneracy at $0$.

\begin{example}[label=example:GThreeMinus2]
For $H \coloneqq \MapOneParameterFamilyThree{2}{\frac {\sqrt{5}} 4}{-}$ we find $p=\left(\frac{376 + 32 \imu}{89 \sqrt{5}}, - \frac{512 + 320 \imu}{89}\right) \in N_H \cap T_H$. We apply isotropies fixing $0$ to $H_p$ and denote the resulting mapping by $G=(f,g)=(f_1,f_2,g)$. We normalize the mapping according to the following conditions:
\begin{multicols}{3}
\begin{compactitem}
\item[\rm{(i)}] $f_z(0) =(1,1)$
\item[\rm{(ii)}] $f_w(0) = (0,0)$
\item[\rm{(iii)}] $f_{z^2}(0) = (0,0)$
\item[\rm{(iv)}] $f_{1 z w}(0) = 0$
\item[\rm{(v)}] $g_{w^2}(0) = 2$
\item[\rm{(vi)}] $f_{1 z w^2}(0) = 0$
\end{compactitem}
\end{multicols}
when we use the following standard parameters:
\begin{align*}
& c = -\frac{2 + 199 \imu}{2848 \sqrt{5}}, \quad r = \frac{1223}{2048}, \quad c_1' =  \frac{1276 - 3243 \imu} {22304 \sqrt{5}}, \quad c_2' = \frac{11484 + 29187 \imu } {55760},\\
& a_1' = \frac{30613535492 - 20104041651 \imu} {353339968 \sqrt{3485}}, \quad a_2' = -\frac{11384417567 - 3593306283 \imu} {353339968 \sqrt{697}},  \quad \lambda' =  \frac{32 \sqrt{697}}{89},\\
&u' = \frac{538504992958 + 544496189479 \imu} {342480284921 \sqrt{5}}, \quad  r' = -\frac{756545275}{32444416},
\end{align*}
and the remaining standard parameters are chosen trivially. The resulting mapping is of the form $(z,w) \mapsto \bigl(z, \frac z {1+w}, \frac{w^2} {1+w} \bigr)$. This mapping is (1,1)-degenerate and not transversal at $0$. If we apply translations $t_q$ and $t'_{p'}$, where $q=(0,-1)$ and $p'=(0,0,-2)$, to the above map we obtain 
\begin{align}
\label{eq:SimpleG21Minus}
\SimpleMapParticularValueTwo{3}{-}: (z,w) \mapsto \left(z, \frac z w, \frac{1+w^2} w \right).
\end{align}
\end{example}

\begin{example}[label=example:GFourMinus]
The map $H \coloneqq \MapParticularValueTwo{4}{-}$ has $p=\bigl(\frac 4 {\sqrt{3}}, \frac {16 \imu} 3\bigr) \in N_H$. First we scale $H_p$ via dilations given by $(z,w)\mapsto (\sqrt{3} z, 3 w)$ and $(z_1',z_2',w') \mapsto \bigl(\frac{11 z_1'} {27},\frac{11 z_2'} {27} ,\frac{121 w'} {729}\bigr)$ and then we compose the resulting mapping with isotropies fixing $0$ to obtain a map denoted by $G=(f,g)=(f_1,f_2,g)$. We impose the following normalization conditions:
\begin{multicols}{3}
\begin{compactitem}
\item[\rm{(i)}] $f_z(0) =(0,\sqrt{3})$
\item[\rm{(ii)}] $G_w(0) = (0,0,-3)$
\item[\rm{(iii)}] $f_{2 z w}(0) = 0$
\item[\rm{(iv)}] $g_{w^2}(0) = 0$
\item[\rm{(v)}] $f_{1 z^2w}(0) = 0$
\item[\rm{(vi)}] $f_{1 z^3}(0) = 24$
\end{compactitem}
\end{multicols}
which are achieved if we take the following standard parameters:
\begin{align*}
& c = \imu, \quad \lambda = \frac 8 3, \quad a'_1 = \frac{14} {11}, \quad a_2'= - \frac{5 \sqrt{3}} {11}, \quad \lambda' = \frac {3 \sqrt{3}} 8, \quad  c_1' = \frac {3 \sqrt{3}\imu} 8,
\end{align*}
and the remaining parameters are chosen trivially. The resulting mapping is given by $(z,w) \mapsto \frac{(4 z^3 , \sqrt{3}(1- w^2)z, -(3 + w^2)w)}{1+ 3 w^2}$, which is $3$-nondegenerate and transversal at $0$.
\end{example}

\begin{example}[label=example:GFourPlus]
Here we consider $H \coloneqq \MapParticularValueTwo{4}{+}$ to obtain $H_p$ for $p=(4 (1+2 \sqrt{2}/3)^{1/2},16 (1+2 \sqrt{2}/3) \imu)$ is not $2$-nondegenerate at $0$. First we compose $H_p$ with the following dilations
\begin{align*}
(z,w) & \mapsto \Bigl((9 + 6 \sqrt{2})^{\frac 1 2} z, 3 (3 + 2\sqrt{2}) w \Bigr), \\
(z_1',z_2',w') & \mapsto \left( \frac{(23 + 20 \sqrt{2}) z_1'}{27},  \frac{(23 + 20 \sqrt{2}) z_2'}{27},  \frac{(1329 + 920 \sqrt{2}) w'}{729} \right),
\end{align*}
to remove some common factors.
Then we compose the resulting mapping with isotropies fixing $0$ and denote this mapping by $G$. Next we consider the same normalization conditions as in \autoref{example:GFourMinus} except we require $G_w(0) = (0,0,3)$ and use the following nontrivial standard parameters:
\begin{align*}
& c = (1+ \sqrt{2}) \imu, \quad \lambda = \frac { 8 \sqrt{2}} 3, \quad u' = -1, \quad a'_1 = \frac{(75 \sqrt{2} - 154 )} {271}, \quad a_2'= -\frac{5 \sqrt{3}(11 + 14 \sqrt{2}) } {271},\\
& \lambda' = \frac 3 8 \sqrt{\frac{51} 2 - 18 \sqrt 2}, \quad c_1' = \frac {- 21\imu (3987 - 2760 \sqrt{2})^{\frac 1 2}} {1084 \sqrt{2}}, \quad c_2' = \frac{45 \imu(40-23 \sqrt{2})} {4336}.
\end{align*}
The resulting mapping is given by $(z,w) \mapsto \frac{(4 z^3 , \sqrt{3}(1+ w^2)z, (3 - w^2)w)}{1- 3 w^2}$, which is $3$-nondegenerate and transversal at $0$.
\end{example}

\begin{remark}[label=remark:FoundFaransMapHeisenberg]
\autoref{example:GFourMinus} and \autoref{example:GFourPlus} show that the mapping 
\begin{align}
\label{eq:SimpleG3eps0}
\SimpleMapParticularValueTwo{4}{\eps}: (z,w) \mapsto \frac{(4 z^3 ,\sqrt{3}(1+ \eps w^2) z, (3 \eps - w^2)w)}{1- 3\eps w^2},
\end{align}
is equivalent to $\MapParticularValueTwo{4}{\eps}$.
\end{remark}

\begin{example}[label=example:G2Minus]
We prove that $H\coloneqq\MapOneParameterFamilyThree{2}{s}{-}$ with $0\leq s  \leq 1/2$ admits no points in $A_H$, where $H$ fails to be $2$-nondegenerate. If we write $p =(r e^{\imu \theta}, v + \imu r^2) \in A_H$ the set $N_H$ is given by 
\begin{align*}
N_H=\{p\in A_H: -4 r  s + e^{\imu \theta} \bigl(4 + (1 - 3 s^2) (r^2  + \imu v) \bigr) = 0,  0 \leq s \leq 1/2\},
\end{align*}
which is the empty set as can be observed easily. We compose $\MapParticularValueTwo{2}{-}$ with isotropies fixing $0$ to obtain a mapping denoted by $G=(f_1,f_2,g)$. We consider the following normalization conditions:
\begin{multicols}{3}
\begin{compactitem}
\item[\rm{(i)}] $f_z(0) =(1,0)$
\item[\rm{(ii)}] $G_w(0) = (0,0,1)$
\item[\rm{(iii)}] $f_{2 z^2 }(0) = 2 \sqrt{2}$
\item[\rm{(iv)}] $f_{1 w^2}(0) = 0$
\item[\rm{(v)}] $g_{w^2}(0) = 0$
\item[\rm{(vi)}] $\re(f_{1 z^2 w}(0)) = 0$
\end{compactitem}
\end{multicols}
These conditions are achieved with the following nontrivial standard parameters $c =-\frac {\imu} {2 \sqrt{2}}, \lambda = \sqrt{2}, \lambda' = \frac {1} {\sqrt{2}}, c_1' = \frac 1 {\sqrt{2}}, c_2' = \frac{\imu} {4}$,
which results in a mapping given by
\begin{align}
\label{eq:SimpleG2Minus12}
\SimpleMapParticularValueTwo{2}{-}: (z,w) \mapsto \left(\frac{z (1 + \sqrt{2} z - \imu w)}{1 + \sqrt{2} z}, \frac{z (\sqrt{2} z - \imu w)}{1 + \sqrt{2} z}, w  \right).
\end{align}
\end{example}

\subsection{Renormalization}
\label{section:renormalization}

\begin{definition}[label=def:newSimpleMaps]
We define and recall
\begin{align*}
\SimpleMapParticularValueTwo{1}{\eps}(z,w) & \coloneqq \left(\frac{z (1 + \imu \eps w)}{1  - \imu \eps w}, \frac{2 z^2}{1 - \imu \eps w}, w \right),\quad
\SimpleMapParticularValueTwo{2}{-}(z,w) \coloneqq \left(\frac{z (1 + \sqrt{2} z - \imu w)}{1 + \sqrt{2} z}, \frac{z (\sqrt{2} z - \imu w)}{1 + \sqrt{2} z}, w  \right), \\
\SimpleMapParticularValueTwo{3}{-}(z,w) & \coloneqq \left(z , \frac z w, \frac{1+w^2} w\right),\qquad \qquad \qquad 
\SimpleMapParticularValueTwo{4}{\eps}(z,w)  \coloneqq \frac{(4 z^3 ,\sqrt{3}(1+ \eps w^2) z, (3 \eps - w^2)w)}{1- 3\eps w^2}.
\end{align*}
\end{definition}

The mapping $\SimpleMapParticularValueTwo{1}{\eps}$ is isotropically equivalent to $\MapOneParameterFamilyThree{2}{0}{\eps}$ by scaling with the isotropies $ (z,w) \mapsto(2 z, 4 w)$ and $(z_1',z_2',w') \mapsto (z_1'/2,z_2'/2,w'/4)$. The map $\SimpleMapParticularValueTwo{2}{-}$ is the one from \eqref{eq:SimpleG2Minus12}, $\SimpleMapParticularValueTwo{3}{-}$ is the map \eqref{eq:SimpleG21Minus} and $\SimpleMapParticularValueTwo{4}{\eps}$ is taken from \eqref{eq:SimpleG3eps0}.

\begin{proof}[Proof of \autoref{theorem:ReductionFinite}]
We define the following sets for $\eps = +1$:
\begin{align*}
\GeneralUpDown{\mathcal W}{1}{+} \coloneqq & ~\left\{(r_0, \imu r_0^2): 0 < r_0 < 1 \right\},\qquad \qquad
\GeneralUpDown{\mathcal W}{4}{+} \coloneqq  ~ \left\{(r_0, \imu r_0^2):  0< r_0 < -1 + \sqrt{2}  \right\},
\end{align*}
and for $\eps=-1$:
\begin{align*}
\GeneralUpDown{\mathcal W}{1}{-} \coloneqq & ~\left\{(r_0, \imu r_0^2): 0 < r_0 <  -1+\sqrt{2} \right\}, \qquad \qquad
\GeneralUpDown{\mathcal W}{3}{-} \coloneqq  ~ \left\{(r_0, \imu r_0^2):   r_0 > 1\right\}\\
\GeneralUpDown{\mathcal W}{4}{-} \coloneqq & ~\left\{(r_0, \imu r_0^2):  r_0 > 1 , r_0 \neq \sqrt{2+\sqrt{3}}\right\},
\end{align*}
and write
\begin{align*}
\GeneralUpDown{\mathcal W}{4,1}{-} \coloneqq  ~\left\{(r_0, \imu r_0^2):  1 < r_0 < \sqrt{2+\sqrt{3}}\right\}, \qquad 
\GeneralUpDown{\mathcal W}{4,2}{-} \coloneqq  ~\left\{(r_0, \imu r_0^2):  r_0 > \sqrt{2+\sqrt{3}}\right\}.
\end{align*}
For $k=1,3$ and $k=4$ if $\eps = +1$ we let  $p_0 =(r_0,\imu r_0^2)\in \GeneralUpDown{\mathcal W}{k}{\eps} \subset \HeisenbergOne{2}$ such that $\SimpleMapTranslationParticularValueTwo{k}{\eps} \in \FTwo$. In the case where $k=4$ and $\eps = -1$ we have that $g_{w,p_0}(0) <0$ in $\SimpleMapTranslationParticularValueTwo{4}{-}$ for $p_0 \in \GeneralUpDown{\mathcal W}{4,2}{-}$. Here we apply the automorphism $\pi'$ from \eqref{eq:Pi} to $\SimpleMapTranslationParticularValueTwo{4}{-}$. We consider $(\sigma,\sigma') \in \Isotropies$ according to \autoref{proposition:NormalForm2Nondeg} and define
\begin{align*}
H_k \coloneqq \sigma' \circ t'_{ \SimpleMapParticularValueTwo{k}{\eps}(p_0)} \circ \SimpleMapParticularValueTwo{k}{\eps} \circ t_{p_0} \circ \sigma.
\end{align*}
In \hyperref[appendix:StandardParameters]{Appendix D} we list the standard parameters such that $H_k \in \NTwo$. The notation we use for the standard parameters is from \eqref{eq:sigma} and \eqref{eq:sigma} and we write $t_k$ for a standard parameter $t$ occurring in $H_k$. If $k=4$ and $\eps = -1$ we normalize the mapping $\pi' \circ \SimpleMapTranslationParticularValueTwo{4}{-}$ for $p_0 \in \GeneralUpDown{\mathcal W}{4,2}{-}$ analogously to the case $p_0 \in \GeneralUpDown{\mathcal W}{4,1}{-}$. The standard parameters can be taken as in the case when $g_{w,p_0}(0) > 0$, except $a_1'$ and $a_2'$ as well as $c_1'$ and $c_2'$ have to be interchanged. \\
We would like to show $H_k  \equiv \MapOneParameterFamilyThree{\ell_k}{s_k^{\eps}(p_0)}{\eps}$, where $\ell_1=\ell_3=2, \ell_4=3$ and $s_k^{\eps}(p_0): \GeneralUpDown{\mathcal W}{k}{\eps} \rightarrow \R$ is given as follows:
\begin{align*}
s_1^{\eps}(p_0) =  \frac{2 r_0 (1+ \eps r_0^2)}{(1- \eps r_0^2)^2},\qquad \quad
s_3^{-}(p_0) =  \frac{1+ r_0^4} {4 r_0^2},\qquad \quad
s_4^{\eps}(p_0) =   \frac{\eps - 33 r_0^4 + 33 \eps r_0^8 + r_0^{12}}{12 \sqrt{3}r_0^2 (-\eps + r_0^4)^2}.
\end{align*}
By \autoref{cor:jetDetermination} it suffices to show  
\begin{align}
\label{eq:JetDetRenormal}
j_0\left(H_k \right) = j_0 \left(\MapOneParameterFamilyThree{\ell_k}{s_k^{\eps}(p_0)}{\eps} \right).
\end{align}
We set $F_k \coloneqq H_k - \MapOneParameterFamilyThree{\ell_k}{s_k^{\eps}(p_0)}{\eps}$ and write $F_k =(f^k,g^k)$ for the components. Since $H_k \in \NTwo$ we only need to consider the coefficients $f^k_{w^2}(0)$ and $f^k_{z^2 w}(0)$ in $F_k$. We denote $\MapOneParameterFamilyThree{k}{s}{\eps} =  \bigl(f_{k,s}^{\eps},g_{k,s}^{\eps}\bigr)$. The right-hand side of \eqref{eq:JetDetRenormal} contains the following coefficients, where we set $s = s_k^{\eps}(p_0)$:
\begin{align}
\label{eq:coeffG2}
f^{\eps}_{2,s w^2}(0) = & \left(\frac s 2, \frac {s^2} 2 \right), \qquad f^{\eps}_{2,s z^2 w}(0) =  \left(2 \imu s, \frac 1 2 \left(\eps + s^2\right)\right),\\
\label{eq:coeffG3}
f^{\eps}_{3,s w^2}(0) = & \left(\frac s 2, -\frac {\eps} 8 \right), \qquad f^{\eps}_{3,s z^2 w}(0) = \left(2 \imu s, \frac {3 \imu \eps} 8\right).
\end{align}
Then we compute for $H_k=(h^k,i^k)=(h_1^k,h_2^k,i^k))$ the coefficients $h^k_{w^2}(0)$ and $h^k_{z^2 w}(0)$. We obtain that $h^1_{1w^2}(0) = s_1^{\eps}(p_0)/2$, $h^3_{1w^2}(0) = s_3^{-}(p_0)/2$ and $h^4_{1w^2}(0) = s_4^{\eps}(p_0)/2$. Further straightforward computations show that the remaining coefficients $h^k_{2 w^2}(0)$ and $h^k_{z^2 w}(0)$ are of the form as given in \eqref{eq:coeffG2} and \eqref{eq:coeffG3}, where $s$ is replaced by the appropriate $s_k^{\eps}(p_0)$. This implies $H_k  \equiv \MapOneParameterFamilyThree{\ell_k}{s_k^{\eps}(p_0)}{\eps}$ for $k=1,3,4$ as claimed above.\\
Next we analyze to which mappings $H_k$ is equivalent. By \eqref{eq:JetDetRenormal} it suffices to study the set $s_k^{\eps}\bigl(\GeneralUpDown{\mathcal W}{k}{\eps}\bigr)$. For $\eps = +1$ the function $s_1^+$ is strictly increasing in $\GeneralUpDown{\mathcal W}{1}{+}$ and $s_4^+$ is strictly decreasing in $\GeneralUpDown{\mathcal W}{4}{+}$. Further we have $s_1^+(\GeneralUpDown{\mathcal W}{1}{+}) = s_4^+(\GeneralUpDown{\mathcal W}{4}{+})= \R^+$. This proves (i) and (ii) of \autoref{theorem:ReductionFinite}.\\
Then for $\eps = -1$ it holds that $s_k^-$ is strictly increasing in $\GeneralUpDown{\mathcal W}{k}{-}$ for $k=1,3,4$. Moreover we have $0<s_1^-(\GeneralUpDown{\mathcal W}{1}{-})<1/2, s_3^-(\GeneralUpDown{\mathcal W}{3}{-})>1/2$ and $s_4^-(\GeneralUpDown{\mathcal W}{4}{-}) = \R^+\setminus\{1/2\}$. More precisely we have  $s_4^-(\GeneralUpDown{\mathcal W}{4,1}{-}) = \{x \in \R: 0< x< 1/2\}$ and $s_4^-(\GeneralUpDown{\mathcal W}{4,2}{-}) = \{x \in \R: x > 1/2\}$. This proves the remaining claims in \autoref{theorem:ReductionFinite}.
 \end{proof}

\section{Proof of the Classification}
\label{sec:ClassificationFinale}

\begin{proof}[Proof of \autoref{theorem:MainTheorem}]
Let $U \subset \C^2$ be an open, connected and sufficiently small neighborhood of $p \in \Sphere{2}$ and $H: U \rightarrow \C^3$ a holomorphic mapping satisfying $H(U \cap \Sphere{2}) \subset \Hyperquadric{\eps}{3}$. We change coordinates as described in \autoref{sec:setup} and  define $S_1(H) \coloneqq T_3 \circ H \circ T_2^{-1}$, where we use the biholomorphisms $T_3$ and $T_2^{-1}$ from \eqref{eq:Cayley} and \eqref{eq:CayleyInv} respectively. We obtain a holomorphic mapping $S_1(H): U \rightarrow \C^3$, which satisfies $S_1(H)(0) = 0$ and maps $V \cap \HeisenbergOne{2}$ to $\HeisenbergTwo{\eps}{3}$, where $V$ is a sufficiently small open and connected neighborhood of $0$. By \Autoref{proposition:FirstProperties}, $S_1(H)$ is either $\MapMainTheoremTwo{1}{\eps}$ or $\MapMainTheoremOne{7}$, after changing coordinates to obtain the corresponding mappings from $\Sphere{2}$ to $\Hyperquadric{\eps}{3}$, or belongs to $\FTwo$.\\
 We define $S_2(H) \coloneqq \sigma_1' \circ H \circ \sigma_1$, where $(\sigma_1,\sigma_1') \in \Isotropies$. If $S_1(H)\in \FTwo$ we consider $S_2(S_1(H))$ and choose appropriate isotropies such that $S_2(S_1(H)) \in \NTwo$ according to \Autoref{proposition:NormalForm2Nondeg}. From \Autoref{theorem:ReductionOneParameterFamilies2} we obtain that $S_2(S_1(H)) \in \{\MapOneParameterFamilyTwo{1}{\eps},\MapOneParameterFamilyThree{2}{s}{\eps}, \MapOneParameterFamilyThree{3}{s}{\eps}\}$. Next we apply \autoref{theorem:ReductionFinite}. We obtain that if $\eps = +1$ the class $\FTwo$ consists of at most $3$ orbits given by the mappings $\MapOneParameterFamilyTwo{1}{+}, \MapParticularValueTwo{1}{+}$ and $\MapParticularValueTwo{4}{+}$ and if $\eps = - 1$ we have at most $5$ orbits in $\FTwo$ given by the mappings $\MapOneParameterFamilyTwo{1}{-}, \MapParticularValueTwo{1}{-},\MapParticularValueTwo{2}{-},\MapParticularValueTwo{3}{-}$ and $\MapParticularValueTwo{4}{-}$.\\
 We show how to deduce the mappings listed in \autoref{theorem:MainTheorem} from the above list of maps from \autoref{theorem:ReductionOneParameterFamilies2} and \autoref{theorem:ReductionFinite}. We introduce $S_3(H) \coloneqq \sigma'_2  \circ \Bigl(t'_{H(p)}  \circ H \circ t_{p}\Bigr) \circ \sigma_2$, where $p\in \HeisenbergOne{2}$ and $(\sigma_2, \sigma_2') \in \Isotropies$. The map $S_1^{-1}(\MapOneParameterFamilyTwo{1}{\eps})$ is equivalent to $\MapMainTheoremTwo{2}{\eps}$, since composing $\MapOneParameterFamilyTwo{1}{\eps}$ with dilations $(z,w) \mapsto (\sqrt{2} z, 2 w)$ and then applying $S_1^{-1}$ results in the mapping $\MapMainTheoremTwo{2}{\eps}$. It holds that $S_1^{-1}(\SimpleMapParticularValueTwo{1}{\eps}) = \MapMainTheoremTwo{3}{\eps}$, where $\SimpleMapParticularValueTwo{1}{\eps}$ from \autoref{def:newSimpleMaps} is equivalent to $\MapParticularValueTwo{1}{\eps}$. We have $S_1^{-1}(\SimpleMapParticularValueTwo{2}{-}) = \MapMainTheoremOne{5}$, where $\SimpleMapParticularValueTwo{2}{-}$ is equivalent to $\MapParticularValueTwo{2}{-}$, see \autoref{example:G2Minus}. \autoref{example:GThreeMinus2} shows that $S_1^{-1}(\SimpleMapParticularValueTwo{3}{-})$, where $\SimpleMapParticularValueTwo{3}{-}$ is equivalent to $\MapParticularValueTwo{3}{-}$, is equivalent to $\MapMainTheoremOne{6}$. We apply the isotropy $(z_1',z_2',w') \mapsto (z_1'/2,\imu z_2'/2,w'/4)$ and then $S_1^{-1}$ to the map in $\SimpleMapParticularValueTwo{3}{-}$ to obtain $\MapMainTheoremOne{6}$. Finally we have $S_1^{-1}(\SimpleMapParticularValueTwo{4}{\eps}) = \MapMainTheoremTwo{4}{\eps}$, where $\SimpleMapParticularValueTwo{4}{\eps}$ is equivalent to $\MapParticularValueTwo{4}{\eps}$ as we noted in \autoref{remark:FoundFaransMapHeisenberg}.\\ 
Next we show that the mappings listed in \autoref{theorem:MainTheorem} are not equivalent to each other. The following lemma whose proof is easy is stated in \cite[Lemma 2.1]{meylan}.

\begin{lemma}[label=lem:degInvariant]
Let $H \in \FTwo$ and $(\phi,\phi') \in \Aut(\HeisenbergOne{2},0) \times \Aut(\HeisenbergTwo{\eps}{3},0)$, then $\widetilde H \coloneqq \phi' \circ H \circ \phi$ satisfies $\deg \widetilde H =  \deg H$.
\end{lemma}

We need the following lemma to treat the mapping $\SimpleMapParticularValueTwo{2}{-}$.

\begin{lemma}[label=lemma:FixedMapping]
We set $G \coloneqq \SimpleMapParticularValueTwo{2}{-}$ and let $H \in \FTwo$ and $(\phi,\phi') \in \Aut(\HeisenbergOne{2},0) \times \Aut(\HeisenbergTwo{\eps}{3},0)$ such that $H = \phi' \circ G \circ \phi$. Then $G$ is isotropically equivalent to $H$.
\end{lemma}

\begin{proof}
By \autoref{rem:FormOfAutomorphism} we write $(\phi,\phi') = (t_p \circ \sigma,\sigma' \circ t_{p'}')$ for $(p,p') \in \HeisenbergOne{2} \times  \HeisenbergTwo{\eps}{3}$ and $(\sigma,\sigma') \in \Isotropies$. Since $H \in \mathcal F$ we conclude that the map $t'_{p'} \circ G \circ t_p$ fixes $0$ and satisfies the conditions in \eqref{eq:F2Cond}, which implies $p'=G(p)$ and hence $G_p \in \FTwo$. Thus we have $H = \sigma' \circ t'_{G(p)} \circ G \circ t_p \circ \sigma$. Since $H \in \FTwo$ we write $H = \phi' \circ \widetilde H \circ \phi$ where $\widetilde H \in \NTwo$ and $(\phi,\phi') \in \Isotropies$. Then we consider 
\begin{align}
\label{eq:G12MinusOrbit}
\widetilde H = \psi' \circ t'_{G(p)} \circ G \circ t_p \circ \psi,
\end{align}
where $(\psi,\psi') = (\sigma \circ \phi^{-1},\phi'^{-1} \circ \sigma')\in \Isotropies$. We want to conclude that $\widetilde H \equiv G$, which proves our claim. We investigate for which $p \in \HeisenbergOne{2}$ we have $G_p \in \FTwo$. We compute the following sets for $G$, where we write $p=(r e^{\imu\theta},v + \imu r^2) \in \HeisenbergOne{2}$ with $r\geq 0$ and $\theta,v \in \R$:
\begin{align*}
D_G =   \left\{p \in \HeisenbergOne{2}: 1+  \sqrt{2} r e^{\imu \theta} = 0 \right\}, \quad N_G = \emptyset, \quad T_G  =  \left\{p \in A_G: e^{\imu \theta} + \sqrt{2} r (1+ e^{2 \imu \theta}) = 0\right\}.
\end{align*}
The triviality of $N_G$ follows from the considerations in \autoref{example:G2Minus}. Now we let the parameters $p$ in $G_p$ be arbitrary in $A_G \setminus T_G$ such that we can normalize $G_p$ according to the normalization conditions given in \autoref{example:G2Minus}. More precisely we consider $\widetilde H$ from \eqref{eq:G12MinusOrbit} above, such that $\widetilde H$ satisfies these new normalization conditions. We list all necessary standard parameters in \hyperref[appendix:StandardParameters]{Appendix D}. Then we apply the jet determination result for $\FTwo$ and consider the coefficients of $\widetilde H$ according to \autoref{cor:jetDetermination} to see that $\widetilde H \equiv G$.
\end{proof}

Now we are ready to prove that the mappings listed in \autoref{theorem:MainTheorem} are not equivalent to each other. We start with the easy cases: $H_7$ is not equivalent to any other map of the list, since it is not immersive. Also $\MapMainTheoremTwo{1}{\eps}$ cannot be equivalent to any other map, since the map is $(1,1)$-degenerate everywhere in its domain and by \autoref{proposition:FirstProperties} the mappings $\MapMainTheoremTwo{2}{\eps}, \MapMainTheoremTwo{3}{\eps}, \MapMainTheoremTwo{4}{\eps}, H_5$ and $H_6$ have points in their domains, where they are $2$-nondegenerate. By \autoref{lem:degInvariant} we observe that $\MapMainTheoremTwo{4}{\eps}$ is not equivalent to any other map in the list. It remains to distinguish mappings of degree $2$.\\
First we treat the case $\eps = +1$. From \autoref{ex:G1plusEverywhereNonDeg} we know that  $\MapOneParameterFamilyTwo{1}{+}$ is $2$-nondegenerate everywhere in its domain, which is equivalent to $\MapMainTheoremTwo{2}{+}$. The map $\MapMainTheoremTwo{3}{+}$ is equivalent to $\MapParticularValueTwo{1}{+}=\MapOneParameterFamilyThree{2}{0}{+}$, which has points in its domain, where the map is not $2$-nondegenerate as can be easily observed. For example the point $(2, 4 \imu) \in N_{\MapParticularValueTwo{1}{+}}$. Thus $\MapMainTheoremTwo{2}{+}$ and $\MapMainTheoremTwo{3}{+}$ are not equivalent.\\ 
Next we consider the case $\eps = -1$. First we note that \autoref{example:G2Minus} shows the maps $\MapMainTheoremTwo{3}{-}$, which is equivalent to $\MapParticularValueTwo{1}{-} = \MapOneParameterFamilyThree{2}{0}{-}$ and $H_5$, which is equivalent to $\MapParticularValueTwo{2}{-} = \MapOneParameterFamilyThree{2}{1/2}{-}$, are both $2$-nondegenerate everywhere in their domains. The maps $\MapMainTheoremTwo{2}{-}$ and $H_6$, which are equivalent to $\MapOneParameterFamilyTwo{1}{-}$ and $\MapParticularValueTwo{3}{-} = \MapOneParameterFamilyThree{2}{1}{-}$ respectively, do contain points in their domains, where the maps are not $2$-nondegenerate. More precisely, $\MapOneParameterFamilyTwo{1}{-}$ is not $2$-nondegenerate excactly at the points $p = (e^{\imu t}, \imu)$ for $t \in \R$ for which the space $E_1'(p)$ is $1$-dimensional. This implies that there are no points in the domain where the map is $(1,1)$-degenerate. We computed in \autoref{example:GThreeMinus2}, that there is a mapping which is $(1,1)$-degenerate at for example $(0,1) \in \HeisenbergOne{2}$ and is equivalent to $H_6$ by \autoref{theorem:ReductionFinite}. Thus the maps $\MapMainTheoremTwo{2}{-}$ and $H_6$ are both not equivalent to any other map of the list.\\
For a mapping $F$ we introduce the set $\stab_0(F) \coloneqq \{(\sigma, \sigma') \in \Isotropies: \sigma' \circ F \circ \sigma = F\}$ called the \textit{isotropic stabilizer of $F$}. Next we observe that if we let $H\in \NTwo$ and $F = \varphi' \circ H \circ \varphi$, where $(\varphi,\varphi') \in \Isotropies$, it is a well-known fact that
\begin{align}
\label{eq:conjIsotropicStab}
 \stab_0(F) = \left\{(\varphi^{-1} \circ \sigma \circ \varphi,\varphi' \circ \sigma' \circ {\varphi'}^{-1}) \in \Isotropies:  (\sigma,\sigma') \in \stab_0(H)\right\}.
\end{align}
It remains to distinguish $\MapMainTheoremTwo{3}{-}$ from $H_5$. On the one hand we observe that for $\MapOneParameterFamilyThree{2}{0}{-}$ and $|u|=1$ the isotropies $\bigl(\sigma(z,w), \sigma'(z_1',z_2',w')\bigr) = \bigl(u z, w,z_1'/u,z_2'/u^2,w'\bigr)$ belong to $\stab_0(\MapOneParameterFamilyThree{2}{0}{-})$. On the other hand the map $\MapOneParameterFamilyThree{2}{1/2}{-}$ has a trivial isotropic stabilizer. This can be concluded from the equations in the proof of \autoref{theorem:NonIntersectingOrbits}, since here we are in the first case, where $s_1=s_2=1/2$. In \autoref{lemma:FixedMapping} we concluded that any map in $\FTwo$ to which $\MapOneParameterFamilyThree{2}{1/2}{-}$ is equivalent must be isotropically equivalent, thus by \eqref{eq:conjIsotropicStab} has a trivial isotropic stabilizer. Hence $\MapMainTheoremTwo{3}{-}$ and $H_5$ are not equivalent. \\
It remains to prove the last statement of \autoref{theorem:MainTheorem}. We show equivalence of $S_1(L_3)$ and $\MapParticularValueTwo{1}{-}$, $S_1(L_4)$ and $\MapParticularValueTwo{2}{-}$, $S_1(L_5)$ and $\MapParticularValueTwo{3}{-}$ and finally equivalence of $S_1(L_6)$ and $\MapParticularValueTwo{4}{-}$.\\
We start by showing the first equivalence by considering $S_3(S_1(L_3))$ and defining $u' = -1, \lambda = 1/2, a'_1=-1, c_2' = \imu / 2$ and the rest of the occurring parameters trivially. Then we have $S_3(S_1( L_3)) = \MapParticularValueTwo{1}{-}$. In the case of the mapping $S_1(L_4)$ we define
\begin{align*}
&p=(2, 4 \imu), \quad c = \frac{11 \imu} 4, \quad u = -1, \quad \lambda = 3,  \quad  \lambda' = \frac{2}{3~3^{3/4}}, \\
& a'_1 = -\frac{2}{3^{1/4}} - \frac{3^{1/4}} 8, \quad  a'_2 =-\frac{2}{3^{1/4}} + \frac{3^{1/4}} 8, \quad  c_1' = -\frac{\imu (272 - 5 \sqrt{3})}{144}, \quad c_2' = \frac{\imu (272+5 \sqrt{3})}{144},
\end{align*} 
and the rest of the parameters trivially. With these choices we obtain $S_3(S_1(L_4)) = \MapParticularValueTwo{2}{-}$. Next we want to see that $S_1(L_5)$ is equivalent to $\MapParticularValueTwo{3}{-}$.  We define the following parameters for $S_3(S_1(L_5))$
\begin{align*}
& p=\Bigl(\sqrt{2},  -1 + 2 \imu \Bigr),  \quad c = \frac{4 + 3 \imu}{8 \sqrt 5},  \quad u = -\frac{1-2 \imu}{\sqrt{5}}, \quad \lambda = \frac 1 {\sqrt{2}}, \quad r = \frac 1 8, \quad r' = 3 \sqrt{2},\\
& \lambda' = 4 ~2^{1/4}, \quad  u' = -\frac {2 - 11 \imu} {5 \sqrt 5}, \quad a'_1 = \frac{-1 + 7 \imu} 5, \quad a'_2  =\frac{-4 + 3 \imu} 5, \quad c_1' = -\frac{1 - 5 \imu}{2^{3/4}}, \quad c_2' = -\frac{\imu}{2^{3/4}},
\end{align*}
and the remaining parameters we choose trivially. Then we have $S_3(S_1(L_5)) = \MapOneParameterFamilyThree{2}{\frac {\sqrt{5}} 4}{-}$, which, since $\sqrt{5}/4 > 1/2$, is equivalent to $\MapParticularValueTwo{3}{-}$ by \autoref{theorem:ReductionFinite}. Finally we consider $S_1(L_6)$ and we want to see that this mapping is equivalent to $\MapMainTheoremTwo{4}{-}$. Here we note that after the linear change of coordinates $(z,w) \mapsto(\imu w, z)$ and $(z_1',z_2',w') \mapsto (- z_1',- \imu z_2',w')$ in $\C^2$ and $\C^3$, $L_6$ is the same mapping as $\MapMainTheoremTwo{4}{-}$, which we know is equivalent to $\MapParticularValueTwo{4}{-}$. This completes the proof of \autoref{theorem:MainTheorem}. 
\end{proof}

\section*{Appendix A: Formula for Jet Parametrization}
\label{appendix:ParametrizationFormula}

In \Autoref{lemma:BasicIdentity2Nondeg} we have the following formulas: Denote $\Psi=(f_1,f_2,g)$. We order the monomials by degree and by assigning the weight $1$ to $z$ and the weight $2$ to the variable $\chi$. The numerator of $f_1(z,2 \imu z \chi)$ is the following expression:
\begin{align*}
& \quad 2 \eps z  
+ 6 A_2 z \chi 
+ \imu C_{22}  z^3 
+ 4 \imu  \eps  B_{21} z^2 \chi 
+ 6 \eps B_2 z \chi^2 
+ \Bigl(2 \eps A_2 + A_{22} - C_{13} \Bigr) z^3 \chi \\
& - 2 \Bigl(3 \imu A_3 + 3 \eps B_{12} +  A_2 (7 \eps - 3 \imu B_{21}) \Bigr) z^2 \chi^2  
+  2  A_2 B_2 z \chi^3  \\
& +  \Bigl(6 A_2^2 +  \imu A_{13} + \eps (-1 - 2 B_{21}^2+ B_{22} +  C_3) - \imu C_4 - 2 \imu B_2 C_{22}\Bigr) z^3 \chi^2 \\
& - 2  \Bigl(5 A_2^2 + 4 \imu \eps B_3 +  4 A_2 B_{12} + B_2 (6 - 2 \imu \eps B_{21})\Bigr) z^2 \chi^3 \\
& +  \Bigl(-A_4 - 2 A_{22} B_2 + B_{12} + 2 A_3 B_{21} + \imu \eps (4 A_3 + B_{13} - 4 B_{12} B_{21}) \\
    & \qquad + A_2  (5 + 4 \eps B_2 - 4 \imu \eps B_{21} - 2 B_{21}^2+ B_{22} +  3 C_3) + 2 B_2 C_{13}\Bigr) z^3 \chi^3 \\
& + 2 \imu \Bigl(B_2 (4 A_3 + \imu \eps B_{12}) + A_2 \bigl(-5 B_3 + B_2 (5 \imu \eps + B_{21})\bigr)\Bigr)  z^2 \chi^4 \\
& + \Bigl(2 \imu B_3 + 2 \imu A_3 B_{12} + \imu A_2 \bigl(4 A_3 + B_{13} + B_{12} (-6 \imu \eps - 4 B_{21})\bigr) + 2 A_2^2 (5 \eps - 2 B_2 - \imu B_{21}) \\
    & \qquad + \eps (-B_4 + 2 B_{12}^2 + 4 B_3 B_{21}) + B_2 \bigl(-2 \imu A_{13} -  6 \imu B_{21} + \eps (2 - B_{22} + 2 C_3) +  2 \imu C_4\bigr) \\
&  \qquad + \imu B_2^2C_{22} \Bigr) z^3 \chi^4- 2  A_2^2 B_2 z^2 \chi^5 \\
& + \Bigl(4 A_2^3 + 2 A_4 B_2 + A_{22} B_2^2+ 3 A_2^2 B_{12} + 5 B_2 B_{12} + 4 \imu \eps B_3 B_{12} - \imu \eps B_2 B_{13} \\
    & \qquad  -  2 A_3 \bigl(B_3 + B_2 (\imu \eps + B_{21}) \bigr) - A_2 \bigl(6 \eps B_2^2+ B_4 - 2 B_{12}^2 + B_3 (-8 \imu \eps - 4 B_{21})  \\
& \qquad + B_2 (-4 + 8 \imu \eps B_{21} + B_{22} + 2 C_3)\bigr) - B_2^2 C_{13}\Bigr) z^3 \chi^5 - 2 \imu B_2 \Bigl(A_3 B_2 - A_2 B_3\Bigr) z^2 \chi^6 \\
&+ \Bigl(-2 \eps B_3^2 + B_2 (4 \imu B_3 + \eps B_4 - 2 \imu A_3 B_{12}) - \imu A_2  \bigl(2 A_3 B_2 - 4 B_3 B_{12} + B_2(6 \imu \eps B_{12} + B_{13})\bigr) \\
     & \qquad + 2 A_2^2 \bigl(-B_2^2+ 2 \imu B_3 +  B_2 (\eps - 2 \imu B_{21})\bigr) + B_2^2(3 \eps + \imu A_{13}- 3 \eps C_3- \imu C_4)\Bigr) z^3 \chi^6  \\
& + \Bigl(B_2 \bigl(-A_4 B_2 + 2 A_3 (-\imu \eps B_2 + B_3)\bigr) + 3 A_2^2 B_2 B_{12}+ A_2 \bigl(-2 B_3^2 + B_2 (4 \imu \eps B_3 + B_4) \\
    & \qquad - B_2^2(-3 + C_3) \bigr)\Bigr) z^3 \chi^7  + 2 \imu  A_2 B_2 \Bigl(-A_3 B_2 + A_2 B_3\Bigr) z^3 \chi^8
\end{align*}
The numerator of $f_2(z,2 \imu z \chi)$ is equal to the following formula:
\begin{align*}
& \quad 2 \eps z^2 
+ 2  A_2 z^3 
+ 6 A_2 z^2 \chi 
+ \Bigl(-1 + C_3\Bigr) z^3 \chi
+ 4 \eps B_2 z^2 \chi^2  
-  \Bigl(2 \imu A_3 +  6 A_2 (\eps + B_2) + \eps B_{12}\Bigr) z^3 \chi^2 \\
& - 4 A_2 B_2 z^2 \chi^3 
+ \Bigl(-4 A_2^2 - 2 \imu \eps B_3 - A_2 B_{12} + B_2 (1 + 4 \imu \eps B_{21} - 3 C_3)\Bigr) z^3 \chi^3 
 - 6 \eps B_2^2 z^2 \chi^4\\
 & + \Bigl(2 \imu B_2 (A_3 + \imu \eps B_{12}) + A_2 \bigl(6 B_2^2- 2 \imu B_3 + 4 B_2 (\eps + \imu B_{21})\bigr)\Bigr) z^3 \chi^4
-  2 A_2 B_2^2 z^2 \chi^5 \\
& +  B_2 \Bigl(4 A_2^2 - 2 A_2 B_{12} + B_2 (-3 - 4 \imu \eps B_{21} + 3 C_3)\Bigr) z^3 \chi^5\\
& +   B_2^2\Bigl(2 \imu A_3 + 3 \eps B_{12} + 2 A_2 (\eps - B_2 - 2 \imu B_{21})\Bigr) z^3 \chi^6  + B_2^2\Bigl(2 \imu \eps B_3 + 3 A_2 B_{12}  \\
&   - B_2 (-3 + C_3)\Bigr) z^3 \chi^7 - 2 \imu B_2^2\Bigl(A_3 B_2 - A_2 B_3\Bigr) z^3 \chi^8
\end{align*}
The numerator of $g(z,2 \imu z \chi)$ is equal to the following formula:
\begin{align*}
& \quad 4 \imu \eps z \chi 
+ 12 \imu A_2 z \chi^2 
- 2 C_{22}   z^3 \chi 
+   \Bigl(4 \imu - 8 \eps B_{21}\Bigr) z^2 \chi^2
+ 12 \imu \eps B_2 z \chi^3 \\
& +  2 \imu \Bigl(4 \eps A_2 + A_{22} - C_{13}\Bigr) z^3 \chi^2 + 12  \Bigl(A_3 - \imu \eps B_{12} + A_2 (-\imu \eps - B_{21})\Bigl) z^2 \chi^3
+  4 \imu A_2 B_2 z \chi^4\\
& + 2 \Bigl(8 \imu A_2^2 - A_{13} - \imu \eps (2 + 2 B_{21}^2- B_{22} -  2 C_3) + C_4 + 2 B_2 C_{22}\Bigr) z^3 \chi^3  \\
& - 4  \Bigl(2 \imu A_2^2 - 4 \eps B_3 + 4 \imu A_2 B_{12} + B_2 (3 \imu + 2 \eps B_{21}) \Bigr) z^2 \chi^4 \\
& + 2 \Bigl(-\imu A_4 - 2 \imu A_{22} B_2 - \eps B_{13}- 2 A_3 (\eps - \imu B_{21}) +  4 \eps B_{12} B_{21} + A_2 \bigl(4 \eps B_{21} - 2 \imu 
B_{21}^2 \\
    & \qquad +  \imu (-2 + B_{22} + 4 C_3) \bigr) + 2 \imu B_2 C_{13}\Bigr) z^3 \chi^4 \\
& - 4 \Bigl(B_2 (4 A_3 + \imu \eps B_{12}) + A_2 \bigl(-5 B_3 + B_2 (\imu \eps + B_{21}) \bigr)\Bigr) z^2 \chi^5 \\
& +  2 \Bigl(-2 A_3 B_{12} - A_2 \bigl(2 A_3 + B_{13} + B_{12} (-4 \imu \eps - 4 B_{21}) \bigr) + 2 A_2^2 (-4 \imu B_2 + B_{21}) \\
	& \qquad - \imu \eps \bigl(B_4 - 2 (B_{12}^2 + 2 B_3 B_{21})\bigr) + B_2 \bigl(2 A_{13} + 2 B_{21} -  \imu \eps (-2 + B_{22}) - 2 C_4 \bigr) - B_2^2 C_{22} \Bigr) z^3 \chi^5 \\
& -  2 \imu \Bigl(-2 A_4 B_2 - A_{22} B_2^2- 2 A_2^2 B_{12} - 2 B_2 B_{12} - 4 \imu \eps B_3 B_{12} + \imu \eps B_2 B_{13} + 2 A_3 \bigl(B_3\\
	& \qquad  + B_2 (\imu \eps + B_{21})\bigr) + A_2 \bigl(4 \eps B_2^2+ B_4 - 2 B_{12}^2 + B_3 (-4 \imu \eps - 4 B_{21}) + B_2 (-2 + B_{22} + 4 C_3)\bigr) \\
& \qquad  + B_2^2 C_{13} \Bigr) z^3 \chi^6 + 4 B_2 \Bigl(A_3 B_2 - A_2 B_3\Bigr) z^2 \chi^7  \\
& -  2 \Bigl(A_{13} B_2^2+ 2 A_2^2 B_3 + 2 B_2 B_3 - 2 A_3 B_2 B_{12} - A_2 (2 A_3 B_2 -  4 B_3 B_{12} +  B_2 B_{13}) \\
	& \qquad + \imu \eps (2 B_3^2 -  B_2 B_4 + 2 B_2^2C_3) - B_2^2 C_4\Bigr) z^3 \chi^7\\
&  - 2 \imu  \Bigl(A_4 B_2^2- 2 A_3 B_2 B_3 + A_2 (2 B_3^2 - B_2 B_4)\Bigr) z^3 \chi^8
 \end{align*}
The denominator of $H$ is of the following form:
\begin{align*}
& 2 \eps 
+ 6 A_2 \chi 
+ \imu  C_{22} z^2 
+ \Bigl(2 + 4 \imu \eps B_{21}\Bigr) z \chi 
 +  6 \eps B_2 \chi^2 
+  \Bigl(6 \eps A_2 + A_{22} -  C_{13}\Bigr) z^2 \chi \\
& -  6  \Bigl(\imu A_3 + \eps B_{12} + A_2 (\eps - \imu B_{21})\Bigr) z \chi^2 
 +  2  A_2 B_2 \chi^3 
-  \imu  \eps C_{22} z^3 \chi \\
& + \Bigl(12 A_2^2 + \imu A_{13} - 2 \imu B_{21} + \eps (-3 - 2 B_{21}^2+ B_{22} +  3 C_3) - \imu C_4 - 2 \imu B_2 C_{22}\Bigr) z^2 \chi^2
 \end{align*}
 \begin{align*}
& - 2 \Bigl(2 A_2^2 + 4 \imu \eps B_3 + 4 A_2 B_{12} + B_2 (3 - 2 \imu \eps B_{21})\Bigr)  z \chi^3 
+  \Bigl(\eps (-A_{22} + C_{13}) + 2 \imu A_2 (\imu + \eps B_{21}  \\
& \qquad - C_{22})\Bigr)  z^3 \chi^2 +  \Bigl(-A_4 - 2 A_{22} B_2 + 3 B_{12} +  2 A_3 B_{21} + \imu \eps (4 A_3 + B_{13} - 4 B_{12} B_{21}) \\
	& \qquad + A_2 (-2 - 10 \imu \eps B_{21} - 2 B_{21}^2+ B_{22} + 6 C_3) + 2 B_2 C_{13}\Bigr) z^2 \chi^3 \\
& + 2 \imu \Bigl(B_2(4 A_3 + \imu \eps B_{12}) + A_2 \bigl(-5 B_3 + B_2 (\imu \eps + B_{21})\bigr)\Bigl) z \chi^4  - \Bigl(-1 + 2 A_2^2 (8 \eps - \imu B_{21}) - 2 B_{21}^2 \\
& \qquad+ B_{22} + C_3 + \imu \eps \bigl(A_{13} - B_{21} (-1 + C_3) - C_4 \bigr) + 2 A_2 (\imu A_3 + A_{22} + \eps B_{12} - C_{13})\Bigr) z^3 \chi^3 \\
& - \Bigl(-4 \imu B_3 - 2 \imu A_3 B_{12} - \imu A_2 \bigl(4 A_3 + B_{13} + B_{12} (-13 \imu \eps - 4 B_{21})\bigr) + 4 A_2^2 (3 B_2 + 2 \imu B_{21}) \\
	&\qquad + \eps \bigl(B_4 - 2 (B_{12}^2 + 2 B_3B_{21})\bigr) + B_2 \bigl(2 \imu A_{13} +  2 \imu B_{21} + \eps (-8 + B_{22}) - 2 \imu C_4\bigr) - \imu B_2^2C_{22} \Bigr)z^2 \chi^4 \\
& - \Bigl(16 A_2^3 + 2 A_2^2 B_{12} + \imu \bigl(\imu \eps A_4 + B_{13} - 3 B_{12} B_{21} - \imu \eps B_{12} C_3 + A_3 (3 + C_3)\bigr)+ A_2 \bigl(2 \imu A_{13} + 3 \imu B_{21}\\
	&\qquad  -  \imu B_{21} C_3+ \eps (2 \imu B_3 - 6 B_{21}^2+ 3 B_{22} + 8 C_3) - 2 \imu C_4 +  2 B_2 (6 + \imu \eps B_{21} - \imu C_{22})\bigr)\Bigl)  z^3 \chi^4 \\
& + \Bigl(2 A_4 B_2 + A_{22} B_2^2+ 10 A_2^2 B_{12} + 3 B_2 B_{12} + 4 \imu \eps B_3 B_{12} - \imu \eps B_2 B_{13} - 2 A_3 \bigl(B_3 +  B_2 (4 \imu \eps + B_{21})\bigr) \\
	&\qquad - A_2 \bigl(6 \eps B_2^2+ B_4 - 2 B_{12}^2 + B_3 (-16 \imu \eps - 4 B_{21}) + B_2 (-10 + 2 \imu \eps B_{21} + B_{22} + 6 C_3)\bigr)  \\
& \qquad - B_2^2 C_13\Bigr) z^2 \chi^5  - 2 \imu  B_2 \Bigl(A_3 B_2 - A_2 B_3\Bigr) z \chi^6  \\
&- \Bigl(2 A_3^2- 2 B_2 - B_4 + B_{12}^2 + 2 B_3 B_{21} + A_3\bigl(-\imu \eps B_{12} + 2 A_2 (-2 \imu B_2 + B_{21}) \bigr) + B_2 B_{22} + 6 B_2 C_3 \\
	&\qquad  + 2 A_2^2 \bigl(\imu B_3 +  B_2 (8 \eps + \imu B_{21}) - 2 B_{21}^2+ B_{22} + 4 C_3\bigr) + A_2 \bigl(-2 A_4 - 2 A_{22} B_2 - 2 B_{12} \\
	& \qquad + \eps (-2 B_2 B_{12} + 3 \imu B_{13}  - 10 \imu B_{12} B_{21}) + B_{12} C_3 + 2 B_2 C_{13}\bigr) + \imu \eps \bigl(B_3 (1 + C_3) \\
	& \qquad  + B_2 (B_{21} (-5 + C_3) -  B_2 C_{22})\bigr)\Bigr) z^3 \chi^5 +  \Bigl(12 \imu A_2^2 B_3 - 2 \eps B_3^2 + B_2 (4 \imu B_3 + \eps B_4 -  2 \imu A_3 B_{12})\\
&  \qquad - \imu A_2 \bigl(12 A_3 B_2 - 4 B_3 B_{12} + B_2 (3 \imu \eps B_{12} + B_{13})\bigr)  + B_2^2 \bigl(3 \eps (1- C_3)+ \imu A_{13}- \imu C_4\bigr)\Bigr) z^2 \chi^6\\
& + \Bigl(-\imu B_3 B_{12} -  \imu B_2 B_{13} + A_2^2 (2 B_2 B_{12} - 2\ \imu B_{13} + 7 \imu B_{12} B_{21} )+  A_3 \bigl(-2 \eps B_3 - \imu A_2 B_{12}  \\
	&\qquad + B_2 \bigl(4 \eps B_{21} +  2 \imu (-1 + C_3) \bigr) \bigr) + \imu A_2 \bigl(2 A_{13} B_2 + 2 \imu B_2^2+ B_3 - \imu \eps \bigl(3 B_4 - 4 (B_{12}^2 + 2 B_3 B_{21}) \bigr) \\
	&\qquad  - B_3 C_3  + B_2 \bigl(2 \eps B_3 -  B_{21} (-7 + C_3) +  \imu (\eps (B_{22} + 8 C_3) + 2 \imu C_4)\bigr)\bigr) + \eps B_2 \bigl(A_{22} B_2 \\
& \qquad + B_{12} (-4 + C_3) - B_2 C_{13}\bigr)\Bigr) z^3 \chi^6 +  \Bigl(-A_4 B_2^2+ 2 A_3 B_2(-2 \imu \eps B_2 + B_3) + A_2 \bigl(-2 B_3^2 \\
& \qquad + B_2 (4 \imu \eps B_3 + B_4)\bigr)\Bigr) z^2 \chi^7+  \Bigl(A_2^2 \bigl(2 \imu B_2 B_3 + 2 B_4 - 3 (B_{12}^2 + 2 B_3 B_{21}) \bigr) - A_2 \bigl(2 A_4 B_2   \\
	&\qquad +  6 B_2 B_{12} + 6 \imu \eps B_3 B_{12} +  \imu \eps B_2 B_{13} + 2 \imu A_3 B_2 (B_2 + 3 \imu B_{21}) - B_2 B_{12} C_3 \bigr) + B_2 \bigl(B_4   \\
	& \qquad + \imu \eps \bigl(3 A_3 B_{12} + B_3 (-3 + C_3) \bigr)- B_2 (3 - \imu \eps A_{13} + C_3 + \imu \eps C_4)\bigr)\Bigr) z^3 \chi^7 \\
& +  \Bigl(-\eps A_4 B_2^2+ A_2 \bigl(B_3 (2 \eps B_3 - 5 \imu A_2 B_{12}) + B_2 \bigl(\eps B_4 + \imu B_3 (-5 + C_3)\bigr)\bigr) - \imu A_3 B_2 \bigl(-2 \imu \eps B_3  \\
	&\qquad - 5 A_2 B_{12} + B_2 (-5 + C_3)\bigr)\Bigl) z^3 \chi^8  + 2 \Bigl(A_3 B_2 - A_2 B_3\Bigr)^2 z^3 \chi^9
 \end{align*}

\newpage

\section*{Appendix B: Case A and B}
\label{appendix:CaseAandB}
In the proof of \Autoref{lemma:Desingularization} the following diagrams occur: 
\begin{figure}[H]
\begin{center}
\includegraphics[scale=0.68]{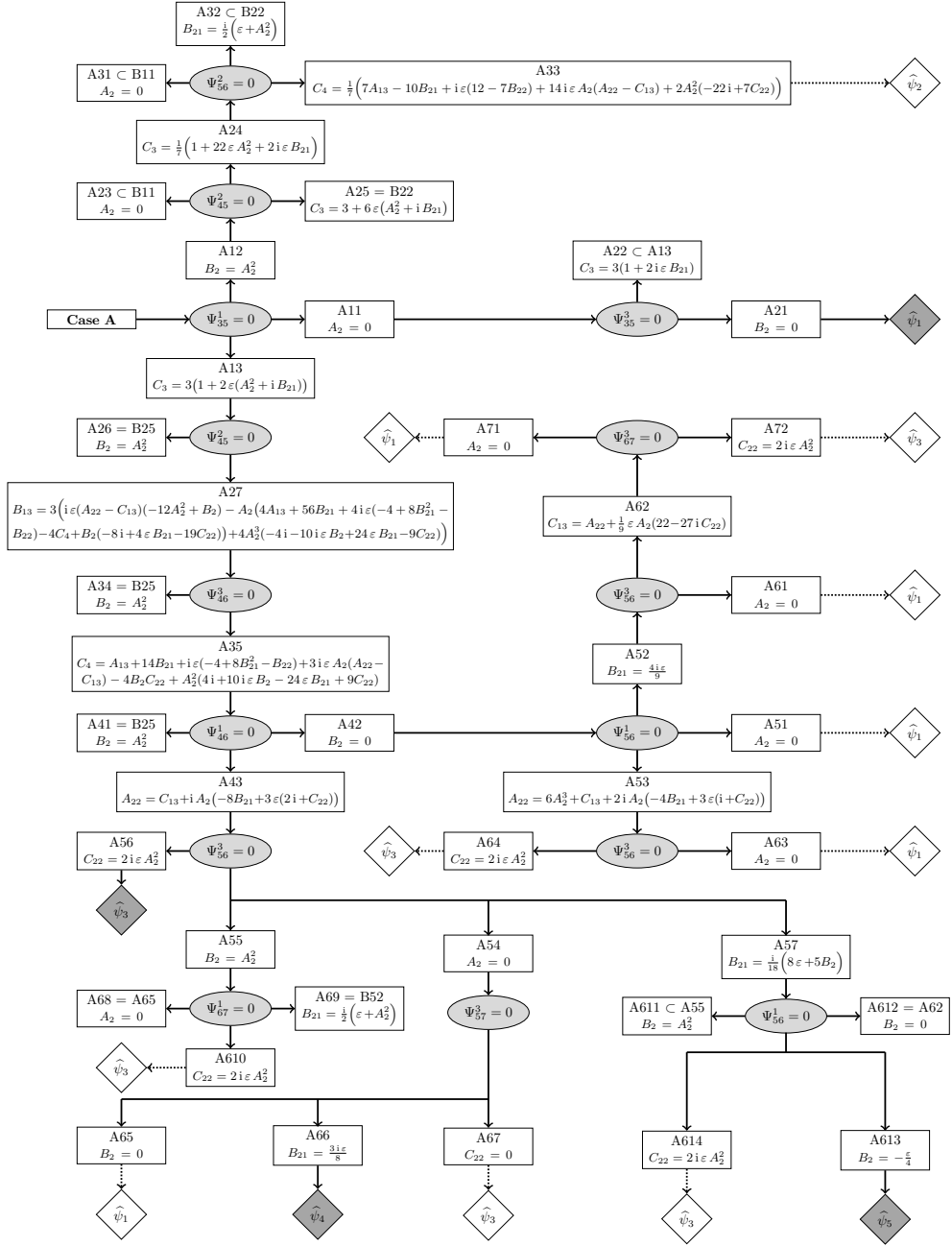}
\vspace{-2cm}
\caption{Diagram for Case A}
\label{figure:CaseA}
\end{center}
\end{figure}

\begin{figure}[H]
\vspace{-0.4cm}
\includegraphics[scale=0.68]{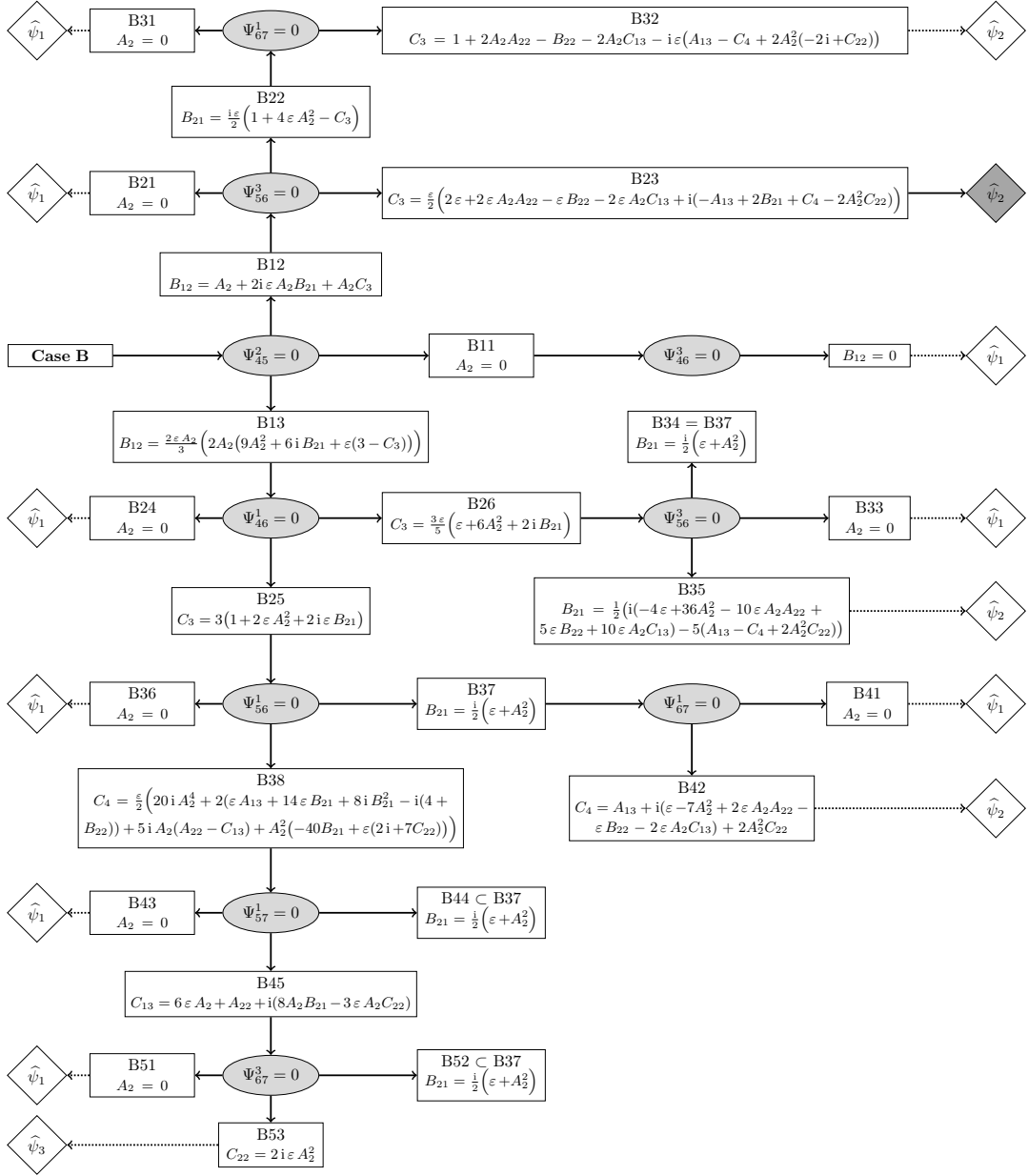}
\vspace{-3cm}
\caption{Diagram for Case B}
\label{figure:CaseB}
\end{figure}

\newpage

\section*{Appendix C: Formulas for \texorpdfstring{$\psi_k$}{PsiK} and \texorpdfstring{$\widehat \psi_k$}{HatPsiK}}
\label{appendix:FormulaPsi}

In \Autoref{lemma:Desingularization} we have the following formulas:
\begin{align*}
\psi_1(z,w)= & \Bigl(2 z \Bigl(8 + 8 B_{21} w + 4 \imu \eps C_{22} z^2 +  \imu \eps A_{22} z w +  (4 + 12 \imu \eps B_{21} - 4 B_{21}^2 - B_{22}) w^2 \Bigr),\\
& \>\>8 z^2 \Bigl( 2 -  \imu (\eps + 3 \imu B_{21}) w \Bigr),
2 w \Bigl(8 - 4 (\imu \eps - 2 B_{21}) w +  4 \imu \eps C_{22} z^2 +  \imu \eps A_{22} z w \\
& \> \> +  (2 + 6 \imu \eps B_{21} - 4 B_{21}^2 - B_{22}) w^2\Bigr) \Bigr) \slash \\
& \> \> \Bigl(16 - 8 (\imu \eps - 2 B_{21}) w +8 \imu \eps C_{22} z^2 + 2 \imu \eps A_{22} z w + 2 (2 \imu \eps B_{21} - 4 B_{21}^2 - B_{22}) w^2  \\
& \> \> - 4 C_{22}  z^2 w- A_{22}  z w^2 -  \bigl(14 B_{21} - \imu \eps (4 - 10 B_{21}^2 - B_{22})\bigr) w^3 \Bigr) \\
\psi_2(z,w)= & \Bigl(2 \Bigl( 16 z + 16 B_{21} z w + 4 A_2 w^2  +  8 \imu  \eps C_{22} z^3 + 2 \bigl(\imu \eps A_{22} + A_2 (8 \imu - 6 \eps B_{21} - 3 C_{22})\bigr) z^2 w \\
& \> \> -\bigl(A_2  A_{22} - 4 (1 + 2 \imu \eps B_{21} + B_{21}^2) + \imu A_2^2 (6 B_{21} + \eps C_{22}) \bigr) z w^2 + 2 \imu  A_2  (\eps + 2 A_2^2\\
& \> \>  + \imu B_{21}) w^3 \Bigr), 4 \Bigl(8 z^2 + 2 A_2^2 w^2 + 8 \eps A_2 z^3  - 4 (\imu \eps - 3 B_{21})  z^2 w + 2 A_2  (2 + 3 \eps A_2^2 \\
& \> \> +  4 \imu \eps B_{21}) z w^2  + \imu A_2^2 (\eps + 2 A_2^2 + \imu B_{21})w^3  \Bigr),
2 w \Bigl( 16 - 8 (\imu \eps - 2 B_{21}) w + 8 \imu \eps C_{22} z^2  + \\
& \> \> 2 \bigl(\imu \eps A_{22} + A_2 (4 \imu  - 6 \eps B_{21} - 3 C_{22})\bigr) z w -  \Bigl(A_2  A_{22} + 4 B_{21} (\imu \eps - B_{21}) + \imu A_2^2 \bigl(6 B_{21}\\
& \> \>  - \eps (4 \imu - C_{22})\bigr)\Bigr) w^2 \Bigr) \Bigr) \slash  \Bigl(32 - 16 (\imu \eps - 2 B_{21}) w + 16 \imu \eps C_{22} z^2  + 4  \bigl(\imu \eps A_{22} - 3 A_2 (2 \eps B_{21} \\
& \> \> + C_{22})\bigr) z w - 2  \Bigl(A_2  A_{22} + 4 (1 + 3 \imu \eps B_{21} - B_{21}^2) + \imu A_2^2 \bigl(6 B_{21} - \eps (12 \imu - C_{22})\bigr)\Bigr) w^2 \\
& \> \> - 8  C_{22} z^2 w - 2  \Bigl(A_{22} + A_2 \bigl(10 \imu B_{21} + \eps (8 - \imu C_{22})\bigr)\Bigr) z w^2 +\bigl(\imu \eps A_2  A_{22} -12 B_{21} \\
& \> \> + 4 \imu \eps (1 - 2 B_{21}^2)  + A_2^2 (12 \imu - 14 \eps B_{21} - C_{22})\bigr)  w^3  \Bigr) \\
\psi_3(z,w)= & \Bigr(4 z - 4 \eps A_2  z^2 + 2 \imu (\eps + \imu B_{21}) z w  + A_2 w^2,
 4 z^2 + w^2 B_2 , 2 w (2 - 2 \eps A_2 z -  B_{21} w)  \Bigr) \slash \\
& \Bigl(4 - 4 \eps A_2 z - 2 B_{21}  w - 2 \imu  A_2 z w - (1 + 2 \eps A_2^2 +   2 \imu \eps B_{21}) w^2  \Bigr)\\
\psi_4(z,w)= & \Bigl(z \Bigl(256 + 96 \imu \eps w + 128 \imu \eps C_{22} z^2 - (5 - 32 \imu \eps B_2 C_{22}) w^2 \Bigr), \\
& \> \> 4 \Bigl(64 z^2 + 16  B_2 w^2 + 4 \imu  \eps z^2 w + \imu  \eps B_2 w^3 \Bigr), w \Bigl(256 - 32 \imu \eps  w+ 128 \imu  \eps C_{22} z^2\\
&\> \>   +  (3 + 32 \imu \eps B_2 C_{22})  w^2  \Bigr)\Bigr)  \slash  \Bigl( 256 - 32 \imu \eps w + 128 \imu  \eps C_{22} z^2 - (13 - 32 \imu \eps B_2 C_{22}) w^2\\ 
& \> \>  - 64  C_{22}  z^2 w - (\imu \eps - 16 B_2 C_{22}) w^3  \Bigr)\\
\psi_5(z,w)= &\Bigl( 256 z + 96 \imu  \eps z w + 64  A_2 w^2 + 128 \imu  \eps C_{22} z^3 + 64 \imu  A_2 z^2  w - (5 - 48 \eps A_2^2   \\
& \>\> + 8 \imu C_{22})  z w^2 +  4 \imu \eps A_2 w^3 ,
 256 z^2 - 16 \eps w^2  + 256  \eps A_2 z^3 + 16 \imu  \eps z^2 w - 16  A_2 z w^2\\
& \> \> - \imu w^3, w \Bigl( 256 - 32 \imu \eps w + 128 \imu \eps C_{22} z^2  -  64 \imu  A_2 z w  +  (3 - 16 \eps A_2^2 - 8 \imu C_{22}) w^2 \Bigr) \Bigr) \slash \\ 
&\> \> \Bigl(256 - 32 \imu \eps w + 128 \imu \eps C_{22} z^2 - 192 \imu  A_2 z w  - (13 + 144 \eps A_2^2 + 8 \imu C_{22}) w^2 \\
& \> \> - 64  C_{22} z^2 w + 8  \eps A_2  (-1 + 8 \imu C_{22}) z w^2 - \eps (\imu + 4 C_{22}) w^3  \Bigr)
\end{align*}

We have $\widehat \psi_k = \psi_k$ for $k=3,4,5$.
\begin{align*}
\widehat \psi_1 (z,w)= & \Bigl(2 z \Bigl(8 \eps + 8 \eps B_{21} w + 4 \imu C_{22} z^2 - 2 \imu  (A_{22} - C_{13}) z w + (\eps - \imu A_{13} + 2 \eps B_{21}^2 - \eps B_{22} \\
& \> \>  - \eps C_3 + \imu C_4)w^2 \Bigr), 4  z^2 \Bigl(4  \eps + \imu (1 - C_3) w \Bigr),
2 w \Bigl( 8 \eps - 4  (\imu - 2 \eps B_{21}) w + 4 \imu C_{22} z^2  \\
& \> \>  -  2 \imu (A_{22} - C_{13}) z w -  \bigl(\imu A_{13} - 2 \eps B_{21}^2 - \eps (2 - B_{22} - 2 C_3) - \imu C_4\bigr) w^2  \Bigr) \Bigr) \slash\\
&\> \>  \Bigl(16 \eps - 8 (\imu - 2 \eps B_{21}) w +  8 \imu C_{22}  z^2 - 4 \imu  (A_{22} - C_{13}) z w - 2  \bigl(\imu A_{13} - 2 \imu B_{21} -2 \eps B_{21}^2 \\
& \> \> - \eps (3 - B_{22} - 3 C_3) - \imu C_4\bigr) w^2 - 4  \eps C_{22} z^2 w +  2  \eps (A_{22} - C_{13}) z w^2 + \bigl(\eps A_{13} + 2 \imu B_{21}^2 \\
& \> \> + \eps B_{21} (1 - C_3) +  \imu (1 - B_{22} - C_3 + \imu \eps C_4) \bigr)  w^3  \Bigr)\\
\widehat \psi_2 (z,w)= & \Bigl( 32 \imu z + 32 \imu B_{21} z w + 8 \imu A_2 w^2 - 16 \eps C_{22} z^3 + 8 \bigl(\eps (A_{22} - C_{13}) + A_2  (2 - 3 \imu C_{22}) \bigr) z^2 w \\
& \> \>+ 2\bigl(\eps A_{13} + 2 \eps B_{21} + 4 \imu B_{21}^2 - \imu B_{22} - \eps C_4 +  4 \imu A_2  (A_{22} - C_{13}) + 6 \eps A_2^2 C_{22}\bigr) z w^2  \\
& \> \> +  A_2\bigl(\imu A_{13} +  6 \imu B_{21} + \eps B_{22} - \imu C_4 -  2 \eps A_2 (A_{22} - C_{13}) + A_2^2 (4 + 2 \imu C_{22})\bigr) w^3, \\
& \> \> 32 \imu z^2 + 8 \imu  A_2^2 w^2 + 32 \imu  \eps A_2 z^3 - 4 \bigl(\imu (A_{13} - 2 B_{21} -  \imu \eps B_{22} - C_4) -  2 \eps A_2(A_{22} - C_{13}) \\
& \>\>+ 2 A_2^2 (6 + \imu C_{22}) \bigr) z^2 w - 4  A_2 \bigl(\imu B_{22} - \eps (A_{13} + 2 B_{21} - C_4) -  2 \imu A_2  (A_{22} - C_{13}) \\
& \>\> +  2 \eps A_2^2  (3 \imu - C_{22})\bigr) z w^2 +  A_2^2(\imu A_{13} + 6 \imu B_{21} + \eps B_{22} - \imu C_4 - 2 \eps A_2(A_{22} - C_{13}) \\
&\> \>+ A_2^2 (4 + 2 \imu C_{22})) w^3, 
4 w \Bigl(8 \imu + 4 (\eps + 2 \imu B_{21}) w  - 4 \eps C_{22} z^2  + 2 \bigl(\eps (A_{22} - C_{13}) \\
& \> \> + A_2 (4 - 3 \imu C_{22}) \bigr) z w  + \bigl(2 (\eps + \imu B_{21}) B_{21} +  \imu A_2  (A_{22} - C_{13}) + 2 \eps A_2^2 (2 \imu + C_{22}) \bigr) w^2 
 \Bigr) \Bigr) \\
&\>\> \slash \Bigl( 32 \imu + 16 (\eps + 2 \imu B_{21}) w - 16 \eps C_{22}  z^2 + 8 \bigl(\eps (A_{22} - C_{13}) + A_2  (6 - 3 \imu C_{22})\bigr) z w \\
& \> \> + 2 \Bigl(4 \imu B_{21}^2 + \imu B_{22} - \eps \bigl(A_{13} - 2 B_{21} -  C_4 - 2 A_2^2 (6 \imu + C_{22})\bigr)\Bigr)  w^2 - 8 \imu C_{22} z^2 w  \\
& \> \> + 4  \Bigl(\imu (A_{22} - C_{13}) + A_2 \bigl(2 B_{21} + \eps (2 \imu + C_{22})\bigr) \Bigr) z w^2 + \Bigl(2 \imu B_{21} + A_{13} (\imu - \eps B_{21}) \\
& \> \> + \imu B_{21} B_{22} - \imu C_4 -\eps (2 B_{21}^2 - B_{22} - B_{21} C_4) - 2 \imu A_2  B_{21} (A_{22} - C_{13}) \\
& \> \> + 2 A_2^2 \bigl(6 - \imu C_{22} + \eps B_{21} (2 \imu - C_{22})\bigr) \Bigr) w^3  \Bigr)
\end{align*}

\section*{Appendix D: Standard Parameters}
\label{appendix:StandardParameters}

Here we give the standard parameters needed in the proofs of \autoref{theorem:ReductionFinite} given in \autoref{section:renormalization} and \autoref{lemma:FixedMapping}. First we list the standard parameters for $H_1$, the renormalization of $\SimpleMapParticularValueTwo{1}{\eps}$. 
\begin{align*}
R_1 \coloneqq & \sqrt{\frac{1 + 6 \eps r_0^2 + r_0^4}{1 + 2 \eps r_0^2 + r_0^4}}, \quad c_{1 1}' \coloneqq \frac{ c_1 (1  - 4 \eps r_0^2 - r_0^4) -  2 \imu \eps r_0 \lambda_1} {\lambda_1 (1+ 6\eps r_0^2 + r_0^4)}, \quad  c_{2 1}' \coloneqq  \frac{2  r_0 (2 c_1 - \imu \eps r_0 \lambda_1)} {\lambda_1 (1+ 6\eps r_0^2 + r_0^4) },\\
\lambda_1' \coloneqq ~ & (\lambda_1 R_1)^{-1},  \quad a_{1 1}' \coloneqq   \frac{ 1  - 4 \eps r_0^2 - r_0^4 }{R_1 (\eps +  r_0^2 )^2}, \quad  a_{2 1}' \coloneqq -\frac{4  r_0} {R_1 (\eps  + r_0^2 )^2}, \quad \lambda_1 \coloneqq ~ \frac{1   + 6 \eps r_0^2 + r_0^4}{2 \sqrt{1  - 2 \eps r_0^2 + r_0^4 }},\\
 c_1 \coloneqq  ~&  \frac{\imu r_0 \lambda_1 (-1+ 4 \eps r_0^2 + r_0^4)} {(\eps - r_0^2) (1 + 6 \eps r_0^2 + r_0^4)},  \qquad u_1 \coloneqq  -1, \qquad  u'_1 \coloneqq  -1,  \qquad r_1' \coloneqq  0, \qquad  r_1 \coloneqq 0.
\end{align*}
We give the standard parameters for $\widetilde H$ for renormalizing $\SimpleMapParticularValueTwo{2}{-}$ in \autoref{lemma:FixedMapping}:
\begin{align*}
R_2  \coloneqq ~& \left(\frac{1 + \sqrt{2} r_0 (e^{-\imu \theta_0} + e^{\imu \theta_0})} {(1 + \sqrt{2} r_0 e^{-\imu \theta_0}) (1 + \sqrt{2} r_0 e^{\imu \theta_0})} \right)^{1/2},\\
S_2  \coloneqq ~& \frac{(1 + \sqrt{2} (e^{-\imu \theta_0} + e^{\imu \theta_0}) r_0 + 2 r_0^2)^2 (2 (e^{-\imu \theta_0} + e^{\imu \theta_0}) r_0 + \sqrt{2} (1 + 2 r_0^2))^2} {(1 + \sqrt{2} r_0 e^{-\imu \theta_0})^4 (1 + \sqrt{2} e^{\imu \theta_0} r_0)^4 (1 + \sqrt{2} (e^{-\imu \theta_0} + e^{\imu \theta_0}) r_0)^2}
\\ c'_{1 2} \coloneqq ~& \bigl((e^{\imu \theta_0} +  \sqrt{2} r_0) (-c_2 u_2 (1 + 3 r_0^2 + 2 e^{2 \imu \theta_0} r_0^2 +  2 \sqrt{2} e^{\imu \theta_0} r_0 (1 + r_0^2) - \imu v_0)  \\
& \quad + \imu e^{\imu \theta_0} r_0 (1 + \sqrt{2} e^{\imu \theta_0} r_0) \lambda_2)\bigr) \slash \bigl((1 + \sqrt{2} e^{\imu \theta_0} r_0) (e^{\imu \theta_0} + \sqrt{2} r_0 + \sqrt{2} e^{2 \imu \theta_0} r_0) \lambda_2\bigr) \\
c'_{2 2} \coloneqq ~& \bigl((e^{\imu \theta_0} + \sqrt{2} r_0) (c_2 u_2 (-r_0 (3 r_0 + 2 e^{2 \imu \theta_0} r_0 +  2 \sqrt{2} e^{\imu \theta_0} (1 + r_0^2)) + \imu v_0) \\
& \quad +  \imu e^{\imu \theta_0} r_0 (1 + \sqrt{2} e^{\imu \theta_0} r_0) \lambda_2)\bigr) \slash \bigl((1 + \sqrt{2} e^{\imu \theta_0} r_0) (e^{\imu \theta_0} + \sqrt{2} r_0 + \sqrt{2} e^{2 \imu \theta_0} r_0) \lambda_2\bigr)\\
\lambda'_2 \coloneqq ~&  (\lambda_2 R_2)^{-1}\\
a'_{1 2} \coloneqq ~&  \frac{1 + 3 r_0^2 + 2 e^{-2 \imu \theta_0} r_0^2 +  2 \sqrt{2} e^{-\imu \theta_0} r_0 (1 + r_0^2) + \imu v_0} {u_2 u_2' R_ 2 (1 + \sqrt{2} e^{-\imu \theta_0} r_0)^2}\\
a'_{2 2} \coloneqq ~& -\frac{3 r_0^2 + 2 e^{-2 \imu \theta_0} r_0^2 +  2 \sqrt{2} e^{-\imu \theta_0} r_0 (1 + r_0^2) +  \imu v_0} {u_2 u_2' R_2 (1 + \sqrt{2} e^{-\imu \theta_0} r_0)^2} \\
u'_2 \coloneqq ~&\frac{e^{\imu \theta_0} (\sqrt{2} r_0 + \sqrt{2} e^{-2 \imu \theta_0} r_0 +  e^{-\imu \theta_0} (1 + 2 r_0^2)) (2 r_0 + 2 e^{-2 \imu \theta_0} r_0 +  \sqrt{2} e^{-\imu \theta_0} (1 + 2 r_0^2))} {(1 + \sqrt{2} e^{-\imu \theta_0} r_0)^4 (e^{-\imu \theta_0} + \sqrt{2} r_0 + \sqrt{2} e^{-2 \imu \theta_0} r_0) S_2 u_2^3} \\
u_2 \coloneqq ~& \frac{2  S_2 (1 + \sqrt{2} r_0 e^{-\imu \theta_0})^4 (1 + \sqrt{2} e^{\imu \theta_0} r_0)^4 (1 + \sqrt{2} r_0e^{-\imu \theta_0} +  \sqrt{2} e^{\imu \theta_0} r_0)} {(1 + \sqrt{2} r_0 e^{-\imu \theta_0} + \sqrt{2} e^{\imu \theta_0} r_0 + 2 r_0^2) (\sqrt{2} + 2 r_0 e^{-\imu \theta_0} + 2 e^{\imu \theta_0} r_0 + 2 r_0^2)^3}\\
\lambda_2 \coloneqq ~& \frac{\sqrt{2} S_2 (1 + \sqrt{2} r_0 e^{-\imu \theta_0})^4 (1 + \sqrt{2} e^{\imu \theta_0} r_0)^4 (1 + \sqrt{2} r_0e^{-\imu \theta_0} +  \sqrt{2} e^{\imu \theta_0} r_0)^2} {(1 + \sqrt{2} r_0 e^{-\imu \theta_0} + \sqrt{2} e^{\imu \theta_0} r_0 + 2 r_0^2)^2 (\sqrt{2} + 2 r_0 e^{-\imu \theta_0} + 2 e^{\imu \theta_0} r_0 + 2 r_0^2)^2}
\end{align*}
The remaining parameters $c_2,r_2$ and $r_2'$ are set to $0$.

We give the standard parameters for the map $H_3$ for the renormalization of $\SimpleMapParticularValueTwo{3}{-}$:
\begin{align*}
R_3 \coloneqq ~& \sqrt{ \frac{-1 + r_0^4}{r_0^4}}, \qquad c'_{1 3} \coloneqq ~ \frac{c_3 r_0^4} {\lambda_3 (1 - r_0^4)}, \qquad  c'_{2 3} \coloneqq  \frac{r_0(\imu c_3 r_0 + \lambda_3)}{\lambda_3 (1-r_0^4)}, \qquad \lambda_3'  \coloneqq ~ \Bigl(\lambda_3 R_3\Bigr)^{-1}, \\
a'_{1 3} \coloneqq ~& -\imu/R_3, \qquad a'_{2 3}\coloneqq  1/(r_0^2 R_3), \qquad c_3  \coloneqq ~ \frac{\imu(-1 + 3 r_0^4)}{8 r_0^2}, \qquad \lambda_3 \coloneqq  \frac{-1 + r_0^4}{2 r_0},\qquad  u'_3 \coloneqq ~\imu,
\end{align*}
and the remaining parameters $u_3, r_3'$ and $r_3$ to be trivial.

If we consider $H_4$ we renormalize the mapping $\SimpleMapParticularValueTwo{4}{\eps}=(f_{1 p_0},f_{2 p_0},g_{p_0})$. Here we use the following standard parameters, which only cover the case when $g_{p_0 w}(0)>0$. If $g_{p_0 w}(0)<0$ we need to interchange some of the standard parameters given here as described in the proof of \autoref{theorem:ReductionFinite}: 
\begin{align*}
R_4 \coloneqq ~& \sqrt{3}\sqrt{\frac{\eps + 14 r_0^4 + \eps r_0^8}{1+ 3 \eps r_0^4}}, \qquad c'_{1 4} \coloneqq ~  \frac{4 c_4 r_0^2 u (-1 + r_0^4 \eps) - 8 \imu r_0^5 \eps \lambda_4} {(14 r_0^4 + \eps +  r_0^8 \eps) \lambda_4}, \\
c'_{2 4} \coloneqq ~& \frac{c_4 u_4 (-1 + 3 r_0^8 + 14 r_0^4 \eps) - 8 \imu r_0^3 \eps \lambda_4} {\sqrt{3} (14 r_0^4 + \eps +  r_0^8 \eps) \lambda_4}, \qquad
\lambda_4'  \coloneqq  \Bigl(\lambda_4 R_4\Bigr)^{-1}  \\
a'_{1 4} \coloneqq ~& \frac{-12  r_0^2 (-1 + r_0^4 \eps)} {u_4 u_4' R_4 (1 + 3 r_0^4 \eps)^2}, \qquad \qquad \qquad  \quad a'_{2 4}\coloneqq -\sqrt{3} \frac{1 - 3 r_0^8 - 14 r_0^4 \eps} {u_4 u_4' R_4 (1 + 3 r_0^4 \eps)^2} \\
c_4 \coloneqq ~& \frac{\imu r_0^3 (-7 - 26 r_0^8 + 9 r_0^{16} - 36 r_0^4 \eps +  60 r_0^{12} \eps) \lambda_4} {u_4 (-19 r_0^4 - 38 r_0^{12} +  9 r_0^{20} - (1 + 74 r_0^8 - 123 r_0^{16}) \eps)}
\end{align*}
\begin{align*}
\lambda_4 \coloneqq ~& \left( 4 \sqrt{3} r_0 \left|\frac{\eps - r_0^4}{1+ 14 \eps r_0^4 + r_0^8}\right|\right)^{-1}, \qquad \quad u'_4 \coloneqq \frac{\sgn(r_0^4 - \eps)} {u_4^3 \sgn(1 + r_0^8 + 14 r_0^4 \eps)} \\
u_4 \coloneqq ~&  \left(\frac{1 - \eps}2\right) \left( \frac{\sgn(-1 - 33 r_0^4 + 33 r_0^8 + r_0^{12})}{ \sgn(1 - 14 r_0^4 + r_0^8)}\right) - \left(\frac{1 + \eps} 2\right) \sgn(-1 + 34 r_0^4 -  34 r_0^{12} + r_0^{16})
\end{align*}
The remaining parameters $r_4$ and $r_4'$ are taken to be $0$.

\address{Texas A\&M University at Qatar, PO Box 23874, Doha, Qatar}\\
\email{michael.reiter@qatar.tamu.edu}

\end{document}